\documentclass[
10pt,
]{article}

\usepackage[ngerman, english]{babel}
\usepackage[utf8]{inputenc}
\usepackage{enumerate}
\usepackage[center]{caption}
\usepackage{cite}
\usepackage{tocbasic}

\DeclareTOCStyleEntry[
beforeskip=.2em plus 1pt,
]{tocline}{section}

\usepackage[%
left=1.25in,
right=1.25in,
bottom=1.5in,
top=1in,
footskip=0.5in,
]{geometry}

\usepackage{color}
\definecolor{LinkColor}{rgb}{0,0,1}
\definecolor{LinkColor2}{rgb}{0,0.5,0}
\definecolor{lbcolor}{rgb}{0.85,0.85,0.85}
\definecolor{FrameColor}{rgb}{0.85,0.85,0.85}

\usepackage{enumitem}

\usepackage{graphicx}
\usepackage{placeins}

\usepackage{amsmath}
\usepackage{amssymb}
\usepackage{dsfont}
\allowdisplaybreaks
\usepackage{amsthm}
\usepackage{mathtools, nccmath}

\usepackage[%
pdftitle={Titel},%
pdfauthor={Autor},%
pdfcreator={LaTeX, LaTeX with hyperref and KOMA-Script},
pdfsubject={Betreff}, 
pdfkeywords={Keywords}
]{hyperref} 

\hypersetup{%
	colorlinks	=true,
	linkcolor	=LinkColor,%
	anchorcolor	=LinkColor,%
	citecolor	=LinkColor2,%
	filecolor	=LinkColor,%
	menucolor	=LinkColor,%
	urlcolor	=LinkColor,%
}

\usepackage{accents}

\numberwithin{equation}{section}

\theoremstyle{plain}
\newtheorem{thm}{Theorem}[section]

\newtheorem{lem}[thm]{Lemma}

\newtheorem{defn}[thm]{Definition}

\theoremstyle{definition}
\newtheorem{rem}[thm]{Remark}

\makeatletter
\renewenvironment{proof}[1][\proofname]{%
	\par\pushQED{\qed}\normalfont%
	\topsep6\p@\@plus6\p@\relax
	\trivlist\item[\hskip\labelsep\bfseries#1\@addpunct{.}]%
	\ignorespaces
}{%
	\popQED\endtrivlist\@endpefalse
}
\makeatother

\makeatletter
\renewcommand\paragraph{\@startsection{paragraph}{4}{\z@}%
	{1ex \@plus1ex \@minus.2ex}%
	{-1em}%
	{\normalfont\normalsize\bfseries}}
\renewcommand\subparagraph{\@startsection{paragraph}{4}{\z@}%
	{1ex \@plus1ex \@minus.2ex}%
	{-1em}%
	{\normalfont\normalsize\itshape}}
\makeatother

\usepackage{setspace}

\newcommand{\N}{\mathbb{N}}
\newcommand{\R}{\mathbb{R}}

\newcommand{\LL}{\mathcal{L}}
\newcommand{\NN}{\mathcal{N}}

\renewcommand{\SS}{\mathcal{S}^L}
\newcommand{\SD}{\mathcal{S}^0}

\newcommand{\VV}{\mathcal{V}}

\newcommand{\HH}{\mathcal{H}}

\newcommand{\HHBM}{\HH_{\beta,m}^k}
\newcommand{\HHBI}{\HH_{\beta,0}^1}

\newcommand{\HHBo}{\HH_{\beta,0}^0}
\newcommand{\HHBMo}{\HH_{\beta,m}^0}

\newcommand{\DD}{\mathcal{D}}
\newcommand{\DDBI}{\DD_{\beta}^1}

\newcommand{\LINF}{\mathcal{L}^\infty}
\newcommand{\IN}{\mathcal{I}_n}
\newcommand{\INF}{\mathcal{I}_n^*}
\newcommand{\LIP}{\mathrm{L}_B}

\newcommand{\eps}{\varepsilon}

\newcommand{\abs}[1]{\ensuremath \left| #1 \right|}
\newcommand{\mean}[1]{\ensuremath \left< #1 \right>}

\newcommand{\meano}[1]{\ensuremath \left< #1 \right>_\Omega}
\newcommand{\meang}[1]{\ensuremath \left< #1 \right>_\Gamma}

\def\Joinref #1#2 #3 {\gdef #1{\eqref{#2}--\eqref{#3}}%
  \wlog{}\wlog{\string #1 -> #2 - #3}\wlog{}}

\newcommand{\norm}[1]{\| #1 \|}
\newcommand{\inn}[2]{ \langle #1 , #2  \rangle}
\newcommand{\biginn}[2]{ \big< #1 , #2  \big>}
\newcommand{\scp}[2]{ \left( #1 , #2  \right)}
\newcommand{\bigscp}[2]{\big( #1 , #2 \big)}

\newcommand{\rO}[1] {\langle #1 \rangle_\Omega}
\newcommand{\rG}[1] {\langle #1 \rangle_\Gamma}
\newcommand{\bet}{\boldsymbol \eta}
\newcommand{\bph}{\boldsymbol \varphi}
\newcommand{\brh}{\boldsymbol \rho}
\newcommand{\bphi}{\ov \phi}
\newcommand{\bpsi}{\ov \psi}

\newcommand{\bphn}{\bph^n}
\newcommand{\bphnp}{\bph^{n+1}}
\newcommand{\phl}{\phi^L}
\newcommand{\psl}{\psi^L}
\newcommand{\ml}{\mu^L}
\newcommand{\nl}{\nu^L}

\def\genspazio #1#2#3#4#5{#1^{#2}(#5,#4;#3)}
\def\spazio #1#2#3{\genspazio {#1}{#2}{#3}T0}

\def\L {\spazio L}
\def\H {\spazio H}

\def\Lx #1{L^{#1}(\Omega)}
\def\Hx #1{H^{#1}(\Omega)}

\def\LxG #1{L^{#1}(\Gamma)}
\def\HxG #1{H^{#1}(\Gamma)}

\newcommand{\ov}{\overline}
\newcommand{\bz}{\ov \zeta}
\newcommand{\bxi}{\ov \xi}

\newcommand{\n}{\mathbf{n}}
\newcommand{\grad}{\nabla}
\newcommand{\gradg}{\nabla_\Gamma}

\newcommand{\Lap}{\Delta}
\newcommand{\Lapg}{\Delta_\Gamma}

\newcommand{\dd}{\mathrm{d}}
\newcommand{\ddt}{\frac{\dd}{\dd t}}
\newcommand{\dtau}{\, \dd\tau}
\newcommand{\dy}{\, \dd y}

\newcommand{\ds}{\, \dd s}
\newcommand{\dx}{\, \dd x}
\newcommand{\dt}{\, \dd t}
\newcommand{\dr}{\,\dd r}
\newcommand{\dS}{\,\dd S}

\newcommand{\dSy}{\,\dd S(y)}

\newcommand{\pp}{{p'}}
\newcommand{\qp}{{q'}}

\newcommand{\del}{\partial}
\newcommand{\delt}{\partial_t}
\newcommand{\deln}{\partial_\n}
\newcommand{\deltau}{\partial_\tau}
\newcommand{\delxi}{\partial_{x_i}}
\newcommand{\delxj}{\partial_{x_j}}

\newcommand{\Om}{{\Omega}}
\newcommand{\Ga}{{\Gamma}}
\newcommand{\OT}{{\Omega_T}}
\newcommand{\GT}{{\Gamma_T}}

\newcommand{\ao}{{a_\Omega}}
\newcommand{\ag}{{a_\Gamma}}

\newcommand{\mo}{m_\Omega}
\newcommand{\mg}{m_\Gamma}

\newcommand{\Fw}{F_{\rm w}}
\newcommand{\Gw}{G_{\rm w}}

\newcommand{\intO}{\int_\Omega}
\newcommand{\intG}{\int_\Gamma}
\newcommand{\intOT}{\int_\OT}
\newcommand{\intGT}{\int_\GT}
\newcommand{\intOto}{\int_{\Om_{t_0}}}
\newcommand{\intGto}{\int_{\Ga_{t_0}}}

\newcommand{\cg}{\circledast}

\newcommand{\wto}{\rightharpoonup}

\newcommand{\emb}{\hookrightarrow}

\begin{document}

\title{
	\bfseries
	On the nonlocal Cahn--Hilliard equation \\ 
	with nonlocal dynamic boundary condition \\
	and boundary penalization
}

\author{Patrik Knopf\footnotemark[1] 
\and Andrea Signori\footnotemark[2]}
\date{May 22, 2020}

\renewcommand{\thefootnote}{\fnsymbol{footnote}}
\footnotetext[1]{Fakult\"at f\"ur Mathematik, Universit\"at Regensburg, Universit\"atsstra\ss e 31, 93053 Regensburg, Germany (\href{mailto:patrik.knopf@ur.de}{patrik.knopf@ur.de}),
ORCID: \href{https://orcid.org/0000-0003-4115-4885}{0000-0003-4115-4885}}
\footnotetext[2]{Dipartimento di Matematica e Applicazioni, Universit\`a di Milano--Bicocca, via Cozzi 55, 20125 Milano, Italy
(\href{mailto:andrea.signori02@universitadipavia.it}{andrea.signori02@universitadipavia.it}),
ORCID: \href{https://orcid.org/0000-0001-7025-977X}{0000-0001-7025-977X}}

\maketitle

\begin{center}
\textit{This is a preprint version of the paper. Please cite as:} \\[1ex] 
P. Knopf and A. Signori, \textit{J. Differential Equations}, 280(4):236-291, 2021. \\
\url{https://doi.org/10.1016/j.jde.2021.01.012}
\end{center}

\medskip

\begin{abstract}
The Cahn--Hilliard equation is one of the most common models to describe phase segregation processes in binary mixtures. Various dynamic boundary conditions have already been introduced in the literature to model interactions of the materials with the boundary more precisely. To take long-range interactions into account, we propose a new model consisting of a nonlocal Cahn--Hilliard equation with a nonlocal dynamic boundary condition comprising an additional boundary penalization term. We rigorously derive our model as the gradient flow of a nonlocal free energy with respect to a suitable inner product of order $H^{-1}$ containing both bulk and surface contributions. In the main model, the chemical potentials are coupled by a Robin type boundary condition depending on a specific relaxation parameter. We prove weak and strong well-posedness of this system, and we investigate the singular limits attained when this relaxation parameter tends to zero or infinity.
\end{abstract}

\medskip

\begin{small}
\\[1ex]

\end{small}

\setlength\parindent{0ex}
\setlength\parskip{1ex}
\newpage

\section{Introduction}

We consider the system of equations
\begin{subequations}
	\label{CH:ROB}
	\begin{alignat}{2}
	&\label{CH:ROB:PHI}
	\delt \phi =  m_\Om \Lap \mu 
	&&\quad\text{in}\; \OT:=\Om\times(0,T),\\
	\label{CH:ROB:MU}
	&\mu = \eps\ao\phi - \eps J * \phi + \tfrac 1 \eps \Fw'(\cdot,\phi) 
	&&\quad\text{in}\; \OT,\\[1ex]
	\label{CH:ROB:PSI}
	&\delt \psi =  m_\Ga \Lapg \nu - \beta m_\Om \deln \mu 
	&&\quad\text{on}\; \GT:=\Ga\times(0,T),\\
	\label{CH:ROB:NU}
	&\nu = \delta \ag\psi - \delta K \cg \psi + \tfrac 1 \delta \Gw'(\cdot,\psi) + \tfrac 1 \delta B'(\cdot,\psi) 
	&&\quad\text{on}\; \GT,\\[1ex]
	\label{CH:ROB:BC}
	&L \deln \mu = \beta\nu - \mu 
	&&\quad\text{on}\; \GT,\\[1ex]
	\label{CH:ROB:PHI:INI}
	&\phi\vert_{t=0} = \phi_0 
	&&\quad\text{in}\; \Om, \\
	\label{CH:ROB:PSI:INI}
	&\psi\vert_{t=0} = \psi_0 
	&&\quad\text{on}\; \Ga
	\end{alignat}
\end{subequations}
where $\eps,\delta>0$ and $\Om\subset \R^d$ with $d=2$ or $d=3$ is a bounded domain with boundary $\Ga=\del\Om$ whose unit outer normal vector field is denoted by $\n$. It consists of a nonlocal Cahn--Hilliard equation in the bulk \eqref{CH:ROB:PHI}--\eqref{CH:ROB:MU} subject to a dynamic boundary condition \eqref{CH:ROB:PSI}--\eqref{CH:ROB:NU} that also has a nonlocal Cahn--Hilliard type structure. The functions $\phi$ and $\psi$ stand for phase-field variables describing the difference of two local relative concentrations of materials in the bulk and on the surface, respectively. 
In \eqref{CH:ROB:PHI} and \eqref{CH:ROB:PSI}, the mobilities $\mo$ and $\mg$ are assumed to be positive constants. 
Moreover, $\mu$ denotes the chemical potential in the bulk whereas $\nu$ denotes the chemical potential on the surface. 
The symbols ``$*$'' in \eqref{CH:ROB:MU} and ``$\cg$'' in \eqref{CH:ROB:NU} stand for the convolutions on $\Omega$ and $\Gamma$, respectively, i.e.,
\begin{alignat*}{3}
(J*\phi)(x,t)& :=\intO J(x-y) \phi(y,t)\dy 
\quad 
&&\hbox{for every $(x,t)\in \Om_T$},
\\
(K\cg \psi)(z,t)& :=\intG K(z-y) \psi(y,t)\dSy 
\quad 
&&\hbox{for every $(z,t)\in \Ga_T$}.
\end{alignat*}
Moreover, the functions $\ao$ and $\ag$ are defined by
\begin{alignat*}{3}
a_\Om(x) &:= \big(J * 1\big)(x) 
\quad\text{and}\quad
a_\Ga(z) &:= \big(K \cg 1\big)(z) 
\end{alignat*} 
for almost all $x\in\Omega$ and $z\in\Gamma$. 
In this model, the chemical potentials are coupled by the Robin type boundary condition \eqref{CH:ROB:BC} with parameters $L> 0$ and $\beta\neq 0$. To this end, we will refer to it as the ``Robin model''. In addition, we also investigate the singular limits $L\to 0$ and $L\to\infty$ of the system \eqref{CH:ROB}
which lead to a Dirichlet type boundary condition ($\mu \vert_\GT=\beta \nu$ $a.e.$ on $\GT$) and a homogeneous Neumann boundary condition ($\partial_{\bold n}\mu=0$ $a.e.$ on $\GT$), respectively. 

In the following subsection, we will explain the motivation behind our model as well as the occurring quantities. 
Before discussing the system \eqref{CH:ROB} and its singular limits in more detail, we first present a short review of previous results to provide a better understanding of the origins and advances of our model.

\subsection{Motivation of our model}

\paragraph{The (local) Cahn--Hilliard equation.} 
The Cahn--Hilliard equation was originally introduced in \cite{cahn-hilliard} to model phase separation and de-mixing processes in binary alloys. Meanwhile, it is frequently used in mathematical models describing phenomena in materials science, life sciences and also image processing. 

The (local) Cahn--Hilliard equation as introduced in \cite{cahn-hilliard} reads as follows:
\begin{subequations}
	\label{CH}
	\begin{alignat}{3}
	\label{CH:1}
	&\delt \phi = m_\Om \Lap \mu, 
	&& \quad \mu = -\eps \Lap \phi + \tfrac 1\eps \Fw'(\phi) 
	&&\quad\text{in}\; \OT,\\
	\label{CH:2}
	&\phi\vert_{t=0}=\phi_0 
	&&&&\quad\text{in } \Om.
	\end{alignat}
\end{subequations}
Here, the functions $\phi=\phi(x,t)$ and $\mu=\mu(x,t)$ depend on position $x\in\Omega$ and time $t\in [0,T]$, where $T>0$ denotes an arbitrary but fixed final time.
To describe a mixture of two materials, the phase-field variable $\phi$ stands for the difference of two local relative concentrations.
We suppose that the parameter $m_\Om$, which represents the so-called mobility, is a positive constant. This is indeed a very typical assumption, although non-constant mobilities are used in some situations (see, e.g., \cite{elliotgarcke}).
After a short period, the phase-field $\phi$ will attain values
close to $\pm 1$ in large regions of the domain $\Omega$. These areas, which correspond to the pure phases of the materials, are 
separated by a thin interface whose thickness is proportional to the parameter $\eps>0$ (which is usually chosen very small).
The function $\mu$ denotes the chemical potential in the bulk (i.e., in $\Omega$). It describes chemical reactions influencing the time evolution of the phase-field
and can be expressed as the Fr\'echet derivative of the following bulk free energy of Ginzburg--Landau type:
\begin{align}
E^\text{GL}_\text{bulk}(\phi) = \int_\Omega \frac \eps 2|\grad \phi|^2 + \frac 1 \eps \Fw(\phi) \dx,
\end{align}
where the function $\Fw$ is the bulk potential. If phase separation processes are to be described, $\Fw$ is usually supposed to exhibit a double-well structure attaining its
global minima at $-1$ and $1$ (corresponding to the pure phases) and a local maximum at $0$. 
A physically relevant choice for $\Fw$ is a singular logarithmic potential (cf. \eqref{POT:LOG}). For simplicity of the mathematical analysis, it is
often approximated by a regular polynomial potential, typically the double-well potential $\Fw(s) = W_\text{dw}(s) = \frac{1}{4}(s^2-1)^2$ (see also Remark~\ref{EX:POT}). Since the time evolution of the phase-field and the chemical potential is considered in a bounded domain, it is necessary to impose suitable boundary conditions. The homogeneous Neumann conditions
\begin{align}
	\label{HNC}
	\del_\n \phi = 0, \quad
	\del_\n \mu  = 0  \quad\text{on}\; \GT
\end{align}
are the classical choice. 
The condition $\eqref{HNC}_1$ implies that the interface intersects the boundary at a perfect contact angle of ninety degrees, while the no-flux condition $\eqref{HNC}_2$ entails that the phase-field $\phi$ satisfies the mass conservation law
\begin{align}
\int_\Omega \phi(t) \dx = \int_\Omega \phi(0) \dx, \quad t\in[0,T].
\end{align}
Moreover both conditions in \eqref{HNC} imply the energy dissipation law
\begin{align}
\ddt E^\text{GL}_\text{bulk}\big(\phi(t)\big) + m_\Om \int_\Omega |\grad\mu(t)|^2 \dx = 0, \quad t\in[0,T].
\end{align}
We point out that the Cahn--Hilliard equation subject to the boundary conditions \eqref{HNC} can be interpreted as a gradient flow of type $H^{-1}$ of the bulk free energy $E_\text{bulk}^\text{GL}$ (cf. \cite{Cowan}). The Cahn--Hilliard equation \eqref{CH} with homogeneous Neumann conditions \eqref{HNC} is already very well understood and there exists an extensive literature (see, e.g., 
\cite{Abels-Wilke,Bates-Fife,Cherfils,elliotgarcke,elliotzheng,zheng,rybka,pego}).

\paragraph{The (local) Cahn--Hilliard equation with dynamic boundary conditions.}
In some situations, it turned out that homogeneous Neumann boundary conditions are not satisfactory as they neglect the influence of certain processes on the boundary to the dynamics in the bulk. For instance, separate chemical reactions on the boundary cannot be taken into account. 
However, especially in certain applications (e.g., applications in hydrodynamics and contact line problems), it proved necessary to describe short-range interactions of the binary mixture with the solid wall of the container more precisely. 
To this end, physicists proposed a surface free energy which is again of Ginzburg--Landau type (cf. \cite{Fis1,Fis2,Kenzler}):
\begin{equation}\label{DEF:EN:GL}
E^\text{GL}(\phi,\psi) := E^\text{GL}_\text{bulk}(\phi) + E^\text{GL}_\text{surf}(\psi)
\quad\text{with}\quad
E^\text{GL}_\text{surf}(\psi)  = \intG \frac{\kappa \delta }{2} \abs{\gradg \psi}^2 + \frac{1}{\delta} \Gw(\psi) \dS.
\end{equation}
Here, the symbol $\gradg$ denotes the surface gradient on $\Gamma$, $G$
is a surface potential, the parameter $\kappa\ge 0$ acts as a weight for surface diffusion effects, 
and $\delta>0$ denotes a small parameter corresponding to the thickness of the interface on the surface. In the case $\kappa = 0$ this energy
is related to the moving contact line problem (see, e.g., \cite{thompson-robbins}). Recently, various dynamic boundary conditions corresponding to the energy $E^\text{GL}$
have been derived and analyzed in the literature, for instance
\cite{colli-fukao-ch,colli-gilardi,Gal1,GalWu,colli-gilardi-sprekels,liero,mininni,miranville-zelik,motoda,racke-zheng,WZ, FukaoWu, Wu,GalGra, CGG, CGW, CFW, MW}.
In particular, Cahn--Hilliard systems with dynamic boundary conditions exhibiting also a Cahn--Hilliard type structure have become very popular in recent times. In general, such models can be interpreted as a gradient flow of the energy $E^\text{GL}$ with respect to a suitable inner product of order $H^{-1}$ which contains both a bulk and a surface contribution (see, e.g., \cite{GK,KL,KLLM}).

The following model which was proposed and analyzed in \cite{KLLM} can be regarded as the local analogue of the model \eqref{CH:ROB} (with $B\equiv 0$) we intend to study. It reads as follows:
\begin{subequations}\label{CH:INT}
	\begin{alignat}{3}
	& \delt\phi = m_\Om \Lap \mu, 
	&&\quad  \mu = - \eps \Lap \phi + \tfrac 1\eps \Fw'(\phi) 
	&&\quad \text{in } \OT, \label{CH:INT:1}\\
	& \delt\psi = m_\Gamma\Lapg  \nu - \beta m_\Om\deln  \mu, 
	&&\quad  \nu = - \delta\kappa \Lapg \psi + \tfrac 1\delta \Gw'(\psi) + \eps\deln \phi 
	&&\quad \text{on } \GT, \label{CH:INT:2}\\
	& \phi\vert_\GT = \psi,
	&&\quad L \deln  \mu = \beta  \nu -  \mu 
	&&\quad \text{on } \GT, \label{CH:INT:3} \\
	& \phi\vert_{t=0} = \phi_0 
	&&&&\quad \text{in } \Om, \label{CH:INT:4} \\
	& \psi\vert_{t=0} = \psi_0 = \phi_0\vert_\Ga 
	&&&&\quad \text{on } \Ga, \label{CH:INT:5}
	\end{alignat}
\end{subequations}
where $\beta\neq 0$, $L>0$.
Here, the chemical potentials $\mu$ and $\nu$ are coupled by a Robin type boundary condition $\eqref{CH:INT:3}_2$ (which is the same as the condition \eqref{CH:ROB:BC} in our system) to model chemical reactions between the materials in the bulk and the materials on the surface. 
In this context, the constant $1/L$ is related to the reaction rate. 
Here, the term ``reactions'' is to be understood in
a general sense including chemical reactions but also adsorption or desorption processes.
By the condition $\eqref{CH:INT:3}_2$, the mass flux $-m_\Om \deln \mu$ (which describes the motion of the materials towards and away from the boundary) is directly influenced by differences in the chemical potentials.

Provided that a solution of \eqref{CH:INT} is sufficiently regular, it satisfies the mass conservation law
\begin{align}
\label{INT:MASS}
	\beta\intO \phi(t) \dx + \intG \psi(t) \dS = \beta\intO \phi_0 \dx + \intG \psi_0 \dS,
	\quad t\in[0,T],
\end{align}
as well as the energy dissipation law
\begin{align}
	\label{INT:NRG}
\begin{aligned}
	& \ddt E^\text{GL}\big(\phi(t),\psi(t)\big) \\
	&\quad + m_\Om\intO |\grad\mu(t)|^2 \dx + m_\Ga\intG |\gradg \nu(t)|^2 \dS 
	+ \frac{m_\Om}{L} \intG \abs{ \beta\nu-\mu } ^2 \dS = 0,
\end{aligned}
\end{align}
for all $t\in [0,T]$. This means that the parameter $\beta$ acts as a weight in the mass conservation relation \eqref{INT:MASS}. 
Moreover, since the total free energy $E^\text{GL}$ is bounded from below (at least for reasonable choices of $\Fw$ and $\Gw$), the dissipation law \eqref{INT:NRG} implies that the potentials $\mu$ and $\nu$ converge to the chemical equilibrium $\mu\vert_\GT=\beta \nu$ over time. For more details, we refer the reader to \cite{KLLM}.

Furthermore, the singular limits $L\to 0$ and $L\to\infty$ were analyzed rigorously in \cite{KLLM}. The limit $L\to 0$ leads to the system
\pagebreak[2] 
\begin{subequations}\label{CH:GMS}
	\begin{alignat}{3}
	& \delt\phi = m_\Om \Lap \mu, 
	&&\quad  \mu = - \eps \Lap \phi + \tfrac 1\eps \Fw'(\phi) 
	&&\quad \text{in } \OT, \label{CH:GMS:1}\\
	& \delt\psi = m_\Gamma\Lapg  \nu - \beta m_\Om\deln  \mu, 
	&&\quad  \nu = - \delta \kappa \Lapg \psi + \tfrac 1\delta \Gw'(\psi) + \eps\deln \phi 
	&&\quad \text{on } \GT, \label{CH:GMS:2}\\
	& \phi\vert_\GT = \psi, 
	&&\quad \mu\vert_\GT = \beta  \nu  
	&&\quad \text{on } \GT, \label{CH:GMS:3} \\
	& \phi\vert_{t=0}= \phi_0 
	&&&&\quad \text{in } \Om, \label{CH:GMS:4} \\
	& \psi\vert_{t=0}= \psi_0 = \phi_0\vert_\Ga 
	&&&&\quad \text{on } \Ga, \label{CH:GMS:5}
	\end{alignat}
\end{subequations}
which was introduced and investigated previously in \cite{GMS}.
The condition $\eqref{CH:GMS:3}_2$ means that the chemical potential in the bulk and the chemical potential on the surface are supposed to differ only by a multiplicative constant%
\footnote{In fact, the setting in the referenced paper is even more general as there the factor $\beta$ is allowed to be a function in $L^\infty(\Ga)$ that is uniformly positive $a.e.$ on $\Ga$.}.
This means that, due to the condition $\eqref{CH:GMS:3}_2$, the potentials $\mu$ and $\nu$ are \textit{always} in chemical equilibrium. As the constant $1/L$ can be interpreted as $\infty$, this model describes the idealized case of instantaneous relaxation to chemical equilibrium.
Sufficiently regular solutions of the model satisfy the mass conservation law \eqref{INT:MASS} and the energy dissipation law
\begin{align}
\label{GMS:NRG}
 \ddt E^\text{GL} \big(\phi(t),\psi(t)\big) + m_\Om \intO |\grad\mu(t)|^2 \dx + m_\Ga \intG |\gradg \nu(t)|^2 \dS  = 0,
\end{align}
for all $t\in[0,T]$.

\pagebreak[2]

On the other hand, in the limit $L\to\infty$ we arrive at the system
\begin{subequations}\label{CH:LW}
	\begin{alignat}{3}
	& \delt\phi = m_\Om \Lap \mu, 
	&&\quad \mu = - \eps \Lap \phi + \tfrac 1\eps \Fw'(\phi) 
	&&\quad \text{in } \OT, \label{CH:LW:1}\\
	& \delt\psi = m_\Gamma\Lapg  \nu - \beta m_\Om\deln  \mu, 
	&&\quad \nu = - \delta\kappa \Lapg \psi + \tfrac 1\delta \Gw'(\psi) + \eps\deln \phi 
	&&\quad \text{on } \GT, \label{CH:LW:2}\\
	& \phi\vert_\GT = \psi,
	&&\quad \deln \mu = 0  
	&&\quad \text{on } \GT, \label{CH:LW:3} \\
	& \phi\vert_{t=0} = \phi_0 
	&&&&\quad \text{in } \Om, \label{CH:LW:4} \\
	& \psi\vert_{t=0}= \psi_0 = \phi_0\vert_\Ga 
	&&&&\quad \text{on } \Ga. \label{CH:LW:5}
	\end{alignat}
\end{subequations}
It was derived in \cite{LW} by an energetic variational approach and it was further analyzed in \cite{LW} and \cite{GK}. The crucial difference between \eqref{CH:LW} and the models \eqref{CH:INT} and \eqref{CH:GMS} is that the chemical potentials in the bulk and on the boundary are completely decoupled. 
Accordingly, the constant $1/L$ (that is related to the reaction rate) can be interpreted as zero.
This means that the model does not describe chemical but only mechanical interactions 
(through the trace relation for the phase-fields)
between the bulk and the surface quantities. As a consequence, the bulk and the surface mass are conserved \textit{separately}, i.e., 
\begin{align}
\label{LW:MASS}
\intO \phi(t) \dx 
= \intO \phi_0 \dx 
\quad\text{and}\quad
\intG \psi(t) \dS 
= \intG \psi_0 \dS, \quad t \in [0,T].
\end{align}
Moreover, the dissipation law \eqref{GMS:NRG} holds true for this model. 

Finally, we point out that a variant of the model \eqref{CH:LW}, where $\eqref{CH:LW:3}_2$ was replaced by a Robin type transmission condition $K\deln\phi = H(\psi) - \phi$ (with $K>0$ and a function $H\in C^2(\R)$ satisfying suitable growth conditions), was studied in \cite{KL}. In this case, the quantity $\psi$ can be interpreted as the difference in volume fractions of two materials which are restricted to the boundary.

\paragraph{The nonlocal Cahn--Hilliard equation.}
Although the derivation of the local Cahn--Hilliard equation is physically sound, nonlocal contributions to the total free energy are ignored. This means that only short-range interactions between the particles of the interacting materials are considered. A rigorous derivation of a phase-field model taking both short and long-range interactions into account was firstly presented in \cite{giacomin-lebowitz-1996} (see also \cite{giacomin-lebowitz-1997,giacomin-lebowitz-1998} for more information). This leads to the nonlocal Cahn--Hilliard equation 
\begin{subequations}
	\label{CH:NL}
	\begin{alignat}{3}
	\label{CH:NL:1}
	&\delt \phi =  m_\Om \Lap \mu,
	&&\qquad \mu = \eps\ao\phi - \eps J * \phi + \tfrac 1 \eps \Fw'(\phi) 
	&&\quad\text{in}\; \OT,\\
	\label{CH:NL:2}
	&\deln \mu = 0  
	&&&&\quad\text{on}\; \GT, \\
	\label{CH:NL:PHI:3}
	&\phi\vert_{t=0} = \phi_0 
	&&&&\quad\text{in}\; \Om,
	\end{alignat}
\end{subequations}
where $J:\R^d\to\R$ is an even interaction kernel, i.e., $J(x)=J(-x)$ for all $x\in\R^d$.
This system of equations can be interpreted as the gradient flow of the nonlocal Helmholtz free energy
\begin{align}
\label{E:NL:BULK}
E_\text{bulk}(\phi) = \frac \eps 4 \intO\intO J(x-y) |\phi(x)-\phi(y)|^2 \dy \dx
	+ \frac 1 \eps \intO \Fw(\phi(x)) \dx
\end{align}
with respect to a suitable inner product of order $H^{-1}$ (see, e.g., \cite{GGG}). By the substitution $F(x,\phi) = \Fw(\phi) + \tfrac 12 a_\Om(x) \phi^2$ with $a_\Om(x) = (J*1)(x)$, the nonlocal Cahn--Hilliard equation \eqref{CH:NL} as well as the free energy $E_\text{bulk}$ can be expressed equivalently as\pagebreak[2]
\begin{subequations}
	\label{CH:NL:ALT}
	\begin{alignat}{3}
	\label{CH:NL:ALT:1}
	&\delt \phi =  m_\Om \Lap \mu, 
	&&\qquad \mu = - \eps J * \phi + \tfrac 1 \eps F'(\cdot,\phi) &&\quad\text{in}\; \OT,\\
	\label{CH:NL:ALT:2}
	&\deln \mu = 0  
	&&&&\quad\text{on}\; \GT, \\
	\label{CH:NL:ALT:3}
	&\phi\vert_{t=0} = \phi_0 &&&&\quad\text{in}\; \Om,
	\end{alignat}
\end{subequations}
and 
\begin{align}
\label{E:NL:ALT:BULK}
E_\text{bulk}(\phi) = - \frac \eps 2 \intO\intO J(x-y)\phi(x)\phi(y) \dx\dy
	+ \frac 1 \eps \intO F(x,\phi(x)) \dx.
\end{align} 
The nonlocal Cahn--Hilliard equation has already been investigated from many different viewpoints. We refer the reader to \cite{abels-bosia-grasselli,bates-han,bates-han-dir,GGG,porta-grasselli,shaker-hassan,gajewski-zacharias,gal-doubly,gal-doubly-strong,davoli-2019} for various well-posedness results, 
to \cite{colli-frigeri-grasselli,frigeri-gal-grasselli,frigeri-grasselli,porta-grasselli-fluid, porta-giorgini-grasselli} for the investigation of the nonlocal Cahn--Hilliard equation coupled to fluid equations,
and also to \cite{abels-bosia-grasselli,frigeri-grasselli-long,gal-grasselli-long,londen-petzeltova,
londen-petzeltova-sep} for results on long-time behavior. The convergence of the nonlocal Cahn--Hilliard equation to the local Cahn--Hilliard equation, under suitable assumptions on the convolution kernel $J$, has been investigated in \cite{davoli-scarpa-trussardi,davoli-scarpa-trussardi-W11,davoli-2019,melchionna,bonetti}.

\paragraph{The nonlocal Cahn--Hilliard equation with dynamic boundary conditions.}
As mentioned above, there are already several works concerning the local Cahn--Hilliard equation with dynamic boundary conditions.  
However, the authors were able to find only one contribution dealing with the nonlocal Cahn--Hilliard equation with dynamic boundary conditions in the literature, see \cite{gal-dynamic}. Therein, the nonlocal Cahn--Hilliard equation subject to fractional dynamic boundary conditions is studied. 
The system of equations (in which the mobilities as well as $\eps$ and $\delta$ are set to one) reads as
\pagebreak[2]
\begin{subequations}
	\label{CH:GAL}
	\begin{alignat}{3}
	\label{CH:GAL:1}
	&\delt \phi =  \Lap \mu,
	\qquad \mu = (-\Lap)^s \phi + \Fw'(\phi) &&\quad\text{in}\; \OT,\\
	\label{CH:GAL:3}
	&\delt \psi = (-\Lapg)^l\psi + C_sN^{2-2s}\phi + \beta\psi + \Gw'(\psi)  &&\quad\text{on}\; \GT,\\
	\label{CH:GAL:INI}
	&\phi\vert_{t=0} = \phi_0 &&\quad\text{in}\; \Om,
	\end{alignat}
\end{subequations}
with $\frac 1 2 <s<1$, $0<l<1$ and $\beta>0$. Here, $(-\Lap)^s$ denotes the regional fractional Laplace operator, $(-\Lapg)^l$ denotes the fractional Laplace--Beltrami operator and $C_sN^{2-2s}$ stands for the fractional normal derivative. By the formal choice $s=l=1$, this model corresponds to the local Cahn--Hilliard system subject to a dynamic boundary condition of Allen--Cahn type (as proposed in \cite{Fis2,binder}).

In this paper, however, we pursue a different idea. To describe both short and long-range interaction of the materials on the boundary, we define the nonlocal total free energy (of Helmholtz type) as
\begin{align}
\label{DEF:EN}
\notag E(\phi,\psi) &:= E_\text{bulk}(\phi) + E_\text{surf}(\psi) + E_\text{pen}(\psi)\\
\notag&:= \left[ \frac \eps 4 \intO \intO J(x-y) \abs{\phi(x)-\phi(y)}^2 \dx\dy 
+ \frac 1 \eps \intO F_{\rm w}(x,\phi(x)) \dx \right]\\
\notag&\quad + \left[ \frac \delta 4 \intG \intG K(z-y) \abs{\psi(x)-\psi(y)}^2 \dS(z)\dS(y) 
+ \frac 1 \delta \intG G_{\rm w}(z,\psi(z)) \dS(z) \right] \\
&\quad + \frac 1 \delta \intG B(z,\psi(z)) \dS(z) ,
\end{align}
where $E_\text{surf}$ is introduced by analogy to the surface free energy defined in \eqref{DEF:EN:GL}. As in \cite{KL}, the function $\psi$ can be interpreted as the difference of local relative concentrations of two materials that are restricted to the boundary.
The function $K:\R^d\to\R$ defines an additional even kernel modeling both short- and long-range interactions among the materials described by $\psi$.

For generality, we also allow the potentials $\Fw=\Fw(x,s)$ and $\Gw=\Gw(z,s)$ to depend also on the spatial variables $x\in\Om$ and $z\in\Ga$. This generalization might make sense especially for $\Gw$ if the solid wall of the container consists of several materials interacting differently with the mixture inside. For this reason, we also allow an additional penalty term $E_\text{pen}$ in the total free energy. For instance, we can choose the function 
\begin{align}
\label{DEF:PEN}
	B:\Ga\times\R\to\R, \quad (z,s) \mapsto b(z)s,
\end{align}
where $b$ can be interpreted as a weight function. Then the corresponding penalty term $E_\text{pen}$ is expected to describe different regions of the boundary that attract the material associated with  $\psi = \pm 1$ and repel the other material associated with $\psi = \mp 1$. For more detail see Remark~\ref{EX:PEN}.

Now, the system \eqref{CH:ROB} can be derived as the gradient flow equation of the energy $E$ with respect to a suitable inner product of order $H^{-1}$ containing both a bulk and a surface contribution. We present a rigorous derivation in Section~\ref{SEC:GF}. 
In contrast to the model \eqref{CH:GAL} studied in \cite{gal-dynamic}, the dynamic boundary condition in \eqref{CH:ROB} is also of nonlocal Cahn--Hilliard type involving the phase field $\psi$ and the chemical potential $\nu$ on the surface. 

Since the chemical potentials are coupled by the Robin type boundary condition \eqref{CH:ROB:BC}, the system \eqref{CH:ROB} (with $B\equiv 0$) can be regarded as the nonlocal analogue of the system \eqref{CH:INT} studied in \cite{KLLM}.
As a consequence, sufficiently regular solutions to \eqref{CH:ROB} satisfy the same mass conservation and dissipation properties as system \eqref{CH:INT}. To be precise, this means that the mass conservation law
\begin{align}
\label{CH:ROB:MASS}
\beta\intO \phi(t) \dx + \intG \psi(t) \dS = \beta\intO \phi_0 \dx + \intG \psi_0 \dS
\end{align}
as well as the energy dissipation law
\begin{align}
\label{CH:ROB:DISS}
\begin{aligned}
&\ddt E\big(\phi(t),\psi(t)\big) \\
&\quad + m_\Om \intO |\grad\mu(t)|^2 \dx 
+ m_\Ga \intG |\gradg \nu(t)|^2 \dS 
+ \frac{m_\Om}{L} \intG \abs{\beta\nu-\mu}^2 \dS = 0
\end{aligned}
\end{align}
are satisfied for all $t\in[0,T]$. 
As the total free energy $E$ is bounded from below (for reasonable choices of $J$, $K$, $\Fw$, $\Gw$ and $B$), we infer that $\ddt E(\phi(t),\psi(t))$ converges to zero as $t\to\infty$. 
Hence, the potentials $\mu$ and $\nu$ converge to the chemical equilibrium $\mu\vert_\GT=\beta \nu$ over the course of time.

By the substitutions $F(x,\phi) = \Fw(x,\phi) + \tfrac 12 a_\Om(x) \phi^2$ with $a_\Om(x) = (J*1)(x)$, $x\in\Om$ and $G(z,\psi) = \Gw(z,\psi) + \tfrac 12 a_\Ga(z) \psi^2$ with $a_\Ga(z) = (K\cg 1)(z)$, $z\in\Ga$, the energy $E$ and the system \eqref{CH:ROB} can be expressed equivalently as
\begin{align}
\label{DEF:EN:ALT}
	\begin{aligned}
	E(\phi,\psi) &= \left[ - \frac \eps 2 \intO (J\ast\phi) \phi \dx
	+ \frac 1 \eps \intO F(\cdot,\phi) \dx \right]\\
	&\quad + \left[ -\frac \delta  2 \intG (K\cg \psi) \psi \dS 
	+ \frac 1 \delta \intG G(\cdot,\psi) \dS \right] 
	+ \frac 1 \delta  \intG B(\cdot,\psi) \dS ,
	\end{aligned}
\end{align}
and 
\begin{subequations}\label{CH:ROB:ALT}
	\begin{alignat}{3}
	& \delt\phi = m_\Om \Lap \mu, 
	&&\quad  \mu = - \eps J*\phi + \tfrac 1\eps F'(\cdot, \phi) 
	&&\quad \text{in } \OT, \label{CH:ROB:ALT:1}\\
	& \delt\psi = m_\Gamma\Lapg  \nu - \beta m_\Om\deln  \mu, 
	&&\quad  \nu = - \delta  K\cg\psi  + \tfrac 1 \delta G'(\cdot,\psi) + \tfrac 1 \delta  B'(\cdot,\psi)
	&&\quad \text{on } \GT, \label{CH:ROB:ALT:2}\\
	&L\deln \mu  = \beta  \nu - \mu 
	&&&& \quad \text{on } \Sigma_T, \label{CH:ROB:ALT:4} \\
	& \phi\vert_{t=0} = \phi_0 
	&&&&\quad \text{in } \Om, \label{CH:ROB:ALT:5} \\
	& \psi\vert_{t=0} = \psi_0 
	&&&&\quad \text{on } \Ga. \label{CH:ROB:ALT:6}
	\end{alignat}
\end{subequations}
In the following, we will switch between these equivalent formulations at our convenience.

\subsection{The singular limits $L\to 0$ and $L\to\infty$}

Similar to the corresponding local model \eqref{CH:INT}, the constant $1/L$ is related to the reaction rate (in the sense of chemical reactions as well as adsorption and desorption processes). As in the local case, we are also interested in the singular limits $L\to 0$ and $L\to\infty$ of the system \eqref{CH:ROB}.

Passing to the limit $L\to 0$ in the system \eqref{CH:ROB} we (formally) obtain the boundary condition $\beta\nu=\mu\vert_\GT$ on $\GT$.  
Thus, we can express the limit system as follows:
\pagebreak[2]
\begin{subequations}
	\label{CH:DIR}
	\begin{alignat}{2}
	\label{CH:DIR:PHI}
	&\delt \phi =  m_\Om \Lap \mu &&\quad\text{in}\; \OT,\\
	\label{CH:DIR:MU}
	&\mu= \eps\, \ao\phi - \eps J * \phi + \tfrac 1 \eps\Fw'(\cdot,\phi) &&\quad\text{in}\; \OT,\\[1ex]
	\label{CH:DIR:PSI}
	&\delt \psi =  \tfrac 1 \beta m_\Ga \Lapg \mu - \beta m_\Om\, \deln \mu &&\quad\text{on}\; \GT,\\
	\label{CH:DIR:NU}
	&\mu\vert_\GT 
	= \beta\delta\, \ag\psi - \delta \beta K \cg \psi + \tfrac \beta \delta \Gw'(\cdot,\psi) + \tfrac \beta \delta B'(\cdot,\psi) &&\quad\text{on}\; \GT,\\[1ex]
	\label{CH:DIR:PHI:INI}
	&\phi\vert_{t=0} = \phi_0 &&\quad\text{in}\; \Om, \\
	\label{CH:DIR:PSI:INI}
	&\psi\vert_{t=0} = \psi_0 &&\quad\text{on}\; \Ga.
	\end{alignat}
\end{subequations} 
It can be interpreted as the nonlocal analogue to the model \eqref{CH:GMS}. 
Here, the constant $1/L$ can again be interpreted as infinity meaning that this model describes the idealized scenario of instantaneous relaxation to the situation where the chemical potentials $\mu$ and $\nu = \beta^{-1}\mu\vert_\GT$ are in chemical equilibrium. Sufficiently regular solutions satisfy the mass conservation law \eqref{CH:ROB:MASS} as well as the energy dissipation law
\begin{align}
\label{CH:DIR:DISS}
	\ddt E\big(\phi(t),\psi(t)\big) 
	+ m_\Om \intO |\grad\mu(t)|^2 \dx 
	+ m_\Ga \intG |\gradg \nu(t)|^2 \dS 
	= 0. 
\end{align}

In the limit $L\to \infty$ we (formally) obtain the boundary condition $\deln\mu = 0$ on $\GT$. As a consequence, the subsystems for $(\phi,\mu)$ and for $(\psi,\nu)$ are completely decoupled. 
To be precise, the pair $(\phi,\mu)$ satisfies the standard nonlocal Cahn--Hilliard equation with homogeneous Neumann boundary condition 
\begin{subequations}
	\label{CH:DEC:BULK}
	\begin{alignat}{3}
	\label{CH:DEC:BULK:1}
	&\delt \phi =  m_\Om \Lap \mu ,
	&&\quad \mu = \ao \phi - \eps J * \phi + \tfrac 1 \eps \Fw'(\cdot,\phi) &&\quad\text{in}\; \OT,\\
	\label{CH:DEC:BULK:2}
	&\deln\mu = 0 
	&&&&\quad\text{on}\; \GT,\\
	\label{CH:DEC:BULK:3}
	&\phi\vert_{t=0} = \phi_0 
	&&&&\quad\text{in}\; \Om,
	\end{alignat}
\end{subequations}
whereas the pair $(\psi,\nu)$ satisfies the following nonlocal Cahn--Hilliard type equation on the surface
\begin{subequations}
	\label{CH:DEC:SURF}
	\begin{alignat}{3}
	\label{CH:DEC:SURF:1}
	&\delt \psi =  m_\Ga\Lapg \nu,
	&&\quad \nu = \ag \psi - \delta K \cg \psi + \tfrac 1 \delta \Gw'(\cdot,\phi) 
		+ \tfrac 1 \delta B'(\cdot,\phi) 
	&&\quad\text{on}\; \GT,\\
	\label{CH:DEC:SURF:2}
	&\psi\vert_{t=0}= \psi_0 
	&&&&\quad\text{on}\; \Ga.
	\end{alignat}
\end{subequations}
For that reason, we refer to (\eqref{CH:DEC:BULK},\eqref{CH:DEC:SURF}) as the ``decoupled model''.
If $B \equiv 0$, this model can be interpreted as the nonlocal analogue to the model \eqref{CH:LW}. However, in contrast to \eqref{CH:LW} where the phase fields $\phi$ and $\psi$ were coupled by the Dirichlet type condition $\eqref{CH:LW:3}_1$, 
the systems \eqref{CH:DEC:BULK} and \eqref{CH:DEC:SURF} are completely independent. This means that the bulk and the surface materials are assumed to not interact at all. 
As the mass flux $- m_\Om \deln \mu$ is zero, we obtain separate conservation of bulk and boundary mass, i.e.,
\begin{align}
\label{DEC:MASS}
\intO \phi(t) \dx 
= \intO \phi_0 \dx 
\quad\text{and}\quad
\intG \psi(t) \dS 
= \intG \psi_0 \dS, \quad t \in [0,T].
\end{align}
Moreover, the energy dissipation law \eqref{CH:DIR:DISS} is satisfied by sufficiently regular solutions. 

Since the interfacial thickness parameters $\eps$ and $\delta$ as well as the constant mobilities $m_\Om$
and $m_\Ga$ will not play any role in the analysis, we conveniently set them to one
in the rest of the paper.

\section{Preliminaries}

\subsection{Notation} Throughout this paper we use the following notation:  For any $1 \leq p \leq \infty$ and $k \geq 0$, the standard Lebesgue and Sobolev spaces defined on $\Omega$ are denoted as $L^p(\Omega)$ and $W^{k,p}(\Omega)$, along with the norms $\norm{\cdot}_{L^p(\Omega)}$ and $\norm{\cdot}_{W^{k,p}(\Omega)}$. For the case $p = 2$, these spaces become Hilbert spaces and we use the notation $H^k(\Omega):= W^{k,2}(\Omega)$. For any exponent $p>1$, we write $\pp$ to denote its dual Sobolev exponent, i.e., $(1/p)+(1/\pp)=1$. We point out that $H^0(\Omega)$ can be identified with $L^2(\Omega)$. A similar notation is used for Lebesgue and Sobolev spaces on $\Gamma$. The definition of tangential gradients on Lipschitz surfaces can be found, e.g., in \cite[Def.~3.1]{buffa}.
For any Banach space $X$, we denote its topological dual space by $X'$ and the associated duality pairing by $\inn{\cdot}{\cdot}_X$.  If $X$ is a Hilbert space, we denote its inner product by $(\cdot, \cdot)_X$.  We define 
\begin{align*}
\mean{u}_\Omega := \begin{cases}
\frac{1}{\abs{\Omega}} \inn{u}{1}_{H^1(\Omega)} & \text{ if } u \in H^1(\Omega)', \\
\frac{1}{\abs{\Omega}} \int_\Omega u \dx & \text{ if } u \in L^1(\Omega)
\end{cases}
\end{align*}
as the spatial mean of $u$, where $\abs{\Omega}$ denotes the $d$-dimensional Lebesgue measure of $\Omega$. 
The spatial mean for $v \in H^1(\Gamma)'$ and $v \in L^1(\Gamma)$ can be defined analogously
and will be denoted by $\mean{v}_\Ga $.
For any real number $p\in[1,\infty]$ and any integer $k\ge 0$, we set
\begin{align*}
\LL^p := L^p(\Omega) \times L^p(\Gamma), 
\quad\text{and}\quad
\HH^k := H^k(\Omega) \times H^k(\Gamma),
\end{align*}
and we identify $\HH^0$ with $\LL^2$. We notice that the space $\HH^k$ is a Hilbert space with respect to the inner product
\begin{align*}
\bigscp{(\phi,\psi)}{(\zeta,\xi)}_{\HH^k} := \bigscp{\phi}{\zeta}_{H^k(\Omega)} + \bigscp{\psi}{\xi}_{H^k(\Gamma)}
\end{align*}
and its induced norm $\norm{\,\cdot\,}_{\HH^k}:= \scp{\cdot}{\cdot}_{\HH^k}^{1/2}$.

\bigskip
\subsection{Assumptions}
\paragraph{General assumptions.}
\begin{enumerate}[label=$(\mathrm{A \arabic*})$, ref = $\mathrm{A \arabic*}$]
\item \label{ass:dom} We take $\Omega \subset \R^d$ with $d \in \{2,3\}$ to be a bounded domain with Lipschitz boundary $\Gamma$. For any $t>0$ we write $\Om_t:=\Omega\times(0,t)$ as well as $\Ga_t:=\Gamma\times(0,t)$. 
Since $\Omega$ is bounded, we can find a positive radius $R>0$ such that $\ov\Omega \subset B_R(0)$ where $B_R(0)$ denotes the open ball in $\R^d$ with radius $R$ and center $0$.
In general, we assume that $T>0$. If not stated otherwise, we suppose that $L>0$ and $\beta\neq 0$. The mobilities $\mo$ and $\mg$, and the interface parameters $\eps$ and $\delta$ are positive constants. As their choice does not have any impact on the mathematical analysis, they will be set to one from now on.
\item \label{ass:kernel}
We assume that the convolution kernels $J,K:\R^d\to\R$ are even (i.e., $J(x) =J(-x)$ and $K(x) =K(-x)$ for almost all $x\in\R^d$), nonnegative almost everywhere, and satisfy $J\in W^{1,1}(\R^d)$ and $K\in W^{2,r}(\R^d)$ with $r>1$. 
Note that the regularity assumption on $K$ is higher than that on $J$ since the traces $K(z-\cdot)\vert_\Gamma$ and $\gradg K(z-\cdot)\vert_\Gamma$ must exist and belong to $L^r(\Gamma)$ for all $z\in\Gamma$ (see, e.g., Lemma~\ref{LEM:JK:UNIQ} and its proof).

In addition, we suppose that 
\begin{subequations}
\label{DEF:A}
	\begin{alignat}{3}
	\label{DEF:AINF}
	a_* &:= \underset{x\in\Om}{\inf} \intO J(x-y) \dy > 0, 
	&&\qquad
	a_\cg &&:= \underset{z\in\Ga}{\inf} \intG K(z-y) \dSy > 0, \\
	\label{DEF:ASUP}
	a^* &:= \underset{x\in\Om}{\sup} \intO J(x-y) \dy < \infty, 
	&&\qquad
	a^\cg &&:= \underset{z\in\Ga}{\sup} \intG K(z-y) \dSy < \infty, \\
	b^* &:= \underset{x\in\Om}{\sup} \intO |\grad J(x-y) |\dy < \infty, 
	&&\qquad
	b^\cg &&:= \underset{z\in\Ga}{\sup} \intG| \gradg K(z-y)| \dSy < \infty
	\label{DEF:BSUP}		
	\end{alignat} 
\end{subequations}
where $\inf$ and $\sup$ are to be understood as the essential infimum and the essential supremum, respectively.
We further use the notation
\begin{subequations}
\begin{alignat}{3}
	\label{DEF:AO}
	a_\Om(x) &:= \big(J * 1\big)(x) &&= \intO J(x-y) \dy, &&\quad\text{for almost all}\; x\in\Omega, \\
	\label{DEF:AG}
	a_\Ga(z) &:= \big(K \cg 1\big)(z) &&= \intG K(z-y) \dSy, &&\quad\text{for almost all}\; z\in\Gamma.
\end{alignat} 
We point out that $a_\Om\in L^\infty(\Omega)$ and $a_\Ga\in L^\infty(\Gamma)$ due to \eqref{DEF:A}.
\end{subequations}
\item  \label{ass:pot:dw}
We suppose that the potentials $\Fw:\Om\times\R \to \R$, $\Gw: \Ga\times\R \to \R$ are nonnegative, satisfy $\Fw(x,\cdot),\Gw(z,\cdot)\in C^2(\R)$ for almost all $x\in\Om$, $z\in\Ga$, and satisfy $\Fw(\cdot,s),\Gw(\cdot,s)\in L^\infty(\R)$ for all $s\in\R$.
The derivatives with respect to the second variable are denoted by $\Fw'$, $\Fw''$, $\Gw'$ and $\Gw''$. We assume that
there exist constants $c_*,c_\cg >0$ such that for all $s\in \R$, and
almost all $x\in\Om$, $z\in\Ga$, 
\begin{align}
	\label{EST:FGW}
	\Fw''(x,s) + a_* \ge c_*
	\quad\text{and}\quad
	\Gw''(z,s) + a_\cg \ge c_\cg.
\end{align}
Moreover, we assume that there exist $\alpha_{\Fw'},\alpha_{\Gw'},\gamma_{\Fw'},\gamma_{\Gw'}>0$ and $\delta_{\Fw'},\delta_{\Gw'}\ge 0$ as well as exponents $p,q \in \R $ with
\begin{align}
\label{ASS:PQ}
\left\{
	\begin{aligned}
		p > 2 &\;\text{ and }\; q \ge 2 &&\text{ if }\; d=2, \\
		p > 3 &\;\text{ and }\; q > 2 &&\text{ if }\; d=3,
	\end{aligned}
\right.
\qquad
\pp=\frac{p}{p-1}, \qquad
\qp=\frac{q}{q-1}
\end{align}
such that for all $s\in\R$ and almost all $x\in\Om$, $z\in\Ga$,
\begin{subequations}
	\label{ass:growth:w}
	\begin{alignat}{3}
	\label{ass:growth:1a}
	\alpha_{\Fw'} \abs{s}^{p-1} - \delta_{\Fw'} &\le \abs{\Fw'(x,s)} &&\le \gamma_{\Fw'}(1 + \abs{s}^{p-1}), \\
	\label{ass:growth:1b}
	\alpha_{\Gw'} \abs{s}^{q-1} - \delta_{\Gw'} &\le \abs{\Gw'(z,s)} &&\le \gamma_{\Gw'}(1 + \abs{s}^{q-1}).
	\end{alignat}
	As a consequence, there exist constants $\alpha_{\Fw},\alpha_{\Gw},\gamma_{\Fw},\gamma_{\Gw}>0$ and $\delta_{\Fw},\delta_{\Gw}\ge 0$ such that 
	\begin{alignat}{3}
	\label{ass:growth:2a}
	\alpha_{\Fw} \abs{s}^{p} - \delta_{\Fw} &\le \Fw(x,s) &&\le \gamma_{\Fw}(1 + \abs{s}^{p}), \\
	\label{ass:growth:2b}
	\alpha_{\Gw} \abs{s}^{q} - \delta_{\Gw} &\le \Gw(z,s) &&\le \gamma_{\Gw}(1 + \abs{s}^{q}),
	\end{alignat}	
	for all $s\in\R$ and almost all $x\in\Om$, $z\in\Ga$. 
\end{subequations}
\item  \label{ass:pot:conv}
We define the potentials $F$ and $G$ by
\begin{align}
\label{EST:FG:1}
	F(x,s) := \Fw(x,s) + \frac 1 2 a_\Omega(x) s^2
	\quad\text{and}\quad
	G(z,s) := \Gw(z,s) + \frac 1 2 a_\Gamma(z) s^2
\end{align}	
for all $s\in\R$ and almost all $x\in\Om$ and $z\in\Gamma$. In view of $a_\Om\in L^\infty(\Omega)$ and $a_\Ga\in L^\infty(\Gamma)$ we conclude that for almost every fixed $x\in\Om$ and $z\in\Ga$, the functions $F(x,\cdot)$ and $G(z,\cdot)$ are twice continuously differentiable. 

Moreover, we write
\begin{alignat*}{3}
	F'(x,s) &= \Fw'(x,s) + \ao(x) s, 
	&&\quad F''(x,s) &&= \Fw''(x,s) + \ao(x), \\
	G'(z,s) &= \Gw'(z,s) + \ag(z) s, 
	&&\quad G''(z,s) &&= \Gw''(z,s) + \ag(z), 
\end{alignat*}
for all $s\in\R$ and almost all $x\in \Om$ and $z\in\Ga$ to denote the derivatives of $F(x,\cdot)$ and $G(z,\cdot)$ with respect to the second variable. It follows from \eqref{EST:FGW} that for all $s\in\R$ and almost all $x\in \Om$ and $z\in\Ga$,
\begin{align}
	\label{EST:FG:2}
	F''(x,s) \ge \Fw''(x,s) + a_* \ge c_*
	\quad\text{and}\quad
	G''(z,s) \ge \Gw''(z,s) + a_\cg \ge c_\cg
\end{align}	
meaning that $F(x,\cdot)$ and $G(z,\cdot)$ are uniformly convex for almost all $x\in \Om$ and $z\in\Ga$. Since $\ao\in L^\infty(\Omega)$, $\ag\in L^\infty(\Gamma)$ and $p,q\ge2$ (according to \eqref{ASS:PQ}), we infer from \eqref{ass:growth:w} that there exist $\alpha_{F'},\alpha_{G'},\alpha_{F},\alpha_{G},\gamma_{F'},\gamma_{G'},\gamma_{F},\gamma_{G}>0$ and $\delta_{F'},\delta_{G'},\delta_{F},\delta_{G}\ge 0$ such that 
\begin{subequations}
	\label{ass:growth}
	\begin{alignat}{3}
	\alpha_{F'} \abs{s}^{p-1} - \delta_{F'} &\le |F'(x,s)| &&\le \gamma_{F'}(1 + \abs{s}^{p-1}), \\
	\alpha_{G'} \abs{s}^{q-1} - \delta_{G'} &\le | G'(z,s) | &&\le \gamma_{G'}(1 + \abs{s}^{q-1}), \\
	\alpha_{F} \abs{s}^{p} - \delta_{F} &\le \,\,\, F(x,s) &&\le \gamma_{F}(1 + \abs{s}^{p}), \\
	\alpha_{G} \abs{s}^{q} - \delta_{G} &\le \,\,\, G(z,s) &&\le \gamma_{G}(1 + \abs{s}^{q}),
	\end{alignat}	
	for all $s\in\R$ and almost all $x\in\Om$ and $z\in\Ga$.
\end{subequations}

\item \label{ass:pen}
We suppose that $B\in L^\infty(\Gamma\times\R)$ with $B(z,\cdot)\in C^1(\R)$ for almost all $z\in\Ga$ and we write $B'$ to denote the derivative with respect to the second variable.
Moreover, we assume that the following holds:
\begin{enumerate}[label=$(\mathrm{A 5.\arabic*})$, ref = $\mathrm{A 5.\arabic*}$]
\item \label{ass:pen:1} There exist constants $\alpha_B,\alpha_{B'},\gamma_B,\gamma_{B'} > 0$ with 
$\alpha_B<\min\{\alpha_{G},\alpha_{\Gw}\}$ such that 
\begin{align}
	\label{EST:PEN}
	\abs{B(z,s)} \le \alpha_B\abs{s}^q + \gamma_B,
	\quad\text{and}\quad
	\abs{B'(z,s)} \le \alpha_{B'}\abs{s}^{q-1} + \gamma_{B'}
\end{align}
for all $s\in\R$ and almost all $z\in\Gamma$ where $q$ is the exponent from \eqref{ass:growth:w}.
\item \label{ass:pen:2} For any sequence $\psi_k\wto \psi$ in $L^q(\Gamma)$, it holds that
\begin{align*}
	\intG B(\cdot,\psi) \dS &\le \underset{k\to\infty}{\lim\inf}\; \intG B(\cdot,\psi_k) \dS, \\
	\intG B'(\cdot,\psi) \theta \dS &= \underset{k\to\infty}{\lim}\; \intG B'(\cdot,\psi_k) \theta \dS,
\end{align*}
for all test functions $\theta\in L^{q}(\Ga)$. Note that $B'(\cdot,\psi)\in L^\qp(\Ga)$ for all $\psi\in L^q(\Ga)$ due to the growth condition on $B'$ in \eqref{ass:pen:1}.
\item \label{ass:pen:uniq} We assume that $B'$ is Lipschitz continuous with respect to its second argument in the following sense: There exists a constant $0\le \LIP<c_\cg$ such that 
\begin{align*}
\abs{B'(z,r)-B'(z,s)} \le \LIP \abs{r-s}
\end{align*}
for all $r,s\in\R$ and almost all $z\in\Ga$.
\end{enumerate}
\end{enumerate}

\medskip

\paragraph{Additional assumptions for higher regularity.}
\begin{enumerate}[label=$(\mathrm{A \arabic*})$, ref = $\mathrm{A \arabic*}$, start=6]
\item \label{ass:dom:strong} We assume that the boundary $\Gamma$ is of class $C^2$. 
\item \label{ass:pot:strong}
We suppose that there exist constants $\gamma_{F''},\gamma_{G''} >0$ such that
\begin{alignat*}{3}
F''(x,s) &\leq \gamma_{F''} (1+|s|^{p-2}), \quad
G''(z,s) &\leq \gamma_{G''} (1+|s|^{q-2})
\end{alignat*}
for all $s\in\R$ and almost all $x\in\Om$ and $z\in\Ga$.
\item \label{ass:pen:strong}
We assume that $B(z,\cdot)\in C^2(\R)$ for almost all $z\in\Ga$ and that there exists a constant $0<\gamma_{B''}<c_\cg$ such that 
\begin{align*}
	\abs{B''(z,s)} \le \gamma_{B''}
\end{align*}
for all $s\in\R$ and almost all $z\in\Ga$. In this case, we can of course choose $\LIP=\gamma_{B''}$.

\end{enumerate}
	
\medskip

\begin{rem}	
	\begin{enumerate}[label=$(\mathrm{\alph*})$, ref = $\mathrm{\alph*}$]
	\item The assumptions on the kernel $J$ in \eqref{ass:kernel} are very typical in the literature (cf. \cite{davoli-scarpa-trussardi-W11}).
	Sometimes the apparently weaker requirement $J\in W^{1,1}_\text{loc}(\R^d)$ is prescribed (see, e.g., \cite{GGG}). However, any kernel $\tilde J\in W^{1,1}_\text{loc}(\R^d)$ can always be multiplied by a compactly supported cut-off function $\rho\in C^1(\R^d) $ with $\rho=1$ in $\ov{B_{2R}(0)}$,
	where $R$ is the radius from \eqref{ass:dom}. 
	Then $J:=\tilde J \rho \in W^{1,1}(\R^d)$, and
	since $x-y\in\ov{B_{2R}(0)}$ for all $x,y\in\ov\Om$, 
	it holds that $J*\phi = \tilde J * \phi$ $a.e.$ in $\Om$.
	For that reason, it is not a real restriction to demand $J\in W^{1,1}(\R^d)$ instead of $J\in W^{1,1}_\text{loc}(\R^d)$ in \eqref{ass:kernel}.
	Of course, we can argue analogously for the assumptions on $K$.
	\item As a direct consequence of \eqref{DEF:ASUP} and \eqref{DEF:BSUP}, we conclude that
	for all functions $\phi\in L^1(\Omega)$ and $\psi\in L^1(\Ga)$, it holds that
	$J\ast\phi\in W^{1,1}(\Om)$ and $K\cg\psi\in W^{1,1}(\Ga)$ with 
	\begin{subequations}
		\begin{alignat}{3}
		\label{EST:J}	
		\Vert J\ast\phi\Vert_{L^1(\Om)}
		&\leq a^\ast\,\Vert\phi\Vert_{L^1(\Om)}\,,
		\\
		\label{EST:K}
		\Vert K\cg\psi\Vert_{L^1(\Ga)}
		&\leq a^\cg\,\Vert\psi\Vert_{L^1(\Ga)}\,,
		\\
		\label{EST:DJ}
		\Vert \nabla (J\ast\phi) \Vert_{L^1(\Om)} =\,  
		\Vert (\nabla J)\ast\phi\Vert_{L^1(\Om)}
		&\leq b^*\,\Vert\phi\Vert_{L^1(\Om)}\,,
		\\
		\label{EST:DK}
		\Vert \gradg ( K\cg\psi ) \Vert_{L^1(\Ga)} =\,
		\Vert ( \gradg K )\cg\psi\Vert_{L^1(\Ga)}
		&\leq b^\cg\,\Vert\psi\Vert_{L^1(\Ga)}
		.	
		\end{alignat}
	\end{subequations}
	\end{enumerate}
	\label{rem:JK}
\end{rem}

\subsection{Special spaces, products and norms}
\begin{enumerate}[label=$(\mathrm{P \arabic*})$, ref = $\mathrm{P \arabic*}$]
\item \label{pre:space} Let $\beta\neq 0$, $m\in\R$, and let $k$ denote any integer. We define the spaces
\begin{alignat*}{3}
	\VV^k &:= \left\{ \zeta \in H^k(\Omega) \,:\, \zeta\vert_\Gamma \in H^k(\Gamma)  \right\}
	&&\quad\text{for}\; k\ge 1,\\
	\HHBM &:= \left\{(\zeta,\xi)\in \HH^k \,:\, \beta\abs{\Omega}\mean{\zeta}_\Om + \abs{\Gamma}\mean{\xi}_\Ga = m\right\}
	&&\quad\text{for}\; k\ge 0.
\end{alignat*}
Note that the space $\VV^k$, endowed with the inner product
\begin{align*}
	(\phi,\zeta)_{\VV^k} := (\phi,\zeta)_{H^k(\Om)} + (\phi,\zeta)_{H^k(\Ga)}
\end{align*}
and its induced norm is a Hilbert space. 

For any $L \ge 0$ and $\beta \neq 0$ we introduce the bilinear form
\begin{align*}
\begin{aligned}
&\bigscp{(\mu,\nu)}{(\zeta,\xi)}_{L,\beta}  \\ 
&\quad:= 
\begin{cases}
	\displaystyle
	\intO \nabla \mu \cdot \nabla \zeta \dx + \intG \gradg \nu \cdot \gradg \xi + \frac{1}{L}(\beta \nu-\mu)(\beta\xi-\zeta) \dS, &\text{if}\; L>0,\\[2ex]
	\displaystyle
	\intO \nabla \mu \cdot \nabla \zeta \dx + \intG \gradg \nu \cdot \gradg \xi \dS, &\text{if}\; L=0,
\end{cases}
\end{aligned}
\end{align*}
for all $(\mu,\nu), (\zeta,\xi) \in \HH^1$. 
It is straightforward to check that
$\scp{\cdot}{\cdot}_{L,\beta}$ defines an inner product on $\HH^1_{\beta,0}$ if $L>0$ and $\beta\neq 0$, and $\scp{\cdot}{\cdot}_{0,\beta}$ defines an inner product on $\HH^1_{\beta,0}$ if $L=0$ and $\beta>0$.
The induced norm is given by $\norm{\,\cdot\,}_{L,\beta}:= \scp{\cdot}{\cdot}_{L,\beta}^{1/2}$. 
Note that the dual space of $\HH^1$ is given as
\begin{align*}
(\HH^1)' &= H^1(\Omega)' \times H^1(\Gamma)', \quad
\end{align*}
We further define the subspace
\begin{align*}
	\HH_{\beta,0}^{-1} = \left\{(\zeta,\xi)\in (\HH^1)' \,:\, \beta\abs{\Omega}\mean{\zeta}_\Om + \abs{\Gamma}\mean{\xi}_\Ga = 0 \right\} \subset (\HH^1)'.
\end{align*}

\item \label{pre:D}
For $\beta>0$, we introduce the subspace
\begin{align*}
	\DDBI := \left\{ (\zeta,\xi) \in \HH^1 \,:\, \zeta\vert_\Ga = \beta\xi \;\; a.e.\;\text{on}\; \Ga \right\} \subset \HH^1.
\end{align*}
Endowed with the inner product
\begin{align*}
	\scp{(\phi,\psi)}{(\zeta,\xi)}_{\DDBI} 
		:= \scp{\phi}{\zeta}_{H^1(\Omega)} + \scp{\psi}{\xi}_{H^1(\Gamma)},
	\quad (\phi,\psi), (\zeta,\xi) \in \DDBI
\end{align*}
and its induced norm $\norm{\cdot}_{\DDBI}:=\scp{\cdot}{\cdot}_{\DDBI}^{1/2}\,$, the space $\DDBI$ is a Hilbert space. We further define the bilinear form
\begin{align*}
\biginn{(\phi,\psi)}{(\zeta,\xi)}_{\DDBI} := \scp{\phi}{\zeta}_{L^2(\Omega)} + \scp{\psi}{\xi}_{L^2(\Gamma)}
\end{align*}
for all pairs $(\phi,\psi), (\zeta,\xi) \in \LL^2$. By means of the Riesz representation theorem, this product can be extended to a duality pairing on $(\DDBI)'\times \DDBI$, which will also be denoted by $\inn{\cdot}{\cdot}_{\DDBI}$. This means that for all $(\phi,\psi)\in (\HH^1)'\subset (\DDBI)'$ and $(\zeta,\xi)\in \DDBI \subset \HH^1$,
\begin{align*}
	\biginn{(\phi,\psi)}{(\zeta,\xi)}_{\DDBI} = \inn{\phi}{\zeta}_{H^1(\Omega)} + \inn{\psi}{\xi}_{H^1(\Gamma)}
\end{align*}  
In particular, the spaces $\big(\DDBI,\LL^2,(\DDBI)'\big)$ form a Gelfand triplet, and the operator norm on $(\DDBI)'$ is given by
\begin{align*}
\norm{\phi}_{(\DDBI)'} = \sup\big\{\, \big|\inn{\phi}{\zeta}_{\DDBI}\big| \,:\, \zeta\in\DDBI \text{ with } \norm{\zeta}_{\DDBI} = 1 \big\}
\quad \text{ for all } \phi\in(\DDBI)'.
\end{align*}
We further point out that the mapping
\begin{align*}
	\mathfrak{I}:\VV^1\to\DDBI,\quad \rho \mapsto (\beta\rho,\rho)
\end{align*}
is an isomorphism.

\item \label{pre:S} Let $L>0$ and $(\phi,\psi)\in \HH_{\beta,0}^{-1}$ be arbitrary. 
According to \cite[Thm.~3.3]{knopf-liu}, there exists a unique weak solution $\SS(\phi,\psi):=(\SS_\Om(\phi,\psi), \SS_\Ga(\phi,\psi)) \in \HHBI$ to the elliptic system
\begin{align*}
	\left\{
	\begin{aligned}
	- \Lap \SS_\Om &= - \phi && \text{ in } \Omega, \\
	- \Lapg \SS_\Ga + \beta \deln \SS_\Om &= - \psi && \text{ on } \Gamma, \\
	L\deln \SS_\Om &= (\beta \SS_\Ga - \SS_\Om) && \text{ on } \Gamma.
	\end{aligned}
	\right.
\end{align*}
This means that
\begin{align}
\label{WF:S}
	\bigscp{\SS(\phi,\psi)}{(\zeta,\xi)}_{L,\beta} = - \biginn{(\phi,\psi)}{(\zeta,\xi)}_{\HH^1} = - \inn{\phi}{\zeta}_{H^1(\Om)} - \inn{\psi}{\xi}_{H^1(\Ga)}
\end{align}
for all $(\zeta,\xi)\in \HHBI$. 
As in \cite[Thm.~3.3 and Cor.~3.5]{knopf-liu}, we can thus define the solution operator
\begin{align*}
	\SS: \HH_{\beta,0}^{-1} \to \HH^1_{\beta,0},\quad (\phi,\psi) \mapsto \SS(\phi,\psi) = (\SS_\Omega(\phi,\psi),\SS_\Gamma(\phi,\psi))  
\end{align*}
as well as an inner product on the space $\HH_{\beta,0}^{-1}$ by 
\begin{align*}
	\bigscp{(\phi,\psi)}{(\zeta,\xi)}_{L,\beta,*} := \bigscp{\SS(\phi,\psi)}{\SS(\zeta,\xi)}_{L,\beta} \quad
	\text{for all}\;\; (\phi,\psi), (\zeta,\xi) \in \HH_{\beta,0}^{-1}
\end{align*}
along with its induced norm 
$$\norm{\cdot}_{L,\beta,*}:=\scp{\cdot}{\cdot}^{1/2}_{L,\beta,*}.$$ 
Because of $\HHBo\subset \HH_{\beta,0}^{-1}$, the product $\scp{\cdot}{\cdot}_{L,\beta,*}$ can also be used as an inner product on $\HHBo$.
Moreover, $\norm{\cdot}_{L,\beta,*}$ is also a norm on $\HHBo$ but $\HHBo$ is not complete with respect to this norm.
\item \label{pre:S:0}
Suppose that $\beta>0$. We define the space
\begin{align*}
	\DD_{\beta}^{-1} &:= \Big\{ (\phi,\psi) \in (\DDBI)' \;:\; \biginn{(\phi,\psi)}{(\beta,1)}_{\DDBI} = 0 \Big\}
	\subset (\DDBI)' .
\end{align*}
Proceeding as in \cite{knopf-liu}, we use Lax--Milgram theorem to show that for any $(\phi,\psi)\in\DD_{\beta}^{-1}$, there exists a unique weak solution $\SD(\phi,\psi) = (\SD_\Omega(\phi,\psi),\SD_\Gamma(\phi,\psi))\in \DDBI \cap \HH^1_{\beta,0}$ to the elliptic problem
\begin{align*}
	\left\{
	\begin{aligned}
		- \Lap \SD_\Om &= - \phi && \text{ in } \Omega, \\
		- \Lapg \SD_\Ga + \beta \deln \SD_\Om &= - \psi && \text{ on } \Gamma, \\
		\SD_\Om\vert_\Ga &= \beta \SD_\Ga && \text{ on } \Gamma.
	\end{aligned}
	\right.
\end{align*}	
This means that $\SD(\phi,\psi)$ satisfies the weak formulation
\begin{align}
	\label{WF:LIN}
	\biginn{ \SD(\phi,\psi) }{ (\zeta,\xi) }_{0,\beta} 
	= - \biginn{(\phi,\psi)}{(\zeta,\xi)}_{\DDBI}
\end{align}
for all test functions $(\zeta,\xi)\in \DDBI$.

Similar to \cite[Thm.~3.3 and Cor.~3.5]{knopf-liu}, we can thus define the solution operator
\begin{align*}
	\SD: \DD_{\beta}^{-1} \to \DDBI \cap \HH^1_{0,\beta},\quad (\phi,\psi) \mapsto \SD(\phi,\psi) = (\SD_\Omega(\phi,\psi),\SD_\Gamma(\phi,\psi))  
\end{align*}
as well as an inner product and its induced norm on the space $\DD_{\beta}^{-1}$ by
\begin{align*}
	\bigscp{(\phi,\psi)}{(\zeta,\xi)}_{0,\beta,*} &:= \bigscp{\SD(\phi,\psi)}{\SD(\zeta,\xi)}_{0,\beta}, \\
	\norm{(\phi,\psi)}_{0,\beta,*} &:=\scp{(\phi,\psi)}{(\phi,\psi)}_{0,\beta,*}^{1/2} 
\end{align*}
for all $(\phi,\psi),(\zeta,\xi)\in \DD_{\beta}^{-1}$.

Since $\HH_{0,\beta}^0 \subset \DD_\beta^{-1}$, the product $\scp{\cdot}{\cdot}_{0,\beta,*}$ can also be used as an inner product on $\HH_{0,\beta}^0$.
Moreover, $\norm{\cdot}_{0,\beta,*}$ is also a norm on $\HH_{0,\beta}^0$ but $\HH_{0,\beta}^0$ is not complete with respect to this norm.

\item \label{pre:N} For any integer $k$, we define the spaces
\begin{alignat*}{2}
	\accentset{\circ} H^k(\Om) &:= \big\{ \varphi \in H^k(\Om) \,:\, \meano{\varphi} = 0 \big\}, \quad
	&\accentset{\circ} H^k(\Ga) &:= \big\{ \varphi \in H^k(\Ga) \,:\, \meang{\varphi} = 0 \big\}, \\
	\accentset{\circ} H^{-1}(\Om) &:=  \big\{ \varphi \in H^1(\Om)' \,:\, \meano{\varphi} = 0 \big\}, \quad
	&\accentset{\circ} H^{-1}(\Ga) &:=  \big\{ \varphi \in H^1(\Ga)' \,:\, \meang{\varphi} = 0 \big\}.
\end{alignat*} 
Let now $\phi\in \accentset{\circ} H^{-1}(\Om)$ and $\psi\in \accentset{\circ} H^{-1}(\Ga)$ be arbitrary. Then the Lax--Milgram theorem guarantees the existence of a unique weak solutions $\NN_\Om(\phi) \in \accentset{\circ} H^1(\Om)$ and $\NN_\Ga(\psi) \in \accentset{\circ} H^1(\Ga)$ to the elliptic systems
\begin{align*}
	\left\{
	\begin{aligned}
		-\Lap \NN_\Om &= -\phi &&\text{in}\;\Om,\\
		\deln \NN_\Om &= 0	&&\text{on}\;\Ga
	\end{aligned}
	\right.
	\quad\text{and}\quad
	-\Lapg \NN_\Ga = -\psi \;\;\text{on}\;\Ga.
\end{align*}
This allows us to define the inner products
\begin{align*}
	\scp{\phi}{\zeta}_{\Om,*} &:= \scp{\grad\NN_\Om(\phi)}{\grad\NN_\Om(\zeta)}_{L^2(\Om)}, 
	\quad \phi,\zeta \in \accentset{\circ} H^{-1}(\Om)\\
	\scp{\psi}{\xi}_{\Ga,*} &:= \scp{\grad\NN_\Ga(\psi)}{\grad\NN_\Ga(\xi)}_{L^2(\Ga)},
	\quad \psi,\xi \in \accentset{\circ} H^{-1}(\Ga)
\end{align*}
on $\accentset{\circ} H^{-1}(\Om)$ and $\accentset{\circ} H^{-1}(\Ga)$, respectively.
Their induced norms 
\begin{align*}
	\norm{\cdot}_{\Om,*} := \scp{\cdot}{\cdot}_{\Om,*}^{1/2}, \quad
	\norm{\cdot}_{\Ga,*} := \scp{\cdot}{\cdot}_{\Ga,*}^{1/2}
\end{align*}
are equivalent to the standard norms on $\accentset{\circ} H^{-1}(\Om)$ and $\accentset{\circ} H^{-1}(\Ga)$, respectively. Because of $\accentset{\circ} H^0(\Om)\subset \accentset{\circ} H^{-1}(\Om)$ and $\accentset{\circ} H^0(\Ga)\subset \accentset{\circ} H^{-1}(\Ga)$, the products $\scp{\cdot}{\cdot}_{\Om,*}$ and $\scp{\cdot}{\cdot}_{\Ga,*}$ as well as the norms $\norm{\cdot}_{\Om,*}$ and $\norm{\cdot}_{\Ga,*}$ can also be used on $\accentset{\circ} H^0(\Om)$ and $\accentset{\circ} H^0(\Ga)$, respectively.

\end{enumerate}

\subsection{Important tools} 

We now present three fundamental lemmata which will be essential for the subsequent approach. The proofs of these lemmata, which are very technical in some parts, can be found in the Appendix.

\begin{lem} \label{LEM:JK:UNIQ}
	Suppose that \eqref{ass:kernel} is satisfied. Then the following holds:
	\begin{enumerate}[label=$(\mathrm{\alph*})$, ref = $\mathrm{\alph*}$]
		\item For all $\phi\in L^2(\Om)$, it holds that $J\ast\phi\in H^1(\Om)$ and there exists a constant $C_J>0$ depending only on $d$ and $J$ such that 
		\begin{align}
			\norm{J\ast\phi}_{H^1(\Om)} \le C_J\, \norm{\phi}_{L^2(\Om)}.
		\end{align}
		\item For all $\psi\in L^2(\Ga)$, it holds that $K\cg\psi\in H^1(\Ga)$ and there exists a constant $C_K>0$ depending only on $d$, $r$ and $K$ such that 
		\begin{align}
			\norm{K\cg\psi}_{H^1(\Ga)} \le C_K\, \norm{\psi}_{L^2(\Ga)}.
		\end{align}
	\end{enumerate}
\end{lem}

\medskip

\begin{rem} 
	If $p,q,\pp,\qp\in\R$ satisfy \eqref{ASS:PQ}, we have $p,q\ge 2$ and thus, $\pp,\qp\le 2$. Hence, Lemma~\ref{LEM:JK:UNIQ} particularly implies that $J\ast\phi\in W^{1,\pp}(\Om)$, $K\cg\psi\in W^{1,\qp}(\Ga)$ with
	\begin{alignat*}{3}
		\norm{J\ast\phi}_{W^{1,\pp}(\Om)}	
			&\le C \norm{J\ast\phi}_{H^1(\Om)}
			&&\le C \norm{\phi}_{L^2(\Om)}
			&&\le C \norm{\phi}_{L^p(\Om)},\\
		\norm{K\cg\psi}_{W^{1,\qp}(\Ga)}	
			&\le C \norm{K\cg\psi}_{H^1(\Ga)}
			&&\le C \norm{\psi}_{L^2(\Ga)}
			&&\le C \norm{\psi}_{L^q(\Ga)}
	\end{alignat*}
	for generic positive constants denoted by $C$ that may depend on 
	$J$, $K$, $d$ and $r$.
\end{rem}

\medskip	

\begin{lem} \label{LEM:JK}
	Suppose that \eqref{ass:kernel} holds and that $p,q,p',q'\in\R$ satisfy the condition \eqref{ASS:PQ}. Then for all sequences $(\phi_k)_{k\in\N} \subset L^p(\Om)$ and $(\psi_k)_{k\in\N}\subset L^q(\Ga)$ satisfying $\phi_k\wto \phi$ in $L^p(\Om)$ and $\psi_k \wto \psi$ in $L^q(\Ga)$ as $k\to\infty$, it holds that
	\begin{align*}
	(J\ast\phi_k) \to (J\ast\phi) \;\;\text{in}\; L^\pp(\Om)
	\quad\text{and}\quad
	(K\cg\psi_k) \to (K\cg\psi) \;\;\text{in}\; L^\qp(\Ga)
	\end{align*}
	as $k\to\infty$ along a non-relabeled subsequence.
\end{lem}

\medskip

\begin{lem}
	\label{LEM:NORMS}
	Let $\beta\neq 0$, $L> 0$ be arbitrary.
	Then, there exists a positive constant $C$ that may depend on $\beta$ but not on $L$ such that 
	\begin{subequations}
	\begin{alignat}{2}
	\label{LEM:EST:1}
	\norm{\bph}_{(\HH^1)'} 
	&\leq 
	C\left( 1 + \tfrac{1}{\sqrt{L}} \right) \norm{\bph}_{L,\beta, *} 
	&&\quad\text{for all $\bph\in\HH_{\beta,0}^{-1}$ if $L>0$,}\\[1ex]
	\label{LEM:EST:2}
	\norm{\bph}_{(\HH^1)'} &\leq C \norm{\SD(\bph)}_{0,\beta} = C \norm{\bph}_{0,\beta, *}
	&&\quad\text{for all $\bph\in\DD_\beta^{-1}$ if $L=0$.} 
	\end{alignat}
	\end{subequations}
\end{lem}

\subsection{Examples for admissible kernels, potentials and penalty functions}

\paragraph{Examples for admissible kernels $J$ and $K$.} 
Since the assumption \eqref{ass:kernel} is rather abstract, we now give concrete examples for admissible singular convolution kernels.

\begin{lem} \label{LEM:KERNEL} 
	\begin{enumerate}[label=$(\mathrm{\alph*})$, ref = $\mathrm{\alph*}$]
		\item Let $0<\omega < d-1$ be arbitrary and let $\rho\in C^1([0,\infty))$ be a positive function decaying sufficiently fast as $s\to\infty$ such that
		\begin{align*}
			s \mapsto |\rho^{(k)}(s)|\, s^{d-1-\omega} \in L^1((0,\infty)), \quad k\in\{0,1\}, 
		\end{align*}
		where $\rho^{(k)}$ denotes the derivative of $\rho$ of order $k$.
		Then the kernel $J$ defined by
		\begin{align*}
		J(x) := \rho(\abs{x})\, \abs{x}^{-\omega}
		\quad\text{for all}\;\; x\in\R^d\setminus\{0\}
		\end{align*}
		belongs to $W^{1,1}(\R^d)$ and satisfies the condition \eqref{ass:kernel} with
		\begin{align*}
			a_* = (2R)^{-\omega} \underset{\abs{x}\le 2R}{\min} \; \rho(\abs{x}).
		\end{align*}
		\item Suppose that $d=3$. Let $0<\gamma < d-2 = 1$ be arbitrary and suppose that $1<r<3/(\gamma+2)$.
		Moreover, let $\sigma\in C^2([0,\infty))$ be a positive function decaying sufficiently fast as $s\to\infty$ such that
		\begin{align*}
			s \mapsto |\sigma^{(k)}(s)|^r\, s^{2-r\gamma} \in L^1((0,\infty)), \quad k\in\{0,1,2\}, 
		\end{align*}
		where $\sigma^{(k)}$ denotes the derivative of $\sigma$ of order $k$.
		Then the kernel $K$ defined by
		\begin{align*}
			K(x) := \sigma(\abs{x})\, \abs{x}^{-\gamma}
			\quad\text{for all}\;\; x\in\R^3\setminus\{0\}
		\end{align*}
		belongs to $W^{2,r}(\R^3)$ and satisfies the condition \eqref{ass:kernel} with
		\begin{align*}
			a_\cg = (2R)^{-\gamma} \underset{\abs{x}\le 2R}{\min} \; \sigma(\abs{x}).
		\end{align*}
	\end{enumerate}
\end{lem}

\medskip

The proof of Lemma~\ref{LEM:KERNEL} can be found in the Appendix.

\medskip

\begin{rem} 
	\begin{enumerate}[label=$(\mathrm{\alph*})$, ref = $\mathrm{\alph*}$]
		\item Unfortunately, in Lemma~\ref{LEM:KERNEL}, the ``sharp'' cases $\omega = d-1$ in (a) and $\gamma = d-2$ in (b) cannot be included as the kernels would not have the desired regularity.
		Moreover, in the case $d=2$, we cannot allow singular potentials because $W^{2,r}(\R^2)$ is continuously embedded in $L^\infty(\R^2)$ as $r>1$.
		Of course, for both $d=2$ and $d=3$, non-singular potentials such as $K(x) = \sigma(\abs{x}) \abs{x}^\gamma$ with $\gamma\ge 0$ can be handled without any problem, provided that $\sigma(s)$ decays sufficiently fast as $s\to\infty$.
		\item In view of Remark~\ref{rem:JK}(a), let $R>0$ be the radius introduced in \eqref{ass:dom} and suppose that the functions $\rho$ and $\sigma$ in Lemma~\ref{LEM:KERNEL} satisfy 
		\begin{align*}
			\rho,\sigma\in C^\infty([0,\infty)), \quad 
			\rho=\sigma=1 \;\;\text{in}\;\; [0,2R], \quad
			\rho=\sigma=0 \;\;\text{in}\;\; [3R,\infty).
		\end{align*}
		Then, since $x-y\in \ov{B_{2R}(0)}$ for all $x,y\in\ov\Om$, the values of $(J*\phi)\vert_\Om$ and $(K\cg\psi)\vert_\Ga$ do not depend on the choice of $\rho$ and $\sigma$. 
		
		This means that, applying this trick, the Riesz potentials
		\begin{alignat*}{2}
			J(x) &:= c_\alpha \abs{x}^{\alpha-d} &&\quad\text{for}\;\; x\in\R^d\setminus\{0\}, \; d\in\{2,3\},\;\text{and}\;\; 1<\alpha<d,\\
			K(x) &:= c_\alpha \abs{x}^{\alpha-3} &&\quad\text{for}\;\; x\in\R^3\setminus\{0\}, \; d=3,\;\text{and}\;\; 2<\alpha<3,
		\end{alignat*}
		where $c_\alpha>0$ stands for a positive constant,
		can also be handeled. 
	\end{enumerate}
	\label{REM:KERNEL}
\end{rem}

\medskip

\paragraph{An example for admissible potentials $\Fw$ and $\Gw$.}

\begin{rem}
	\label{EX:POT}
	The smooth double well potential
	\begin{align*}
	W_\text{dw}(s)=\frac 1 4 (s^2-1)^2,\quad s\in\R
	\end{align*} 
	is one of the standard choices for the Cahn--Hilliard equation to model phase segregation processes.
	Let us fix arbitrary functions $f\in L^\infty(\Om)$ and $g\in L^\infty(\Ga)$ satisfying
	\begin{align*}
		f\ge 0 \;\; a.e. \; \text{in}\; \Om, \quad
		g\ge 0 \;\; a.e. \; \text{on}\; \Ga, \quad
		\norm{f}_{L^\infty(\Om)} < a_*, \quad
		\norm{g}_{L^\infty(\Ga)} < a_\cg.
	\end{align*} 
	Then the potentials $\Fw$ and $\Gw$ defined by
	\begin{align*}
		\Fw(x,s) := f(x)\, W_\text{dw}(s), \quad
		\Gw(z,s) := g(z)\, W_\text{dw}(s),
	\end{align*} 
	satisfy the assumptions \eqref{ass:pot:dw}, \eqref{ass:pot:conv} and \eqref{ass:pot:strong} with
	\begin{align*}
		p=q=4, \quad
		c_*:=a_* - \norm{f}_{L^\infty(\Om)}>0, \quad
		c_\cg:=a_\cg - \norm{g}_{L^\infty(\Ga)}>0.
	\end{align*}  
	In most situations, it should be reasonable to assume that $f$ is constant. However, as already discussed in the introduction, it might make sense to use a nonconstant factor $g$ to describe the influence of certain materials on the boundary more precisely.
	
	We also want to point out that singular potentials like the logarithmic potential
	\begin{align}
	\label{POT:LOG}
		W_\text{log}(s)= \frac\vartheta 2 \big((1+s)\ln(1+s) + (1-s)\ln(1-s)\big) 
			- \frac{\vartheta_c}{2} s^2, \quad s\in (-1,1)
	\end{align}
	(for some $0<\vartheta<\vartheta_c$) or the double-obstacle potential
	\begin{align}
	\label{POT:OBST}
	W_\text{obst}(s) = 
	\begin{cases}
	\vartheta(1-s^2), & s\in [-1,1],\\
	\infty, & \text{else},
	\end{cases}
	\end{align} 
	(for some $\vartheta>0$) are not allowed to be chosen instead of $W_\text{dw}$ as they do not satisfy the assumption \eqref{ass:pot:dw}. 
\end{rem}

\medskip

\paragraph{An example for an admissible penalty function $B$.}

\begin{rem}
	\label{EX:PEN}
	For any weight function $b\in L^\infty(\Ga)$, the penalty function $B$ defined (as in \eqref{DEF:PEN}) by
	\begin{align}
	\label{DEF:PEN:U}
		B(z,s) := b(z)\, s \quad \text{for almost all}\; z\in\Ga, \;\text{and all}\; s\in\R
	\end{align}
	obviously satisfies the assumptions \eqref{ass:pen} and \eqref{ass:pen:strong} with $\gamma_{B''}=\LIP=0$.
	In particular, the estimate \eqref{EST:PEN} can be verified with $\alpha_B$ arbitrarily small by means of Young's inequality for products.
	The corresponding penalty term in the free energy reads as
	\begin{align*}
	E_\text{pen}(\psi)=\intG b(z) \psi(z) \dS(z).
	\end{align*}
	As the free energy $E$ decreases along solutions of the system \eqref{CH:ROB} due to dissipation effects (cf.~\eqref{CH:ROB:DISS}), the term $E_\text{pen}$ is also expected to become small over the course of time.
	In this way, the penalization tries to enforce that $\psi$ attains values close to $1$ on the set $\Ga_1:=\{z\in\Ga : b(z)<0 \}$ and that $\psi$ attains values close to $-1$ on 
	$\Ga_{-1}:=\{z\in\Ga : b(z)>0 \}$. On $\Ga_0:=\{z\in\Ga : b(z)=0 \}$, the penalization behaves neutrally towards $\psi$. 
	This means that $\Ga_{\pm 1}$ attracts the material associated with $\psi=\pm 1$ and repels the material associated with $\psi=\mp 1$.
	Of course, since the total free energy becomes small and not only the penalty term, this behavior will also depend on effects caused by the potentials $\Fw$ and $\Gw$. In this regard, it might also be reasonable that at least the surface potential $\Gw=\Gw(z,s)$ is allowed to depend on the spatial variable $z\in\Ga$.
\end{rem}

\section{The gradient flow structure}
\label{SEC:GF}
In this section we show that the Cahn--Hilliard systems \eqref{CH:ROB}, \eqref{CH:DIR} and
(\eqref{CH:DEC:BULK},\eqref{CH:DEC:SURF}) can be interpreted as gradient flow equations of type $H^{-1}$ (both in the bulk and on the surface) with respect to suitable inner products. This structure will be a key ingredient in the well-posedness proof as it allows us to derive a priori bounds for the time-discrete approximate solution. As stated in \eqref{ass:dom}, the constants $\mo$, $\mg$, $\eps$ and $\delta$ are set to one.

\subsection{The gradient flow structure of the Robin model} \label{SUB:GF:1}
Suppose that the functions $\phi$, $\psi$, $\mu$ and $\nu$ are sufficiently regular. For convenience, we use the notation $\bph= (\phi,\psi)$. We claim that the Cahn--Hilliard system \eqref{CH:ROB} can be interpreted as the flow equation 
\begin{align}
	\label{GFE:ROB}
	\scp{\partial_t \bph}{\bet}_{L,\beta,*} = - \frac{\delta E}{\delta \bph}(\bph)[\bet]
	\quad\text{for all}\;\;
	\bet \in \HHBo \cap \LINF.
\end{align}
We notice that the only difference to the gradient flow equation of the local analogue \eqref{CH:INT} is that the local free energy $E^\text{GL}$ is replaced by the nonlocal free energy $E$. The inner product, however, remains the same. 

To prove this claim, let $\ov \bet=(\bz, \bxi)$ denote an arbitrary test function in $\LINF$.
Then, the first variation of the energy \eqref{DEF:EN:ALT}
at the point $\bph$ in the direction $\ov \bet$ reads as follows:
\begin{align}
\label{EQ:GF0}
\begin{aligned}
	\frac {\delta E}{\delta \bph}(\bph)[\ov \bet]
	& = - \intO\intO J(x-y) \phi(y)\, \ov \zeta(x) \dy \dx
		+ \intO F'(\cdot,\phi) \ov \zeta \dx
	\\ 
	& \quad -\intG\intG K(z-y)  \psi(y)\, \ov \xi(z) \dS(y) \dS(z)
		+ \intG 
	G'(\cdot,\psi) \ov \xi
	+ B'(\cdot,\psi) \ov \xi \dS.
\end{aligned}
\end{align}
To identify the gradient of $E$ at $\bph$
with respect to the inner product $(\cdot,\cdot)_{L,\beta,*}$  we look for the element $\nabla E (\bph)\in \HHBo$ such that $(\nabla E({\bph}),\bet)_{L,\beta,*}=\frac {\delta E}{\delta \bph}(\bph)[\bet]$ for every 
$\bet \in \HHBo\cap\LINF$. Denoting this element by $\brh=\nabla E({\bph})$ and using integration by parts along with \eqref{pre:S}, we have 
\begin{align}
\label{EQ:GF1}
\begin{aligned}
	\frac {\delta E}{\delta \bph}(\bph)[\bet]
	& = \intO \nabla \SS_\Om(\brh)  \cdot \nabla \SS_\Om(\bet)\dx 
	+
	\intG \nabla_\Ga  \SS_\Ga(\brh) \cdot \nabla_\Ga \SS_\Ga(\bet) \dS
	\\ & \qquad\quad 
	+ \frac 1 L \intG\big(\beta \SS_\Ga(\brh) - \SS_\Om(\brh) \big)\big(\beta \SS_\Ga (\bet) - \SS_\Om(\bet) \big)\dS 
	\\[1ex] &  
	= - \intO \SS_\Om(\brh) \zeta \dx - \intG \SS_\Ga(\brh) \xi \dS
\end{aligned}
\end{align}
for every $\bet=(\zeta,\xi) \in \HHBo \cap \LINF$.
We point out that \eqref{EQ:GF1} actually holds true for every test function $\ov\bet \in \LINF$ if we shift the integrands in the last line of \eqref{EQ:GF1} by suitable additive constants.
Indeed, we notice that for any arbitrary $\ov \bet =(\bz,\bxi) \in \LINF$,
the function $\bet_0:=(\zeta_0, \xi_0)$ defined by 
\begin{align}
\label{DEF:ET0}
	\zeta_0:= \bz + \beta c_0,
	\quad
	\xi_0:= \bxi + c_0
	\quad\text{with}\quad
	c_0 := -\frac{\beta\intO\bz \dx + \intG\bxi \dS}{\beta^2\abs{\Omega}+\abs{\Gamma}}
\end{align}
belongs to $\HHBo \cap \LINF$.
Choosing now the constant
\begin{align*}
c
:= - \frac {\beta |\Om| \mathcal F(\phi) +|\Ga| \mathcal G(\psi) }
{\beta^2 |\Om|+|\Ga|}
\end{align*}
with 
\begin{align}
\label{DEF:FG}
\begin{aligned}
\mathcal F(\phi) &:= \rO {- J * \phi + F'(\cdot,\phi)}, \\
\mathcal G(\psi) &:= \rG {-K \cg \psi + G'(\cdot,\psi)+ B'(\cdot,\psi) },
\end{aligned}
\end{align}
and defining the functions
\begin{alignat*}{3}
	\mu:= -\SS_\Om (\brh) + \beta c 
	\;\; \hbox{ in $\Om$,}
	\quad\text{and}\quad
	\nu:= -\SS_\Ga (\brh) + c 
	\;\; \hbox{ on $\Ga$}
\end{alignat*}
we conclude that
\begin{align}
\label{EQ:GF2}
\frac {\delta E}{\delta \bph}(\bph)[\ov\bet] = \intO \mu \ov\zeta \dx + \intG \nu \ov\xi \dS
\quad\text{for all}\; \ov\bet\in \LINF.
\end{align}
In particular, in view of \eqref{EQ:GF0}, we can use the fundamental lemma of calculus of variations to deduce the relations
\begin{align*}
\begin{aligned}
\mu& = - J * \phi + F'(\cdot,\phi) &&\quad \hbox{$a.e.$ in $\Om$,}
\\ 
\nu& = - K \cg \psi + G'(\cdot,\psi)+ B'(\cdot,\psi) &&\quad \hbox{$a.e.$ on $\Ga$.}
\end{aligned}
\end{align*}
On the other hand, for all $\bet\in\HHBo \cap\LINF$, we have
\begin{align*}
	\scp{\partial_t \bph}{\bet}_{L,\beta,*}
	&= 
	\intO \nabla \SS_\Om(\partial_t \bph)  \cdot \nabla \SS_\Om(\bet)\dx 
	+
	\intG \nabla_\Ga  \SS_\Ga(\partial_t \bph) \cdot \nabla_\Ga \SS_\Ga(\bet) \dS
	\\ 
	& \qquad
	+ \frac 1 L \intG \big(\beta \SS_\Ga (\partial_t \bph) - \SS_\Om(\partial_t \bph) \big)\big(\beta \SS_\Ga (\bet) - \SS_\Om(\bet) \big)\dS
	\\
	&=   \intO S^L_\Om(\partial_t \bph) \zeta\dx 
	 + \intG S^L_\Ga(\partial_t \bph) \xi \dS.
\end{align*}
Hence, the gradient flow equation 
\begin{align*}
	\scp{\partial_t \bph}{\bet}_{L,\beta,*} = - \frac {\delta E}{\delta \bph}(\bph)[\bet] = - (\nabla E({\bph}),\bet)_{L,\beta,*}
\end{align*}
can be expressed as
\begin{align} 
\label{EQ:GF3}
\begin{aligned}
	 \intO S_\Om^L(\partial_t \bph) \zeta\dx 
	+ \intG S_\Ga^L(\partial_t \bph) \xi \dS
	& = - \intO \mu \zeta\dx 
	- \intG \nu \xi \dS
\end{aligned}
\end{align}
for all $\bet\in\HHBo\cap\LINF$. Again, we can show that \eqref{EQ:GF3} holds true for all test functions $\ov\bet\in \LINF$ if the integrands on the left hand side are modified by appropriate additive constants.  
Proceeding as above, we choose the constant
\begin{align*}
	\ov c := -\frac{\beta\abs{\Om}\meano{\mu} + \abs{\Ga} \meang{\nu} }{\beta^2 \abs{\Om} + \abs{\Ga} }
\end{align*}
to conclude that
\begin{align} 
\label{EQ:GF4}
\begin{aligned}
\intO S_\Om^L(\partial_t \bph) \ov\zeta + \beta \ov c\, \ov\zeta\dx 
+ \intG S_\Ga^L(\partial_t \bph) \ov\xi + \ov c\, \ov\xi \dS
& = - \intO \mu \ov\zeta\dx 
- \intG \nu \ov\xi \dS
\end{aligned}
\end{align}
holds for all test functions $\ov\bet=(\ov\zeta,\ov\xi)\in \LINF$.
Recalling \eqref{pre:S}, it follows from the fundamental theorem of calculus of variations that $(\mu,\nu)$ satisfies the system
\begin{align*}
\begin{cases}
- \Lap \mu = - \partial_t \phi & \text{ in } \Omega, \\
- \Lapg \nu + \beta \deln \mu = - \partial_t \psi & \text{ on } \Gamma, \\
L\deln \mu = \beta \nu - \mu & \text{ on } \Gamma.
\end{cases}
\end{align*}
Thus, combining \eqref{EQ:GF2} and \eqref{EQ:GF4} we conclude \eqref{GFE:ROB}.

\subsection{The gradient flow structure of the limit models}

Provided that the functions $\phi$, $\psi$, $\mu$ and $\nu$ are sufficiently regular, we can argue similarly to derive the gradient flow equation for the Dirichlet model \eqref{CH:DIR} and the decoupled model (\eqref{CH:DEC:BULK},\eqref{CH:DEC:SURF}). 
\begin{itemize}
\item The gradient flow equation corresponding to the Dirichlet model \eqref{CH:DIR} reads as 
\begin{align}
\label{GFE:DIR}
\scp{\partial_t \bph}{\bet}_{0,\beta,*} = - \frac{\delta E}{\delta \bph}(\bph)[\bet]
\quad\text{for all}\;\;
\bet \in \HHBo \cap \LINF,
\end{align}
where $\bph=(\phi,\psi)$.  
Here, the only difference to the gradient flow equation of the local analogue \eqref{CH:GMS} (see \cite{GK}) is that the local free energy $E^\text{GL}$ is replaced by the nonlocal free energy 
$E$, while the inner product remains the same.

\item The gradient flow equation corresponding to the system \eqref{CH:DEC:BULK} reads as 
\begin{align}
\label{GFE:NEUM:BULK}
\scp{\partial_t \phi}{\zeta}_{\Om,*} = - \frac{\delta E_\text{bulk}}{\delta \phi}(\phi)[\zeta]
\quad\text{for all}\;\;
\zeta \in \Lx\infty \;\text{with}\; \meano{\zeta}=0,
\end{align}
whereas the gradient flow equation corresponding to the system \eqref{CH:DEC:SURF} reads as
\pagebreak[2]
\begin{align}
\label{GFE:NEUM:SURF}
\scp{{\partial_t \psi}}{\xi}_{\Ga,*} = - \frac{\delta (E_\text{surf}+ E_\text{pen})}{\delta \psi}(\psi)[\xi]
\quad\text{for all}\;\;
\xi \in \LxG\infty \;\text{with}\; \meang{\xi}=0.
\end{align}
In contrast to the local analogue \eqref{CH:LW}, the gradient flow equations in the bulk and on the surface are fully decoupled. However, if we sum \eqref{GFE:NEUM:BULK} and \eqref{GFE:NEUM:SURF}, we obtain the gradient flow equation for the corresponding system \eqref{CH:LW} (see, e.g., \cite{GK,KL}) only with $E^\text{GL}$ replaced by $E$. 
\end{itemize}

\section{Weak and strong well-posedness of the Robin model}

\subsection{Notion of a weak solution to the Robin model}

We first introduce the notion of a weak solution to the system \eqref{CH:ROB}. We point out that, as stated in \eqref{ass:dom}, the constants $\mo$, $\mg$, $\eps$ and $\delta$ are set to one since they have no impact on the mathematical analysis.

\begin{defn}[Definition of a weak solution to the system \eqref{CH:ROB}] \label{DEF:WP:ROB}
	Let $T,L > 0$, $m \in \R$, $\beta\neq 0$ and $(\phi_0,\psi_0)\in \HHBMo$ be arbitrary and suppose that the conditions \eqref{ass:dom}--\eqref{ass:pen} hold. 
	The quadruplet $(\phi,\psi,\mu,\nu)$ is called a weak solution of the system \eqref{CH:ROB} if the following holds:
	\begin{enumerate}[label=$(\mathrm{\roman*})$, ref = $\mathrm{\roman*}$]
		\item The functions $(\phi,\psi,\mu,\nu)$ have the following regularity
		\begin{align}
		\label{REG:ROB}
		\left\{
		\begin{aligned}
		\phi & \in C^{0,\frac 12}([0,T]; \Hx1 ') \cap H^1(0,T;H^1(\Omega)') \cap \L\infty {\Lx p}
		, \\
		\psi & \in C^{0,\frac 12}([0,T];\HxG1 ') \cap H^1(0,T;H^1(\Gamma)') \cap \L\infty {\LxG q} 
		, \\
		\mu &\in L^2\big(0,T;H^1(\Omega)\big),\\
		\nu &\in L^2\big(0,T;H^1(\Gamma)\big)
		\end{aligned}
		\right.
		\end{align}
		where $p$ and $q$ are the exponents from \eqref{ass:pot:conv},
		and it holds that $(\phi(t),\psi(t))\in \HHBMo$ for almost all $t\in[0,T]$. 
		\item
		The weak formulation 
		\begin{subequations}
			\label{WF:ROB}
			\begin{align}
			\label{WF:ROB:1}
			&\inn{\delt \phi}{ \theta}_{H^1(\Omega)} = - \intO \grad\mu \cdot \grad \theta \dx 
			-  \intG \tfrac 1 L(\beta\nu-\mu)  \theta \dS, \\
			\label{WF:ROB:2}
			&\inn{\delt \psi}{ \sigma}_{H^1(\Gamma)} = - \intG \gradg\nu \cdot \gradg \sigma \dS 
			+ \intG \tfrac 1 L (\beta\nu-\mu) \beta \sigma \dS, \\
			\label{WF:ROB:3}
			&\intO \mu \zeta \dx = \intO - (J * \phi)\zeta + F'(\cdot,\phi)\zeta \dx, \\
			\label{WF:ROB:4}
			&\intG \nu \xi \dS = \intG - (K \cg \psi)\xi + G'(\cdot,\psi)\xi 
			+ B'(\cdot,\psi)\xi  \dS	 
			\end{align}
		\end{subequations}
		is satisfied almost everywhere in $[0,T]$ for all test functions $\theta\in H^1(\Omega)$, $\sigma\in H^1(\Gamma)$, $\zeta\in L^\infty(\Omega)$ and $\xi\in L^\infty(\Gamma)$. Moreover, the initial conditions $\phi\vert_{t=0}=\phi_0$ and $\psi\vert_{t=0}=\psi_0$ are satisfied almost everywhere in $\Omega$ and on $\Gamma$, respectively.
		\item The energy inequality 
		\begin{align}
		\notag
		&E\big(\phi(t),\psi(t)\big) 
		+\frac 1 2 \int_0^t 
		\norm{\grad\mu(s)}_{L^2(\Omega)}^2 
		+ \norm{\gradg\nu(s)}_{L^2(\Gamma)}^2
		+ \tfrac 1 L \norm{\beta\nu(s) - \mu(s)}_{L^2(\Gamma)}^2
		\ds \\
		&\qquad \le E(\phi_0,\psi_0)
		\label{energy:ROB}
		\end{align}
		is satisfied for all $t\in[0,T]$.
	\end{enumerate}
\end{defn}

\subsection{Weak well-posedness}

The weak well-posedness result reads as follows.

\begin{thm}[Weak well-posedness for the system \eqref{CH:ROB}] \label{THM:WP:ROB}
	Let $T,L > 0$, $m \in \R$ and $\beta\neq 0$ be arbitrary and suppose that the conditions \eqref{ass:dom}--\eqref{ass:pen} hold. 
	Then, for any initial datum $(\phi_0,\psi_0)\in \HHBMo$ satisfying $F(\cdot,\phi_0)\in L^1(\Omega)$, $G(\cdot,\psi_0), B(\cdot,\psi_0)\in L^1(\Gamma)$, there exists a unique weak solution of the system \eqref{CH:ROB} in the sense of Definition~\ref{DEF:WP:ROB}. 
\end{thm}

\begin{rem} 
\label{REM:weakLpq}
Since $\phi(t) \in L^p(\Om)$ and $\psi(t) \in L^q(\Ga)$ for almost all $t\in[0,T]$, it holds that $F'(\cdot,\phi(t))\in L^\pp(\Om)$ and $G'(\cdot,\phi(t)), B'(\cdot,\phi(t))\in L^\qp(\Ga)$ for almost all $t\in[0,T]$ due to the growth conditions in \eqref{ass:growth}.
Moreover, we have $(J\ast\phi(t))\in L^\pp(\Om)$ and $(K\cg\psi(t))\in L^\qp(\Ga)$ for almost all $t\in[0,T]$ according to Remark~\ref{rem:JK}. 
Hence, by a density argument, the weak formulations \eqref{WF:ROB:3} and \eqref{WF:ROB:4} remain valid for all test functions $\zeta\in L^p(\Om)$ and $\xi\in L^q(\Ga)$.
\end{rem}
	
\begin{proof}[Proof of Theorem~\ref{THM:WP:ROB}]
The proof is divided into seven steps. Throughout this proof, the symbol $C$ will denote generic positive constants independent of $N$, $n$ and $\tau$ that may change their value from line to line.

\subparagraph{Step 1: Implicit time discretization.} 
We devise a particular minimization
movements scheme based on the gradient flow structure discussed in Section~\ref{SEC:GF}. To this end, let us fix an arbitrary $N \in \mathbb{N}$ and let
$\tau := T / N$ denote our time-step size. 
We now define a time-discrete approximate solution $\bph^n=(\phi^n,\psi^n) \in \HHBMo$, $n=1,...,N$ by the following recursion: 
The zeroth iterate is given by the initial data, i.e., $\bph^0=(\phi^0,\psi^0):=(\phi_0,\psi_0)$. Assuming that the $n$-th iterate $\bphn$ is already constructed, we define $\bphnp$ as a minimiser of the functional 
\begin{align}
\IN (\bph) :=\frac 1 {2\tau} \norm{\bph - \bphn}_{L,\beta,*}^2 
+ E(\bph)
\label{Jn}
\end{align}
over the set $\HHBMo$, i.e.,
\begin{align}
\bphnp \in {\rm arg\min} \, \IN(\bph) 
\quad \hbox{for every $\bph\in \HHBMo$}.
\end{align}
The existence of such a minimiser will be addressed in Step 2. The idea behind this construction is the following:
Formally, the minimality of $\bphnp$ entails that 
\begin{align*}
0&=\frac {\delta \IN }{\delta \bph}(\bphnp)[\bet] =
\Big <\frac{ \bphnp - \bphn}\tau , \bet \Big >_{L,\beta,*}
+ \frac{\delta E}{\delta \bph} (\bphnp)[\bet]
\\[1ex]
&
\begin{aligned}
&=\Big <\frac{ \bphnp - \bphn}\tau , \bet \Big >_{L,\beta,*}
+ \intO -(J* \phi^{n+1})\zeta
+ F'(\cdot,\phi^{n+1}) \zeta \dx
\\ 
& \qquad +\intG -(K \cg \psi^{n+1}) \xi
+ G'(\cdot,\psi^{n+1}) \xi
+ B'(\cdot,\psi^{n+1})  \xi \dS
\end{aligned}
\end{align*}
for every test function $\bet=(\zeta,\xi)\in \HHBo\cap\LL^\infty$.
Thus, $\{\bphn\}_{n=1,...,N}$ can be interpreted as a time-discrete approximate solution of the gradient flow equation \eqref{GFE:ROB}. 
However, depending on the growth conditions on the potentials, it is possible that the functional $\IN$ is not Gâteaux differentiable at the point $\bphnp$ as $\IN$ might attain the value $+\infty$ in every neighbourhood of $\bphnp$. However, exploiting the convexity of $F$ and $G$, we can proceed as in \cite{garckeelas} to rigorously show that the Euler--Lagrange equation
\begin{align}
\label{EQ:EL}
\begin{aligned}
0& =\Big <\frac{ \bphnp - \bphn}\tau , \bet \Big >_{L,\beta,*}
+ \intO -(J* \phi^{n+1})\zeta
+ F'(\cdot,\phi^{n+1}) \zeta \dx
\\ 
& \quad +\intG -(K \cg \psi^{n+1}) \xi
+ G'(\cdot,\psi^{n+1}) \xi
+ B'(\cdot,\psi^{n+1})  \xi \dS
\end{aligned}
\end{align}
holds true for all $\bet=(\zeta,\xi)\in \HHBo\cap\LL^\infty$ nevertheless.
Now, recalling the solution operator defined in \eqref{pre:S}, we set 
\begin{align}
\label{DEF:MUNU:0}
(\accentset{\circ} \mu^{n+1} ,\accentset{\circ} \nu^{n+1}):= \SS\Big(\frac{ \bphnp - \bphn}\tau\Big) \in \HHBI.
\end{align}
Then for every $n = 0,...,N-1$ and any $\bet \in \HHBo\cap \LL^\infty$, a straightforward computation reveals that 
\begin{align}
\label{EQ:MUNU:0}
\begin{aligned}
&\intO \accentset{\circ} \mu^{n+1} \zeta \dx + 
\intG \accentset{\circ} \nu^{n+1} \xi \dS \\
&\quad =
\intO -(J* \phi^{n+1})\zeta 
+ F'(\cdot,\phi^{n+1}) \zeta \dx
\\ 
& \qquad 
+ \intG -(K \cg \psi^{n+1}) \xi
+ G'(\cdot,\psi^{n+1}) \xi
+ B'(\cdot,\psi^{n+1})  \xi \dS.
\end{aligned}
\end{align}
Arguing as in Section~\ref{SEC:GF}, we now want to
replace 
$(\accentset{\circ} \mu^{n+1} ,\accentset{\circ} \nu^{n+1})$
by functions $(\mu^{n+1} ,\nu^{n+1})$
such that \eqref{EQ:MUNU:0} holds true for all test functions $\bet=(\zeta,\xi)\in \HH^0\cap\LL^\infty = \LL^\infty$ (i.e., there is no restriction to the mean values of $\zeta$ and $\xi$ anymore). 
Setting
\begin{align}
\label{DEF:MUNU}
\mu^{n+1}:= \accentset{\circ} \mu^{n+1} + \beta c^{n+1},
\quad 
\nu^{n+1}:= \accentset{\circ} \nu^{n+1} + c^{n+1},
\end{align}
with 
\begin{align*}
c^{n+1}
:= - \frac {\beta |\Om| \mathcal F(\phi^{n+1}) +|\Ga| \mathcal G(\psi^{n+1}) }
{\beta^2 |\Om|+|\Ga|},
\end{align*}
where $\mathcal F$ and $\mathcal G$ are defined as in \eqref{DEF:FG},
we 
conclude that the functions $\bph^n$, $\bph^{n+1}$ and $(\mu^{n+1},\nu^{n+1})$
satisfy the equations
\begin{subequations}
	\label{disc}
\begin{align}
& \label{disc:1}
{\intO \Big (\frac{ \phi^{n+1} - \phi^{n} }\tau \Big)\theta \dx
	=- \intO \nabla \mu^{n+1} \cdot \nabla \theta \dx
	-\intG \tfrac 1L (\beta \nu^{n+1}-\mu^{n+1}) \theta \dS,
}
\\ 
& \label{disc:2}
\intG \Big (\frac{ \psi^{n+1} - \psi^{n} }\tau \Big) \sigma \dS
= - \intG \nabla_\Ga \nu^{n+1}\cdot \nabla_\Ga \sigma \dS
+ \intG \tfrac 1L (\beta \nu^{n+1}-\mu^{n+1}) \beta \sigma 	\dS,
\\ 
& \label{disc:3}
\intO \mu^{n+1} \zeta \dx  =
\intO -(J* \phi^{n+1})\zeta 
+ F'(\cdot,\phi^{n+1}) \zeta \dx,
\\ 
& \label{disc:4}
\intG \nu^{n+1} \xi \dS =
\intG -(K \cg \psi^{n+1}) \xi
+ G'(\cdot,\psi^{n+1}) \xi
+ B'(\cdot,\psi^{n+1})  \xi \dS
\end{align}
\end{subequations}
for all test functions $\theta \in H^1(\Om)$, $\sigma \in H^1(\Ga)$ and $\bet=(\zeta,\xi) \in  \LL^\infty$.
This system can be interpreted as an implicit time discretization of the weak formulation \eqref{WF:ROB}
and hence, the quadruplet $(\phi^n,\psi^n,\mu^n,\nu^n)$ is a time-discrete approximate solution of the system \eqref{CH:ROB}.
In order to show that this approximate solution converges to a weak solution of \eqref{CH:ROB} in some suitable sense, it must first be extended onto the whole time interval.
To this end, for every $n=1,...,N$, we define
the \textit{piecewise constant extension} by 
\begin{align*}
\big(\phi_N,\psi_N,\mu_N,\nu_N\big)(\cdot,t)&:=\big(\phi_N^n,\psi_N^n,\mu_N^n,\nu_N^n\big)(\cdot,t):= \big(\phi^n,\psi^n,\mu^n,\nu^n\big)
\end{align*}
for $t\in \big((n-1)\tau,n\tau\big]$, and the \textit{piecewise linear extension} by
\begin{align*}
\big(\ov \phi_N,\ov \psi_N,\ov \mu_N,\ov\nu_N\big)(\cdot,t)&:=
\alpha \big(\phi_N^n,\psi_N^n,\mu_N^n,\nu_N^n\big)
+(1-\alpha)\big(\phi_N^{n-1},\psi_N^{n-1},\mu_N^{n-1},\nu_N^{n-1}\big)
\end{align*}
for, every $\alpha \in [0,1]$, and $t= \alpha n t + (1-\alpha)(n-1)\tau$, respectively. 	

\subparagraph{Step 2: Existence of a minimiser.}
We now show that, for every $n\in\mathbb{N}$, the functional $\IN$ defined in \eqref{Jn} admits a minimiser in the set $\HHBMo$.
For the proof we apply the direct method of calculus of variations.
It is worth mentioning that the aforementioned discussion concerning the relation between the energies \eqref{DEF:EN} and \eqref{DEF:EN:ALT}
(as well as between $F,G$ and $\Fw,\Gw$) allows us to switch between these two frameworks at our convenience.
Indeed, we will first make use of the representation \eqref{DEF:EN:ALT}
to show that $\IN$ is bounded from below, while the rest 
of the proof is carried out using the depiction \eqref{DEF:EN}.

First, using the representation \eqref{DEF:EN}, 
the nonnegativity of $\Fw$, \eqref{ass:growth:2b} and \eqref{ass:pen:1},
we deduce that 
\begin{align*}
\IN(\bph) 
&\geq 
\intG \Gw (\cdot,\psi)\dS
+ \intG B (\cdot,\psi)\dS
\\
&\geq 
\alpha_{\Gw} \intG |\psi|^q \dS
- \delta_{\Gw}|\Ga|
- \intG |B (\cdot,\psi)|\dS
\\ 
& \geq
(\alpha_{\Gw}- \alpha_B) \intG |\psi|^q \dS
- |\Ga|(\delta_{\Gw}+\gamma_B)
\geq
-  |\Ga|(\delta_{\Gw}+\gamma_B)
\end{align*}
for any $\bph\in\HHBMo$. This proves that the functional $\IN$ is bounded from below and thus, 
\begin{align*}
	\INF:=\inf_{\bph\in\HHBMo} \IN(\bph)
\end{align*}
exists. This allows us to choose a minimising sequence $(\bph_k)_{k\in\N}=(\phi_k,\psi_k)_{k\in\N}\subset \HHBMo$ 
which satisfies
\begin{align*}
\lim_{k\to\infty} \IN(\bph_k) = \INF,
\quad\text{and}\quad
 \IN(\bph_k) \leq \INF +1 \quad \hbox{for every $k\in \mathbb{N}$}.
\end{align*}
Furthermore, invoking the growth assumption of the potentials $\Fw$ and $\Gw$ pointed out by \eqref{ass:growth:2a}--\eqref{ass:growth:2b}, we infer that 
\begin{align*}
\alpha_{\Fw} \intO \abs{\phi_k}^{p} \dx
- |\Om|\delta_{\Fw} 
+ (\alpha_{\Gw}-\alpha_B) \intG \abs{\psi_k}^{q} \dS
- |\Ga|(\delta_{\Gw} + \gamma_B)
&\leq 
E(\bph_k)
\leq \INF +1
\end{align*}
which leads to 
\begin{align*}
\alpha_{\Fw}\intO \abs{\phi_k}^{p} \dx
+  (\alpha_{\Gw}-\alpha_B)\intG \abs{\psi_k}^{q} \dS
\leq 
 \INF  +1+ \big(|\Om| \delta_{\Fw}+ |\Ga|(\delta_{\Gw}+\gamma_B)\big).
\end{align*} 
Hence, since $\alpha_B < \alpha_{\Gw}$, there exists a non-relabeled subsequence of $(\bph_k)_{k\in\N}$ 
and a limit $\ov \bph=(\bphi,\bpsi)\in L^p(\Om) \times L^q(\Ga) \subset \HH^0$
such that
\begin{align*}
\phi_k \wto \bphi \quad \hbox{in $L^p(\Om)$},
\quad 
\psi_k \wto \bpsi \quad \hbox{in $L^q(\Ga)$},
\end{align*}
as $k\to\infty$. Due to these convergence properties it is straightforward to check that $\ov\bph \in\HHBMo$.
It remains to show that $\ov\bph$ is indeed a minimizer of the functional $\IN$.
To this end, we use the representation \eqref{DEF:EN} of the energy $E$.
Recalling Lemma~\ref{LEM:JK}, that $F$ and $G$ are convex and that  
\begin{align*}
\intG B(\cdot,\bpsi) \dS &\le \underset{k\to\infty}{\lim\inf} \intG B(\cdot,\psi_k) \dS
\end{align*}
according to \eqref{ass:pen:2}, we conclude that
\begin{align*}
\IN(\ov\bph)\leq \liminf_{k\to\infty} \IN(\bph_k)= \INF.
\end{align*}
By the definition of $\INF$ this proves that $\ov\bph$ is a minimizer of $\IN$ on the domain $\HHBMo$.

\subparagraph{Step 3: Uniform estimates.}
In order to prove convergence of the piecewise constant extension, we now aim to establish uniform bounds on the functions $(\phi^N,\psi^N,\mu^N,\nu^N)$.
We claim that there exists a positive constant $C$
independent of $N,n,\tau$ such that
\begin{align}
\label{unif_est}
\begin{aligned}
&\norm{\phi_N}_{\L\infty {\Lx{p}}}
+ \norm{\psi_N}_{\L\infty {\LxG{q}}} 
\\
&\quad + \norm{\mu_N}_{\L2 {\Hx1}}
+ \norm{\nu_N}_{\L2 {\HxG1}}
\leq C.
\end{aligned}
\end{align}
As $\bphnp$ is a minimizer of $\IN$ over $\HHBMo$, we infer that
\begin{align*}
\IN(\bphnp)=\tfrac 1 {2\tau} \norm{\bphnp - \bphn}_{L,\beta,*}^2 
+ E(\bphnp)
\leq
\IN(\bphn) =E(\bphn)
\end{align*}
for all $n\in\{0,...,N-1\}$. Moreover, by induction and using \eqref{ass:pot:conv}, we infer that
\begin{align*}
\alpha_{F}\intO \abs{\phi^{n+1}}^{p} 
+ (\alpha_{G}-\alpha_B)\intG \abs{\psi^{n+1}}^{q} 
\leq 
E(\bphnp) + C \leq 
E(\bph^0) + C
\end{align*}
so that, owing to the boundedness of $E(\bph^0)$, we obtain the uniform bound
\begin{align}
\label{EST:PHIPSI}
\|{\phi^{n+1}}\|_{L^p(\Om)}^p
+ \|{\psi^{n+1}}\|_{L^q(\Ga)}^q
\leq C.
\end{align}
Next, fixing an arbitrary nontrivial function $\rho\in C^\infty(\Om)$ as a test function in \eqref{disc:3}, we get 
\begin{align*}
\intO \mu^{n+1} \rho \dx  =
\intO -(J* \phi^{n+1})\rho 
+ F'(\cdot,\phi^{n+1}) \rho \dx.
\end{align*}
Now, using \eqref{EST:J}, \eqref{EST:PHIPSI} and the polynomial growth of the potential $F$ as postulated in
\eqref{ass:pot:conv},
we obtain
\begin{align*}
&\left|
\intO - (J* \phi^{n+1})\rho 
+ F'(\cdot,\phi^{n+1}) \rho \dx
\right|
\\
&\quad \leq
\left[ a^* \norm{\phi^{n+1}}_{L^1(\Om)}
+\gamma_{F'} \left(|\Om|+ \norm{\phi^{n+1}}^{p-1}_{L^{p-1}(\Om)}\right)\right] \norm{\rho}_\infty \le C\norm{\rho}_\infty .
\end{align*}
Hence, there exists a constant $C(\rho)$ independent of $N$ such that
\begin{align}
\label{genpoi}
\left|
\intO \mu^{n+1} \rho \dx 
\right|
\leq 
C(\rho).
\end{align}
Arguing as in \cite{GK}, we define  
\begin{align*}
{\cal M}_{\rho}
:= 
\left \{
v \in H^1(\Om) \, : \, \left|\intO v \rho \dx \right| \leq C(\rho)
\right \},
\quad
C_0:= \frac {C(\rho)}{|\intO v \rho\dx |}.
\end{align*}
Note that ${\cal M}_\rho $ is a nonempty, closed and convex subset of $H^1(\Om)$.
The definition of $C_0$ entails that
\begin{align*}
|\xi| \leq \frac {|\intO \xi \rho \dx|}{\intO \rho \dx}\leq C_0
\quad 
\hbox{ for all } \; \xi \in \R \;\hbox{ such that }\; \xi \chi_{\Om}\in {\cal M}_\rho,
\end{align*}
where $\chi_{\Om}$ denotes the characteristic function of the set $\Om$. 
This allows us to apply the generalized Poincar\'e inequality \cite[p.~242]{alt} which yields 
\begin{align*}
	\norm{v}_{L^2(\Om)}
	\leq C(1+\norm{\nabla v}_{L^2(\Om)}) \quad\text{for all $v\in \mathcal M_\rho$}.
\end{align*} 
In particular, since $\mu^{n+1}\in {\cal M}_{\rho}$, we conclude that 
\begin{align}
\label{EST:MU}
\norm{\mu^{n+1}}_{L^2(\Om)}
\leq C(1+\norm{\nabla \mu^{n+1}}_{L^2(\Om)})
\quad 
\hbox{for all $n\in\{0,...,N-1\}$}.
\end{align}
It thus remains to establish a uniform bound on $\grad \mu_N$. We first recall from the definition of the piecewise constant extension that for 
arbitrary $n \in \{1,...,N-1\}$, $t = n \tau$ and 
for any $s \in (t-\tau,t]$, it holds that
\begin{align*}
\big(\phi_N,\psi_N,\mu_N,\nu_N\big)(s)=\big(\phi_N,\psi_N,\mu_N,\nu_N\big)(t)=\big(\phi_N^n,\psi_N^n,\mu_N^n,\nu_N^n\big).
\end{align*}
Using the above estimates as well as the definition of $\mu_N$
and $\nu_N$ we infer that
\begin{align*}
& E(\bph_N(t)) + \frac 12 \int_{t-\tau}^t \norm{\nabla \mu_N(s)}^2_{\Lx2}
+\norm{\gradg \nu_N(s)}^2_{\LxG2} 
+	\frac 1L \norm{\beta \nu_N(s) - \mu_N(s)}^2_{\LxG2} 
\ds
\\ & \quad
=
E(\bph_N(t)) + \frac 1{2\tau^2} \int_{t-\tau}^t \norm{\bph_N(s)-\bph_N(s-\tau)}^2_{L,\beta,*}\ds
\\ & \quad
=
E(\bph_N(t)) + \frac 1{2\tau^2}  \int_{t-\tau}^t \norm{\bph_N(t)-\bph_N(t-\tau)}^2_{L,\beta,*}\ds
\\ & \quad
=
E(\bph_N(t)) + \frac 1{2\tau}  \norm{\bph_N(t)-\bph_N(t-\tau)}^2_{L,\beta,*}
\leq
E(\bph_N(t-\tau)).
\end{align*}
Arguing by induction we conclude that
\begin{align}
\label{EST:ENG}
\begin{aligned}
&E(\bph_N(t)) + \frac 12 \int_{0}^t \norm{\nabla \mu_N(s)}^2_{\Lx2}
+\norm{\gradg \nu_N(s)}^2_{\LxG2} 
+	\frac 1L \norm{\beta \nu_N(s) - \mu_N(s)}^2_{\LxG2} 
\ds
\\
&\quad\leq 
E (\bph^0).
\end{aligned}
\end{align}
Therefore, choosing $t= T =N \tau$ leads to
\begin{align*}
\norm{\nabla \mu_N}^2_{\L2 {\Lx2}}
+\norm{\gradg \nu_N}^2_{\L2 {\LxG2}}
\leq 2 (E(\bph^0) + C)
\leq C
\end{align*}
and in combination with \eqref{EST:MU} we conclude the uniform bound
\begin{align}
	\label{BND:MU}
	\norm{\mu_N}_{\L2 {\Hx1}}\leq C.
\end{align}
Next, testing \eqref{disc:4} with $\xi \equiv 1$ gives
\begin{align*}
\intG \nu^{n+1}(t)  \dS =
\intG - (K \cg \psi^{n+1}) (t)
+ G'(\cdot,\psi^{n+1}(t)) 
+ B'(\cdot,\psi^{n+1}(t))  \dS
\end{align*}
for almost all $t \in [0,T]$.
Hence, using the assumptions postulated in \eqref{ass:pot:conv} and \eqref{ass:pen} as well as the uniform bound \eqref{EST:PHIPSI}, we infer that
\begin{align*}
&\left|\intG \nu^{n+1}\dS \right|
\leq a^\cg \norm{\psi^{n+1}}_{\Lx{1}}
+ \gamma_{G'}\norm{\psi^{n+1}}_{\Lx{q-1}}^{q-1}
+ \alpha_B \norm{\psi^{n+1}}_{\Lx{q}}^{q}
+  |\Ga|(\gamma_{G'}+\gamma_B )
\leq C
\end{align*}
which means that $\rG {\nu_N}$ is uniformly bounded in $L^\infty(0,T)$.
Thus, we can use Poincar\'e's inequality to conclude the uniform bound
\begin{align}
	\label{BND:NU}
	\norm{\nu_N}_{\L2 {\HxG1}} \le C
\end{align}
which proves the claim.

\subparagraph{Step 4: H\"older estimates in time.}
\begin{subequations}
Using interpolation type arguments, we now show that the piecewise linear extension is H\"older continuous in time. We claim that for all $t,s\in[0,T]$,
	\begin{alignat}{4}
	\label{EST:HLD}
	\norm{\ov\phi_N(t)-\ov\phi_N(s)}_{H^1(\Omega)'} + \norm{\ov\psi_N(t)-\ov\psi_N(s)}_{H^1(\Gamma)'} 
	&\le C \abs{t-s}^{\frac 1 2}, \\
	\label{EST:DIFF}
	\norm{\phi_N(t)-\ov\phi_N(t)}_{H^1(\Omega)'} + \norm{\psi_N(t)-\ov\psi_N(t)}_{H^1(\Gamma)'} 
	&\le C \tau^{\frac 1 2}, \\
	\label{EST:HLD:D}
	\norm{\delt \ov \phi_N}_{L^2(0,T;H^1(\Omega)')} + \norm{\delt \ov \psi_N}_{L^2(0,T;H^1(\Gamma)')} 
	&\le C.
	\end{alignat}
\end{subequations}

To prove this assertion, we first deduce from \eqref{disc:1} and \eqref{disc:2} that for all
$\theta \in H^1(\Om)$, $\sigma \in H^1(\Ga)$ and almost all $t\in[0,T]$,
\begin{subequations}
	\label{EQ:DTU}
	\begin{align}
	\label{EQ:DTU:1}	\inn{\delt\ov\phi_N(t)}{\theta}_{H^1(\Om)} 
	&= - \intO \grad \mu_N(t) \cdot \grad \theta \dx 
	- \intG \tfrac 1 L \big(\beta \nu_N(t) - \mu_N(t)\big) \theta \dS, \\
	\label{EQ:DTU:2}	\inn{\delt\ov\psi_N(t)}{\sigma}_{H^1(\Ga)} 
	&= - \intG \gradg \nu_N(t) \cdot \gradg \sigma \dS 
	+ \intG \tfrac 1 L \big(\beta \nu_N(t) - \mu_N(t)\big) \beta \sigma \dS.
	\end{align}
\end{subequations}
Let $s, t \in [0,T]$ be arbitrary and without loss of generality we suppose that $s < t$. Using \eqref{EQ:DTU:1}, along with the Cauchy--Schwarz inequality, we obtain
\begin{align}
\label{EST:DIFF:PHI}
\begin{aligned}
& \norm{\ov\phi_N(t)-\ov\phi_N(s)}_{H^1(\Om)'} 
= \underset{\norm{\theta}_{H^1(\Om)}=1}{\sup}\, \Big| \inn{\ov\phi_N(t)-\ov\phi_N(s)}{\theta}_{H^1(\Om)} \Big| \\
&\quad \le \underset{\norm{\theta}_{H^1(\Om)}=1}{\sup}\; \int_s^t \Big| \inn{\delt\ov\phi_N(r)}{\theta}_{H^1(\Om)} \Big| \,dr \\
&\quad \le C\left(1+\tfrac{1}{\sqrt{L}}\right) \int_s^t 
\norm{\grad \mu_N(r)}_{L^2(\Om)} + \tfrac{1}{\sqrt{L}}\norm{\beta\nu_N(r) - \mu_N(r)}_{L^2(\Ga)} \,dr \\
&\quad \le C\left(1+\tfrac{1}{\sqrt{L}}\right) \abs{t-s}^{\frac 1 2} 
\int_s^t\norm{\grad \mu_N}_{L^2(0,T;L^2(\Om))}^2 + \tfrac{1}{L}\norm{\beta\nu_N - \mu_N}_{L^2(0,T;L^2(\Ga))}^2 \dtau.
\end{aligned}
\end{align}
Proceeding similarly, we derive the estimate
\begin{align}
\label{EST:DIFF:PSI}
\begin{aligned}
&\norm{\ov\psi_N(t)-\ov\psi_N(s)}_{H^1(\Ga)'} \\
&\quad \le C\left(1+\tfrac{1}{\sqrt{L}}\right) \abs{t-s}^{\frac 1 2} 
\int_s^t\norm{\gradg \nu_N}_{L^2(0,T;L^2(\Ga))}^2 + \tfrac{1}{L}\norm{\beta\nu_N - \mu_N}_{L^2(0,T;L^2(\Ga))}^2 \dtau.
\end{aligned}
\end{align}
In combination with the uniform bound \eqref{EST:ENG} this proves \eqref{EST:HLD}.
Furthermore, for any $t\in[0,T]$ we can find $n\in\{1,...,N\}$ and $\alpha\in [0,1]$ such that $t=\alpha n\tau + (1-\alpha)(n-1)\tau$.
Hence, it follows immediately that
\begin{align*}
&\|\ov\phi_N(t)-\phi_N(t)\|_{H^1(\Om)'} 
\le \| \alpha \phi^n_N(t) + (1-\alpha) \phi_N^{n-1}(t) - \phi_N^n(t) \|_{H^1(\Om)'} \\
&\quad = (1-\alpha)\,\| \phi_N^n(t)  - \phi_N^{n-1}(t) \|_{H^1(\Om)'} = (1-\alpha) \, \| \ov \phi_N(n\tau) - \ov \phi_N((n-1)\tau)\|_{H^1(\Om)'}.
\end{align*}
The estimate 
\begin{align*}
\|\ov\psi_N(t)-\psi_N(t)\|_{H^1(\Ga)'} 
\le   (1-\alpha) \, \| \ov \psi_N(n\tau) - \ov \psi_N((n-1)\tau)\|_{H^1(\Ga)'}
\end{align*}
can be derived analogously.
Now, applying \eqref{EST:HLD} with $t = n\tau$ and $s = (n-1)\tau$ verifies \eqref{EST:DIFF}. Moreover, proceeding similarly as in \eqref{EST:DIFF:PHI} and \eqref{EST:DIFF:PSI}, we obtain
\begin{align}
\label{EST:DTPHIPSI}
	 &\int_0^T \norm{\delt\ov\phi_N(t)}_{H^1(\Om)'}^2 + \norm{\delt\ov\psi_N(t)}_{H^1(\Ga)'}^2 \dt \\
	 \notag &\quad \le C \left(1+\tfrac{1}{L}\right) 
	 	\int_0^T \norm{\grad \mu_N(t)}_{L^2(\Om)}^2
	 		+ \norm{\gradg \nu_N(t)}_{L^2(\Ga)}^2 
	 		+ \tfrac{1}{L}\norm{\beta\nu_N(t) - \mu_N(t)}_{L^2(\Ga)}^2 \dt.
\end{align}
In combination with \eqref{EST:ENG}, this proves \eqref{EST:HLD:D} and thus, the claim is established.

\subparagraph{Step 5: Convergence results.} 
Now, we intend to use the estimates established in Step~3 and Step~4 to prove convergence of our approximate solutions.
We claim that, there exists a quadruplet of functions
$(\phi,\psi,\mu,\nu)$ satisfying the regularity conditions \eqref{REG:ROB} such that
for every $r \in (0,\tfrac 12)$, 
\begin{alignat}{3}
\label{CONV:PHI}
\notag \phi_N & \to \phi
\quad &&\hbox{weakly-$^*$ in $\L\infty {\Lx p}$, $a.e.$ in $\OT$}, 
\\
&&&\quad \hbox{and strongly in $\L\infty{\Hx1'}$},
\displaybreak[1]
\\[1ex]
\label{CONV:PSI}
\notag \psi_N & \to \psi
\quad &&\hbox{weakly-$^*$ in $\L\infty {\LxG q}$, $a.e.$ on $\GT$}, 
\\
&&&\quad \hbox{and strongly in $\L\infty{\HxG1'}$}, 
\\[1ex]
\label{CONV:BPHI}
\notag \bphi_N & \to \phi
\quad &&\hbox{weakly in $\H1 {\Hx1'}$},
\\
&&&\quad \hbox{and strongly in $C^{0,r}([0,T];\Hx1')$},
\\[1ex]
\label{CONV:BPSI}
\notag \bpsi_N & \to \psi
\quad &&\hbox{weakly in $\H1 {\HxG1'}$}, 
\\
&&&\quad \hbox{and strongly in $C^{0,r}([0,T];\HxG 1')$},
\\[1ex]
\label{CONV:MU} 
\mu_N & \to \mu
\quad &&\hbox{weakly in $\L2{\Hx1}$},
\\[1ex]
\label{CONV:NU}
\nu_N & \to \nu
\quad &&\hbox{weakly in $\L2 {\HxG1}$}
\end{alignat}
along a non-relabeled subsequence.

To prove this assertion we proceed similarly as in \cite[s.~4.6]{GK} where the approach is carried out in more detail.
Due to \eqref{EST:HLD:D} and the uniform estimates established in Step~3, there exist functions $(\phi,\psi,\mu,\nu)$ such that the first lines of \eqref{CONV:PHI}, \eqref{CONV:PSI}, \eqref{CONV:BPHI} and \eqref{CONV:BPSI}, as well as \eqref{CONV:MU} and \eqref{CONV:NU} directly
follow from the Banach--Alaoglu theorem. 
Furthermore, arguing as in \cite[s.~4.6]{GK}
and invoking the compactness of the embeddings $\Lx p \subset \Lx2 \cong \Lx2 ' \emb \Hx1'$ and $\LxG q \subset \LxG2 \cong \LxG2' \emb \HxG1'$,
the Arzelà--Ascoli theorem for functions with values in a Banach
space (see, e.g., \cite[Lem.~1]{simon}) implies that
\begin{alignat*}{2}
\bphi_N &\to \phi \quad &&\hbox{strongly in $C^0([0,T];\Hx1')$},\\
\bpsi_N &\to \psi \quad &&\hbox{strongly in $C^0([0,T];\HxG1')$}.
\end{alignat*}
Passing to the limit in the H\"older estimate \eqref{EST:HLD} we even get $\phi \in C^{0,1/2}([0,T];\Hx1')$ and $\psi \in C^{0,1/2}([0,T];\HxG1')$.
An interpolation argument now verifies the second lines of \eqref{CONV:BPHI} and \eqref{CONV:BPSI}.
Eventually, using the estimate \eqref{EST:DIFF}, we conclude the second lines of \eqref{CONV:PHI} and \eqref{CONV:PSI}. Thus, the claim is established.

\subparagraph{Step 6: Existence of weak solutions.}
We now intend to show that the limit quadruplet $(\phi,\psi,\mu,\nu)$
is a weak solution to the system \eqref{CH:ROB}. According to Step 5, the functions $(\phi,\psi,\mu,\nu)$ exhibit the regularity demanded in \eqref{REG:ROB}.
To show that they fulfill the weak formulation \eqref{WF:ROB}, we want to pass to the limit $N\to \infty$ in the time-discrete scheme. Using both the piecewise constant extension and the piecewise linear extension, the system \eqref{disc} can be expressed as
\begin{subequations}
	\label{discN}
\begin{align}
\label{discN:1}
\intOT \partial_t \bphi_N(t) \theta \dx\dt
&= -\intOT {\nabla \mu_N}\cdot{\nabla \theta} \dx\dt
- \intGT \tfrac 1L (\beta \nu_N-\mu_N) \theta \dS\dt	,
\\
\label{discN:2}
\intGT \partial_t \bpsi_N(t) \sigma \dS\dt
&=
-\intGT {\nabla_\Ga \nu_N}\cdot{\nabla_\Ga \sigma} \dS\dt
+ \intGT \tfrac 1L (\beta \nu_N-\mu_N) \beta \sigma \dS\dt,
\\ 
\label{discN:3}
\intOT \mu_N \zeta \dx\dt
&  =
\intOT -(J* \phi_N)\zeta 
+ F'(x,\phi_N) \zeta \dx\dt,
\\ 
\label{discN:4}
\intGT \nu_N \xi \dS \dt&=
\intGT -(K \cg \psi_N) \xi
+ G'(x,\psi_N) \xi
+ B'(x,\psi_N)  \xi \dS\dt
\end{align}
\end{subequations}
holding for every $\theta \in \L2 {H^1(\Om)}$, $\sigma \in \L2 {H^1(\Ga)$} 
and $\bet=(\zeta,\xi) \in  \L2 {\LL^\infty}$.
Recalling the convergence properties \eqref{CONV:BPHI}--\eqref{CONV:NU}, it is straightforward to pass to the limit in \eqref{discN:1} and \eqref{discN:2}.
Using \eqref{CONV:PHI} and \eqref{CONV:PSI}, the terms depending on $J$ and $K$ in the equations \eqref{discN:3} and \eqref{discN:4} can be handled by Lemma~\ref{LEM:JK}. Moreover, the convergence  of the terms depending on $F'$ and $G'$ follows from \eqref{CONV:PHI} and \eqref{CONV:PSI}, the growth conditions \eqref{ass:growth} and Lebesgue's general convergence theorem (see \cite[p.~60]{alt}).
Eventually, the convergence of the term depending on $B'$
follows directly from the assumption in \eqref{ass:pen}. This allows us to pass to the limit also in \eqref{discN:3} and \eqref{discN:4}. 
Hence, the quadruplet $(\phi,\psi,\mu,\nu)$ satisfies the weak formulation \eqref{WF:ROB}.
To verify the energy inequality \eqref{energy:ROB} we first observe that 
\begin{align*}
	E\big(\phi(t),\psi(t)\big) \le \underset{N\to\infty}{\lim\inf}\; E\big(\phi_N(t),\psi_N(t)\big)
	\quad\text{for almost all $t\in[0,T]$}
\end{align*}
follows by similar arguments as in Step 2 by means of Lemma~\ref{LEM:JK}, the assumption \eqref{ass:pen}, and the convexity of $F$ and $G$. Hence, taking the limes inferior in \eqref{EST:ENG} and recalling the convergence properties \eqref{CONV:MU} and \eqref{CONV:NU}, we conclude \eqref{energy:ROB}. This proves that $(\phi,\psi,\mu,\nu)$ is indeed a weak solution of the system \eqref{CH:ROB}.

\subparagraph{Step 7: Uniqueness.}
To prove uniqueness, we consider two solutions $(\phi_i,\psi_i,\mu_i,\nu_i)_i$, $i=1,2$ to the system \eqref{CH:ROB} as given by Theorem~\ref{THM:WP:ROB}.
For convenience, we write
\begin{align*}
(\phi,\psi,\mu,\nu) := (\phi_1,\psi_1,\mu_1,\nu_1) - (\phi_2,\psi_2,\mu_2,\nu_2).
\end{align*}
Let now $t_0 \in [0,T]$, $\eta\in \L2 {\Hx1}$ and $\rho\in \L2 {\HxG1}$ be arbitrary. Testing the difference of the weak formulations \eqref{WF:ROB:1}--\eqref{WF:ROB:4} written for $(\phi_1,\psi_1,\mu_1,\nu_1)$ and $(\phi_2,\psi_2,\mu_2,\nu_2)$, respectively, with the test functions
\begin{align*}
\theta  := 
\begin{cases}
\int_t^{t_0} \eta \ds \quad &\hbox{if $t \leq t_0$}\\
0 \quad &\hbox{if $t > t_0$}
\end{cases},
\quad \hbox{and} \quad
\sigma  :=
\begin{cases}
\int_t^{t_0} \rho \ds \quad &\hbox{if $t \leq t_0$}\\
0 \quad &\hbox{if $t > t_0$}
\end{cases},
\end{align*}
and integrating in time from $0$ to $t_0$, we obtain the relations
\begin{align*}
\intOto \phi \eta \dx\dt
&= - \intOto \nabla \left( \int_0^{t} \mu \ds \right) \cdot \nabla \eta \dx\dt
- \frac 1 L \intGto \left( \int_0^{t} (\beta \nu-\mu) \ds \right)  \eta \dS\dt,
\\
\intGto \psi \rho \dS\dt
&= - \intGto \gradg \left( \int_0^{t} \nu \ds \right) \cdot \gradg \rho \dS\dt
+ \frac 1 L \intGto \left( \int_0^{t} (\beta \nu-\mu) \ds \right)  \beta\rho \dS\dt,
\end{align*}
by means of Fubini's theorem.
Recall that $\SS_\Om (\phi)$ and $\SS_\Ga (\psi)$ can be represented by
\begin{align*}
\SS_\Om (\phi)= - \int_0^t \mu \ds + ct \beta, 
\quad
\SS_\Ga (\psi)= - \int_0^t \nu \ds + ct
\end{align*}
for some constant $c\in\R$.
Hence, choosing $\eta=\mu$ and $\rho=\nu$, a straightforward computation leads to  
\begin{align*}
\frac 12 \norm{(\phi,\psi)(t_0)}^2_{L,\beta,*}
=  -\intOto \phi \mu \dx\dt  -	\intGto \psi \nu \dS\dt.
\end{align*}
We now recall Remark~\ref{REM:weakLpq} which allows to test
\eqref{WF:ROB:3} and \eqref{WF:ROB:4} with functions $\zeta\in L^p(\Om)$ and $\xi\in L^q(\Ga)$.
We are thus allowed to choose the test functions
\begin{align*}
\zeta= \phi  \chi_{[0,t_0]}, \quad \xi= \psi \chi_{[0,t_0]}
\end{align*}
where $\chi_{[0,t_0]}$ stands for the characteristic function of the interval $[0,t_0]$. Plugging $\zeta$ and $\xi$ into \eqref{WF:ROB:3} and \eqref{WF:ROB:4} now gives
\begin{align*}
& \intOto \phi\mu  \dx\dt  
+ \intGto \psi\nu  \dS\dt \\
&\quad = - \intOto (J* \phi)\phi \dx\dt 
- \intGto (K \cg \psi)\psi \dS\dt	
+ \intGto (B'(\cdot,\psi_1)-B'(\cdot,\psi_2))\psi \dS\dt\\
&\qquad\quad 
+ \intOto (F'(\cdot,\phi_1)-F'(\cdot,\phi_2))\phi \dx\dt 
+ \intGto (G'(\cdot,\psi_1)-G'(\cdot,\psi_2))\psi \dS\dt.
\end{align*}
Invoking \eqref{EST:FG:2}, which implies uniform monotonicity of $F'$ and $G'$ with respect to their second argument, we obtain the estimate
\begin{align*}
&\intOto (F'(\cdot,\phi_1)-F'(\cdot,\phi_2))\phi \dx\dt
+ \intGto (G'(\cdot,\psi_1)-G'(\cdot,\psi_2))\psi \dS\dt \\
&\quad \geq
c_* \intOto |\phi|^2 \dx\dt
+ c_\cg \intGto |\psi|^2 \dS\dt
\end{align*}
where $c_*$ and $c_\cg$ are the constants from \eqref{EST:FGW}.
Furthermore, using the assumption \eqref{ass:pen:uniq}, Lemma~\ref{LEM:JK:UNIQ},
Lemma~\ref{LEM:NORMS}, and Young's inequality, we infer that for any $\delta>0$,
\begin{align*}
& \left|\, - \intOto (J* \phi)\phi \dx\dt - \intGto (K \cg \psi)\psi \dS\dt	
+ \intGto (B'(\cdot,\psi_1)-B'(\cdot,\psi_2))\psi \dS\dt \,\right|
\\ & \qquad \leq
\int_0^{t_0} 
	\big| 
		\biginn{(\phi \,{,}\, \psi)}
		{\big(J\ast\phi \,{,}\, K\cg\psi\big)}_{\HH^1}  
	\big| \dt
	+ \LIP \intGto |\psi|^2 \dS\dt
\\ & \qquad \leq
\int_0^{t_0} 
		\norm{(\phi,\psi)}_{(\HH^1)'}\,
		\big\| (J\ast\phi \,{,}\, K\cg\psi\big) \big\|_{\HH^1}  
	\dt
	+ \LIP \intGto |\psi|^2 \dS\dt
\\ & \qquad
\leq
C \int_0^{t_0} \norm{(\phi,\psi)}_{L,\beta,*} \Big( \norm{J * \phi}_{\Hx1}^2 + \norm{K \cg \psi}_{\HxG1}^2 \Big)^{1/2} \dt
+ \LIP \intGto |\psi|^2 \dS\dt
\\ & \qquad
\leq 
\delta  \intOto |\phi|^2 \dx\dt
+ (\delta+\LIP) \intGto |\psi|^2 \dS\dt
+ \frac{C}{\delta} \int_0^{t_0} \norm{(\phi,\psi)(t)}^2_{L,\beta,*} \dt.
\end{align*}
Combining the above estimates, we infer that 
\begin{align*}
&\frac 12 \norm{(\phi,\psi)(t_0)}^2_{L,\beta,*}
+ (c_*-\delta) \intOto |\phi|^2 \dx\dt
+ (c_\cg - \LIP - \delta) \intGto |\psi|^2 \dS\dt\\
&\quad \leq 
\frac{C}{\delta} \int_0^{t_0} \norm{(\phi,\psi)(t)}^2_{L,\beta,*} \dt.
\end{align*}
Recalling that $\LIP < c_\cg $, we fix $\delta = \frac 1 2 \min\{c_*,c_\cg-\LIP\}$, and thus, Gronwall's lemma yields
\begin{align*}
\norm{\SS (\phi,\psi)(t)}_{L,\beta} = \norm{(\phi,\psi)(t)}_{L,\beta,*} = 0 \quad \hbox{for almost all $t \in [0,T]$}.
\end{align*}
By the definition of the solution operator $\SS$ in \eqref{pre:S} we conclude that $\phi=0$ $\,a.e.$ in $\OT$ and $\psi=0$ $\,a.e.$ on $\GT$.
Finally, the identities $\mu=0$ $\,a.e.$ in $\OT$ and $\nu=0$ $\,a.e.$ on $\GT$ follow from \eqref{WF:ROB:3} and \eqref{WF:ROB:4} by a standard comparison argument. Hence, the proof of Theorem~\ref{THM:WP:ROB} is complete.
\end{proof}

\subsection{Higher regularity and strong well-posedness}

Under the additional assumptions \eqref{ass:dom:strong}--\eqref{ass:pen:strong}, we can even establish strong well-posedness of the system \eqref{CH:ROB}.

\begin{thm}[Strong well-posedness of the system \eqref{CH:ROB}]
	\label{THM:STRONG:ROB} 
	Let $T,L > 0$, $m \in \R$ and $\beta\neq 0$ be arbitrary and suppose that the conditions \eqref{ass:dom}--\eqref{ass:pen:strong} hold. 
	For any initial datum $(\phi_0,\psi_0)\in \HH^2_{\beta, m}$ let $(\phi,\psi,\mu,\nu)$ denote the unique weak solution to \eqref{CH:ROB} as given by Theorem~\ref{THM:WP:ROB}. Then the solution $(\phi,\psi,\mu,\nu)$ enjoys the following additional regularity:
	\begin{align*}
	(\phi,\psi) &\in \H1 {\LL^2}, 
	\quad
	(\mu,\nu) \in \L\infty {\HH^1}\cap \L2 {\HH^2}, 
	\quad
	\partial_{\bold n} \mu \in \L2 {\LxG2}.
	\end{align*}
	This means that $(\phi,\psi,\mu,\nu)$ is a strong solution to 
	system \eqref{CH:ROB} as all equations are satisfied $a.e.$ in $\OT$ and $\Om$, and $a.e.$ on $\GT $ and $\Ga$, respectively.
\end{thm}

\begin{rem}
	\label{REM:STRONG:WP}
	We point out that $F(\cdot,\phi_0) \in \Lx1$ and $B(\cdot,\psi_0)$, $G(\cdot,\psi_0) \in \LxG1$ are satisfied as
	a consequence of $(\phi_0,\psi_0)\in \HH^2_{\beta, m}$, assumption \eqref{ass:pot:strong}, and the continuous embedding $\HH^2\emb\LINF$.
\end{rem}

\begin{proof}[Proof of Theorem~\ref{THM:STRONG:ROB}]
The formal idea behind this proof is to test \eqref{WF:ROB:1} with $-\delt \mu$,
\eqref{WF:ROB:2} with $-\delt \nu$, the time derivative of \eqref{WF:ROB:3} with $\delt \phi$ and the time derivative of \eqref{WF:ROB:4} with $\delt \psi$. 
This strategy can be made rigorous by following the line of argument in \cite[s.~3.2]{KLLM} (which was in turn greatly inspired by the approach in \cite[s.~4.4]{Colli-Fukao}). To this end, we use once more the time-discrete approximate solution $(\phi^n, \psi^n, \mu^n, \nu^n)_{n=1,...,N}$ constructed in Step $1$ of the proof of Theorem~\ref{THM:WP:ROB}. We recall that the discretized weak formulation \eqref{disc} is satisfied and we define the backward difference quotients
\begin{align*}
	(\deltau \phi^{n+1},\deltau \psi^{n+1},\deltau \mu^{n+1},\deltau \nu^{n+1})
	:= \frac 1\tau \Big[ (\phi^{n+1},\psi^{n+1},\mu^{n+1},\nu^{n+1})
		- (\phi^{n},\psi^{n},\mu^{n},\nu^{n}) \Big].
\end{align*}
For brevity, we will also use the notation
\begin{align*}
	\deltau F' := \frac 1 \tau \big[ F'(\cdot,\phi^{n+1}) - F'(\cdot,\phi^{n}) \big], \quad
	\deltau H' := \frac 1 \tau \big[ H'(\cdot,\psi^{n+1}) - H'(\cdot,\psi^{n}) \big] \;\text{for}\,\; H\in\{G,B\}.
\end{align*}
In the following, the letter $C$ will denote generic positive constants independent of $n$, $N$, $\tau$ and $L$ which may change their value from line to line.
For any arbitrary $n \in \{0, ..., N-1\}$, we test \eqref{disc:1} with $\theta= - \deltau \mu^{n+1}\in \Hx1$ and \eqref{disc:2}
with $\sigma= - \deltau \nu^{n+1}\in \HxG1$. Adding the resulting equations, we conclude that
\begin{align}
\begin{aligned}
	& - \intO \deltau \phi^{n+1}\deltau \mu^{n+1} \dx 
	- \intG \deltau \psi^{n+1}\deltau \nu^{n+1} \dS
	\\ & \quad
	=
	\frac 1{2\tau} \Big( \norm{\nabla \mu^{n+1}}^2_{\Lx2}
		-\norm{\nabla \mu^{n}}^2_{\Lx2}
		+\norm{\nabla \mu^{n+1}-\nabla \mu^{n}}^2_{\Lx2} \Big) 
	\\ &\label{strong:1} \qquad
	+
	\frac 1{2\tau} \Big( \norm{\nabla_\Ga \nu^{n+1}}^2_{\LxG2}
		-\norm{\nabla_\Ga \nu^{n}}^2_{\LxG2}
		+\norm{\nabla_\Ga \nu^{n+1}-\nabla_\Ga \nu^{n}}^2_{\LxG2} \Big) 
	\\ & \qquad
	+
	\frac 1{2\tau} \frac 1{L} 
	\Big( \norm{\beta \nu^{n+1}-\mu^{n+1}}^2_{\LxG2}
		-\norm{\beta \nu^{n}-\mu^{n}}^2_{\LxG2}
	\\ & \qquad\qquad\qquad\qquad
		+\norm{\beta (\nu^{n+1}-\nu^{n}) - (\mu^{n+1} -\mu^{n})}^2_{\LxG2} \Big) .
\end{aligned}	
\end{align}
Next, we take the difference of \eqref{disc:3} and \eqref{disc:4} written at the step $n+1$ and at the step $n$, respectively, and we add the resulting equations.
Using $\zeta = \tfrac 1\tau \deltau \phi^{n+1} \in L^p(\Om)$ and $\xi = \tfrac 1\tau \deltau \psi^{n+1} \in L^q(\Ga)$ as test functions, we obtain
\begin{align}
	\label{strong:2}
	&  \intO \deltau \mu^{n+1} \deltau \phi^{n+1} \dx +
	\intG \deltau \nu^{n+1} \deltau \psi^{n+1} \dS \notag 
	\\ 
	&\begin{aligned}
	& \quad=
	\intO -(J * \deltau \phi^{n+1}) \deltau \phi^{n+1}+\deltau F'\, \deltau \phi^{n+1} \dx
	\\ & \qquad 
	+ \intG -(K \cg \deltau \psi^{n+1}) \deltau \psi^{n+1} 
	+ \deltau G'\,\deltau \psi^{n+1} \dS
	\\ & \qquad
	+\intG \deltau B' \, \deltau \psi^{n+1} \dS	.
	\end{aligned}
\end{align}
From the uniform convexity of the potentials (see \eqref{EST:FG:2}), we infer that 
\begin{align*}
	&\intO \deltau F' \deltau \phi^{n+1} \dx
	+ \intG \deltau G' \deltau \psi^{n+1} \dS  
	\geq 
	c_* \norm{\deltau\phi^{n+1}}^2_{\Lx2}
	+ c_\cg \norm{\deltau\psi^{n+1}}^2_{\LxG2}.
\end{align*}
Using \eqref{ass:pen:uniq}, Lemma~\ref{LEM:JK:UNIQ} and Young's inequality, we deduce the estimate
\begin{align*}
	& \left| \intO -(J * \deltau \phi^{n+1}) \deltau \phi^{n+1} \dx
	+\intG \big( -(K \cg \deltau \psi^{n+1})
	+ \deltau B'\big) \deltau \psi^{n+1}\dS \right|
	\\[1ex]  & \quad
	\leq 
	\delta \norm{\deltau \phi^{n+1}}^2_{\Lx2}
	+ (\delta + \LIP) \norm{\deltau \psi^{n+1}}^2_{\LxG2}
	+ \frac C \delta \norm{\deltau  \bph^{n+1}}^2_{L,\beta, *}
\end{align*}
for any $\delta>0$, where $\bph^{n+1}=(\phi^{n+1},\psi^{n+1})$. 
According to \eqref{DEF:MUNU:0} and \eqref{DEF:MUNU}, the functions $\mu^{n+1}$ and $\nu^{n+1}$ can be expressed as
\begin{align*}
	\mu^{n+1} = \SS_\Om(\deltau\bph^{n+1}) + \beta c^{n+1},\quad \nu^{n+1} = \SS_\Ga(\deltau\bph^{n+1}) + c^{n+1}.
\end{align*}
Consequently, a straightforward computation gives
\begin{align*}
	\norm{\deltau  \bph^{n+1}}^2_{L,\beta, *}
	= \norm{\nabla \mu^{n+1}}^2_{\Lx2} 
	+  \norm{\nabla_\Ga \nu^{n+1}}^2_{\LxG2} 
	+  \tfrac 1L \norm{\beta \nu^{n+1} - \mu^{n+1} }^2_{\LxG2}.
\end{align*}
Hence, adding \eqref{strong:1} and \eqref{strong:2}, and using the above estimates, we obtain
\begin{align*}
	& \norm{\nabla \mu^{n+1}}^2_{\Lx2}
		-\norm{\nabla \mu^{n}}^2_{\Lx2}
	+ \norm{\nabla_\Ga \nu^{n+1}}^2_{\LxG2}
		-\norm{\nabla_\Ga \nu^{n}}^2_{\LxG2}
	\\ & \notag\qquad
	+
	\tfrac 1{L} \big(\norm{\beta \nu^{n+1}-\mu^{n+1}}^2_{\LxG2}
	-\norm{\beta \nu^{n}-\mu^{n}}^2_{\LxG2} \big)
	\\ & \notag\qquad
	+ (c_*-\delta) \norm{\deltau\phi^{n+1}}^2_{\Lx2}
	+ (c_\cg - \LIP -\delta)\norm{\deltau\psi^{n+1}}^2_{\LxG2}
	\\[1ex] & \quad
	\leq C \big( \norm{\nabla \mu^{n+1}}^2_{\Lx2} 
	+  \norm{\nabla_\Ga \nu^{n+1}}^2_{\LxG2} 
	+  \tfrac 1L \norm{\beta \nu^{n+1} - \mu^{n+1} }^2_{\LxG2}\big).
\end{align*}
We now fix $\delta = \frac 1 2 \min\{c_*,c_\cg-\LIP\}$ and sum from $n=0$ to an arbitrary index $j< N-1$ to infer that
\begin{align}
\label{EST:STR}
\begin{aligned}
	& \norm{\nabla \mu^{j+1}}^2_{\Lx2}
	+ \norm{\nabla_\Ga \nu^{j+1}}^2_{\LxG2}
	+\tfrac 1{L} \norm{\beta \nu^{j+1}-\mu^{j+1}}^2_{\LxG2} 
	\\ & \qquad
	+ \frac {c_*} 2  \int_0^{j \tau} \norm{\delt\bphi_N}^2_{\Lx2} \dt
	+ \frac 1 2 (c_\cg-\LIP) \int_0^{j \tau}  \norm{\delt\bpsi_N}^2_{\LxG2}\dt
	\\[1ex] & \quad
	\leq  \norm{\nabla \mu_N(0)}^2_{\Lx2} 
	+  \norm{\nabla_\Ga \nu_N(0)}^2_{\LxG2} 
	+  \tfrac 1L \norm{\beta \nu_N(0) - \mu_N(0) }^2_{\LxG2}
	\\ & \qquad
	+ C \int_0^T \big( \norm{\nabla \mu_N}^2_{\Lx2} 
	+  \norm{\nabla_\Ga \nu_N}^2_{\LxG2} 
	+  \tfrac 1L \norm{\beta \nu_N - \mu_N }^2_{\LxG2}\big) \ds,
\end{aligned}
\end{align}
where $\bphi_N,\bpsi_N$ denote the piecewise linear extension introduced
in Step 1 of the proof of Theorem~\ref{THM:WP:ROB}. 
By the fundamental theorem of calculus of variations we deduce that $\mu_N(0)=\mu^0$ and $\nu_N(0)=\nu^0$ satisfy the equations
\begin{subequations}
\begin{alignat}{2}
\label{EQ:COMP:1}
	\mu^{0} &= - (J* \phi_0)+ F'(\cdot,\phi_0) 
	&&\quad a.e.\;\text{in}\;\Omega,\\
\label{EQ:COMP:2}	
	\nu^{0} &= - (K \cg \psi_0) + G'(\cdot,\psi_0) + B'(\cdot,\psi_0) 
	&&\quad a.e.\;\text{on}\;\Gamma.
\end{alignat}
\end{subequations}
Since $\bph_0 \in \HH^2_{\beta,m}\,$, we use Lemma~\ref{LEM:JK:UNIQ}, the assumptions \eqref{ass:pot:strong}, \eqref{ass:pen:strong} and the continuous embedding $\HH^2\emb \LINF$
to deduce that the right-hand side of \eqref{EQ:COMP:1} belongs to $H^1(\Om)$ whereas the right-hand side of \eqref{EQ:COMP:2} belongs to $H^1(\Ga)$.
By comparison, we infer that $\mu_N(0)=\mu^0 \in \Hx1$ and $\nu_N(0)=\nu^0 \in \HxG1$.
Invoking the uniform bound \eqref{unif_est}, we conclude from \eqref{EST:STR} that
\begin{align}
\label{EST:STR:2}
\begin{aligned}
& \norm{\nabla \mu^{j+1}}^2_{\Lx2}
+ \norm{\nabla_\Ga \nu^{j+1}}^2_{\LxG2}
+ \tfrac 1{L} \norm{\beta \nu^{j+1}-\mu^{j+1}}^2_{\LxG2} 
\\ & \qquad
+ \frac {c_*} 2  \int_0^{j \tau} \norm{\delt\bphi_N}^2_{\Lx2} \dt
+ \frac 1 2 (c_\cg-\LIP) \int_0^{j \tau}  \norm{\delt\bpsi_N}^2_{\LxG2}\dt
\le C
\end{aligned}
\end{align}
for all $j\in \{0,...,N-1\}$. This directly implies the uniform bound
\begin{align}
\label{EST:STR:3}
\begin{aligned}
	&\norm{\delt\bphi_N}_{\L2 {\Lx2}}
	+ \norm{\delt\bpsi_N}_{\L2 {\LxG2}} \\
	&\quad + \norm{\mu_N}_{\L\infty {\Hx1}}
	+ \norm{\nu_N}_{\L\infty {\HxG1}}
	\le C.
\end{aligned}
\end{align}
Now, invoking the Banach--Alaoglu theorem, we conclude that
\begin{align*}
	(\phi,\psi) \in \H1 {\LL^2}, \quad (\mu,\nu) \in \L\infty {\HH^1}
\end{align*}
due to the uniqueness of the weak limit. In particular, we obtain the bound
\begin{align}
	\label{BND:PP:L}
	\norm{(\phi,\psi)}_{H^1(0,T;\LL^2)} + \norm{(\mu,\nu)}_{L^\infty(0,T;\HH^1)} \le C.
\end{align}
Recall that, according to \eqref{WF:ROB:1} and \eqref{WF:ROB:2}, the elliptic problems 
\begin{align*}
	\left\{
	\begin{aligned}
	\Delta \mu &= \delt \phi &&\hbox{in $\Om$,}
	\\
	\deln \mu &= \tfrac 1L (\beta \nu - \mu) &&\hbox{on $\Ga$,}
	\end{aligned}
	\right.
	\qquad 
	\Delta_\Ga \nu = \delt \psi + \tfrac \beta L (\beta \nu - \mu) \;\;\hbox{on $\Ga$,}
\end{align*}
are satisfied in the weak sense. 
Employing elliptic regularity theory (see, e.g., \cite[s.~5, Prop.~7.7]{taylor} for the Poisson--Neumann problem in the bulk and \cite[s.~5, Thm.~1.3]{taylor} for Poisson's equation on the boundary), we obtain that $\mu(t)\in H^2(\Om)$ and $\nu(t)\in H^2(\Ga)$ for almost all $t\in [0,T]$ with
\begin{align}
	\label{BND:MU:L}
	\begin{split}
	\norm{\mu(t)}_{H^2(\Om)}^2 &\le C \left( \norm{\mu(t)}_{H^1(\Om)}^2 + \norm{\delt\phi(t)}_{L^2(\Om)}^2 + \frac{1}{L^2}\norm{\beta\nu(t)-\mu(t)}_{H^{1/2}(\Ga)}^2 \right) \\
	&\le C \left( \Big(1+\frac{1}{L^2}\Big)\norm{\mu(t)}_{H^1(\Om)}^2 + \norm{\delt\phi(t)}_{L^2(\Om)}^2 + \frac{\beta^2}{L^2}\norm{\nu(t)}_{H^1(\Ga)}^2  \right), 
	\end{split}
	\\[1ex]
	\label{BND:NU:L}
	\norm{\nu(t)}_{H^2(\Ga)}^2 &\le C \left( \norm{\nu(t)}_{H^1(\Ga)}^2 + \norm{\delt\psi(t)}_{L^2(\Ga)}^2 + \frac{1}{L^2}\norm{\beta\nu(t)-\mu(t)}_{L^2(\Ga)}^2 \right).
\end{align}
Integrating the above estimates in time from $0$ to $T$, we conclude that $(\mu, \nu ) \in \L2{\HH^2}$ and $\deln \mu \in \L2 {\LxG2}$.
In combination with \eqref{REG:ROB}, this proves the regularity assertion. It directly follows that $(\phi,\psi,\mu,\nu)$ is even a strong solution to the system \eqref{CH:ROB} and thus, the proof of Theorem~\ref{THM:WP:ROB} is complete. 
\end{proof}

\section{Singular limits of the Robin model}
This section is devoted to investigating the asymptotic limits of the system \eqref{CH:ROB}
as $L$ tends to zero or to infinity. To this end, we prove that the sequence of solutions 
$(\phi^L,\psi^L,\mu^L,\nu^L)$ to the Robin model \eqref{CH:ROB} converges in a suitable topology such that the limit
is a weak solution to the corresponding limiting system. Therefore, let us first introduce 
the notions of weak solutions for these systems.

\subsection{Notion of weak solutions to the limit models}

We now present the definitions of weak solutions to the systems \eqref{CH:DIR}, \eqref{CH:DEC:BULK} and \eqref{CH:DEC:SURF}.

\begin{defn}[Definition of a weak solution to \eqref{CH:DIR}] \label{DEF:WP:DIR}
	Let $T>0$, $m\in\R$, $\beta> 0$ and $(\phi_0,\psi_0)\in \HHBMo$ be arbitrary and suppose that the conditions \eqref{ass:dom}--\eqref{ass:pen} hold. 
	The triplet $(\phi,\psi,\mu)$ is called a weak solution of the system \eqref{CH:DIR} if the following holds:
	\begin{enumerate}[label=$(\mathrm{\roman*})$, ref = $\mathrm{\roman*}$]
		\item The functions $(\phi,\psi,\mu)$ have the following regularity
		\begin{align}
		\label{REG:DIR}
		\left\{
		\begin{aligned}
		\phi & \in \L\infty {\Lx p}, \\
		\psi & \in \L\infty {\LxG q}, \\
		(\phi,\psi)&\in C^{0,\frac 12}([0,T];(\DDBI)')\cap H^1(0,T;(\DDBI)'),\\
		\mu &\in L^2\big(0,T;\VV^1\big),
		\end{aligned}
		\right.
		\end{align}
		and it holds that $(\phi(t),\psi(t))\in \HHBMo$ for almost all $t\in[0,T]$. 
		\item
		The weak formulation
		\begin{subequations}
			\label{WF:DIR}
			\begin{align}
			\label{WF:DIR:1}
			&\biginn{\big( \delt \phi,\delt \psi \big)}{(\theta,\sigma)}_{\DDBI}  =  - \intO \grad\mu \cdot \grad \theta \dx - \frac 1 \beta \intG \gradg\mu \cdot \gradg \sigma \dS,
			\\
			\label{WF:DIR:2}
			&\intO \mu \eta \dx = \intO - (J * \phi)\eta + F'(\cdot,\phi)\eta \dx, \\
			\label{WF:DIR:3}
			&\intG \mu \theta \dS = \intG - \beta(K \cg \psi)\theta + \beta G'(\cdot,\psi)\theta 
			+ \beta B'(\cdot,\psi)\theta  \dS	 
			\end{align}
		\end{subequations}		
		is satisfied almost everywhere in $[0,T]$ for all test functions $(\theta,\sigma)\in \DDBI$, $\eta\in L^\infty(\Omega)$ and $\theta\in L^\infty(\Gamma)$. Moreover, the initial conditions $\phi\vert_{t=0}=\phi_0$ and $\psi\vert_{t=0}=\psi_0$ are satisfied $a.e.$ in $\Omega$ and on $\Gamma$, respectively.
		\item The energy inequality 
		\begin{align}
		\label{energy:DIR}
		&E\big(\phi(t),\psi(t)\big) 
		+ \frac 1 2 \int_0^t 
		\norm{\grad\mu(s)}_{L^2(\Omega)}^2 
		+ \norm{\gradg\nu(s)}_{L^2(\Gamma)}^2
		\ds 
		\le E(\phi_0,\psi_0)
		\end{align}
		is satisfied for all $t\in[0,T]$.
	\end{enumerate}
\end{defn}

\medskip

\begin{defn}[Definition of weak solutions to \eqref{CH:DEC:BULK} and \eqref{CH:DEC:SURF}]
\label{DEF:LTOINFTY}
Let $T>0$, and $(\phi_0,\psi_0)\in \LL^2$ be arbitrary, and suppose that the conditions \eqref{ass:dom}--\eqref{ass:pen} hold. 
The pairs $(\phi,\mu)$ and $(\psi,\nu)$ are called weak solutions of the systems \eqref{CH:DEC:BULK}
and \eqref{CH:DEC:SURF} if the following holds:
	\begin{enumerate}[label=$(\mathrm{\roman*})$, ref = $\mathrm{\roman*}$]
		\item The functions $(\phi,\mu)$ and $(\psi,\nu)$ have the following regularity
		\begin{align}
		\label{REG:NEUM}
		\left\{
		\begin{aligned}
		&\phi  \in \H1 {\Hx1'} \cap \L\infty {\Lx p},
		\quad
		&&\mu \in \L2 {\Hx1}, \\
		&\psi  \in \H1 {{\HxG1'}}\cap \L\infty {\LxG q}, \quad
		&&\nu \in L^2\big(0,T; \HxG1 \big)
		\end{aligned}
		\right.
		\end{align}
		and it holds that $\meano{\phi(t)}=\meano{\phi_0}$ and $\meang{\psi(t)}=\meang{\psi_0}$ for almost all $t\in [0,T]$.
		\item
		The weak formulations
		\begin{subequations}
			\label{WF:NEU:BULK}
			\begin{align}
			\label{WF:NEU:BULK:1}
			&\inn{\delt \phi}{ \theta}_{H^1(\Omega)}=  {-}  \intO \grad\mu \cdot \grad \theta \dx,
			\\
			\label{WF:NEU:BULK:2}
			&\intO \mu \zeta \dx = \intO - (J * \phi)\zeta + F'(\cdot,\phi)\zeta \dx,
			\end{align}
		\end{subequations}
		and 
		\begin{subequations}
			\label{WF:NEU:SUR}
			\begin{align}
			\label{WF:NEU:SUR:1}
			&\inn{\delt \psi}{ \sigma}_{H^1(\Ga)} =  - \intG \gradg\nu \cdot \gradg \sigma \dS ,
			\\
			\label{WF:NEU:SUR:2}
			&\intG \nu  \xi  \dS = \intG - (K \cg \psi) \xi  +  G'(\cdot,\psi) \xi  
			+  B'(\cdot,\psi) \xi   \dS	
			\end{align}
		\end{subequations}
		are satisfied almost everywhere in $[0,T]$ for all test functions $\theta \in \Hx1$,
		$\zeta\in L^\infty(\Omega)$ and $\sigma\in \HxG1, \xi \in  L^\infty(\Gamma)$. Moreover, the initial conditions $\phi\vert_{t=0}=\phi_0$ and $\psi\vert_{t=0}=\psi_0$ are satisfied $a.e.$ in $\Omega$ and on $\Gamma$, respectively.
		\item The energy inequalities 
		\begin{subequations}
		\label{energy:NEU}
		\begin{align}
		\label{energy:NEU:BULK}
		E_\textnormal{bulk}\big(\phi(t)\big) 
		+ \frac 1 2 \int_0^t 
		\norm{\grad\mu(s)}_{L^2(\Omega)}^2 
		\ds 
		&\le E_\textnormal{bulk}(\phi_0),
		\\
		\label{energy:NEU:SURF}
		E_\textnormal{surf}\big(\psi(t)\big) 
		+ E_\textnormal{pen}\big(\psi(t)\big)
		+ \frac 1 2 \int_0^t 
		\norm{\gradg\nu(s)}_{L^2(\Gamma)}^2
		\ds 
		&\le E_\textnormal{surf}(\psi_0) + E_\textnormal{pen}(\psi_0)
		\end{align}
		\end{subequations}
		are satisfied for all $t\in[0,T]$.
	\end{enumerate}
\end{defn}

\begin{rem}
	We are convinced that weak well-posedness of the Dirichlet model \eqref{CH:DIR} and the decoupled model (\eqref{CH:DEC:BULK},\eqref{CH:DEC:SURF}) can be proved in a similar fashion to the proof of Theorem~\ref{THM:WP:ROB} by exploiting the gradient flow equations \eqref{GFE:DIR}, \eqref{GFE:NEUM:BULK} and \eqref{GFE:NEUM:SURF}.
	However, as we want to investigate the singular limits $L\to 0$ and $L\to\infty$ of the Robin model anyway, we proceed differently and construct the weak solutions as the singular limits of solutions to the system \eqref{CH:ROB}. 
	
	The weak and strong well-posedness of the Dirichlet model will be established in Theorem~\ref{THM:Ltozero}, whereas the weak and strong well-posedness of the decoupled model will be presented in Theorem~\ref{THM:Ltoinfty}.
\end{rem}

\subsection{Uniform bounds}

To investigate the singular limits, we first derive uniform bounds on solutions of the Robin model.

\paragraph{Uniform bounds on weak solutions.}
Suppose that $T>0$, $\beta\neq 0$, $m\in\R$ and that \eqref{ass:dom}--\eqref{ass:pen} hold. Let $\bph_0=(\phi_0,\psi_0)\in \HHBMo$ with $F(\cdot,\phi_0)\in L^1(\Om)$ and $G(\cdot,\psi_0), B(\cdot,\psi_0)\in L^1(\Ga)$ be arbitrary. For $L>0$, let $(\phi^L,\psi^L,\mu^L,\nu^L)$ denote the corresponding weak solution to the system \eqref{CH:ROB} in the sense of Definition~\ref{DEF:WP:ROB}. 
In the following, the letter $C$ will denote a generic positive constant that does not depend on the parameter $L$. 
From the energy inequality \eqref{energy:ROB} and the growth conditions in \eqref{ass:growth} we infer that
\begin{align}
	\label{UNI:1}
	\norm{\phi^L}_{L^\infty(0,T;L^p(\Om))}
		+ \norm{\psi^L}_{L^\infty(0,T;L^q(\Ga))} 
	&\le C, \\
	\label{UNI:2}
	\norm{\grad\mu^L}_{L^2(\Om_T)}^2 
		+ \norm{\gradg\nu^L}_{L^2(\Ga_T)}^2
		+ \frac 1 L \norm{\beta\nu^L-\mu^L}_{L^2(\Ga_T)}^2
	&\le C.
\end{align}
On the basis of these estimates, we can now follow the line of argument in Step 3 of the proof of Theorem~\ref{THM:WP:ROB} to deduce the uniform bound
\begin{align}
	\label{UNI:3}
	\norm{\mu^L}_{L^2(0,T;H^1(\Om))} 
	+ \norm{\nu^L}_{L^2(0,T;H^1(\Ga))}
	\le C.
\end{align}
Moreover, proceeding as in Step 4 of the proof of Theorem~\ref{THM:WP:ROB}, we derive the estimate
\begin{align}
\label{UNI:4}
\norm{\delt\phi^L}_{L^2(0,T;H^1(\Om)')}
+ \norm{\delt\psi^L}_{L^2(0,T;H^1(\Ga)')} 
&\le C \left( 1 + \frac{1}{\sqrt{L}} \right).
\end{align}
In the case $\beta>0$, we choose an arbitrary pair of test functions $(\theta,\sigma)\in\DDBI$. We now test \eqref{WF:ROB:1} with $\theta$ and \eqref{WF:ROB:2} with $\sigma$. Adding the resulting equations we observe that a cancellation occurs due to the relation $\theta\vert_\GT = \beta\sigma$ $a.e.$ on $\GT$. Recalling that $(\delt\phi^L,\delt\psi^L) \in (\HH^1)' \subset (\DDBI)'$ $a.e.$ on $[0,T]$, we obtain
\begin{align*} 
\biginn{(\delt\phi^L,\delt\psi^L)}{(\theta,\sigma)}_{\DDBI}
&= \inn{\delt\phi^L}{\theta}_{H^1(\Om)}
	+ \inn{\delt\psi^L}{\sigma}_{H^1(\Ga)} \\
&= -\intO \grad\mu^L\cdot\grad\theta \dx
	- \intG \gradg\nu^L\cdot\gradg\sigma \dS
\end{align*}
$a.e.$ on $[0,T]$. Invoking the uniform bound \eqref{UNI:2}, we conclude that
\begin{align}
\label{UNI:5}
\norm{(\delt \phi^L,\delt \psi^L)}_{L^2(0,T;(\DDBI)')}
\le C 
\quad \text{if}\; \beta>0.
\end{align}

\medskip

\paragraph{Additional uniform bounds on strong solutions.}
Let now $\beta\neq 0$ be arbitrary again and in
addition, we suppose that the conditions \eqref{ass:dom:strong}--\eqref{ass:pen:strong} hold and that $\bph_0\in \HH^2_{\beta, m}$. Then, according to Theorem~\ref{THM:STRONG:ROB}, the quadruplet $(\phi^L,\psi^L,\mu^L,\nu^L)$ is the unique strong solution of the system \eqref{CH:ROB}. We already know from \eqref{BND:PP:L} that
\begin{align}
\label{UNI:6}
\norm{(\phi,\psi)}_{H^1(0,T;\LL^2)} + \norm{(\mu,\nu)}_{L^\infty(0,T;\HH^1)} \le C.
\end{align}
Integrating \eqref{BND:MU:L} and \eqref{BND:NU:L} in time from $0$ to $T$, we can use \eqref{UNI:2} and \eqref{UNI:3} to infer that
\begin{align}
\label{UNI:7}
\norm{\mu}_{L^2(0,T;H^2(\Om))} \le C \left(1+\frac{1}{L}\right)
\quad\text{and}\quad
\norm{\nu}_{L^2(0,T;H^2(\Ga))} \le C \left(1+\frac{1}{\sqrt{L}}\right).
\end{align}

\medskip

We can now use these estimates to investigate the singular limits $L\to 0$ and $L\to \infty$ of the Robin model \eqref{CH:ROB}.

\subsection{The singular limit $L\to 0$ and well-posedness of the Dirichlet model}

\begin{thm}[The limit $L\to 0$ and well-posedness of \eqref{CH:DIR}] \label{THM:Ltozero} $\;$\\
	Let $T,L>0$, $m\in\R$ and $\beta>0$ be arbitrary and suppose that the assumptions \eqref{ass:dom}--\eqref{ass:pen} hold. 
	For any initial datum $(\phi_0,\psi_0)\in \HHBMo$ satisfying $F(\cdot,\phi_0)\in L^1(\Omega)$ and $G(\cdot,\psi_0)$, $B(\cdot,\psi_0)\in L^1(\Gamma)$, let $(\phi^L,\psi^L,\mu^L,\nu^L)$ denote the corresponding unique weak solution to the system \eqref{CH:ROB} in the sense of Definition~\ref{DEF:WP:ROB}. Then the following holds:
	\begin{enumerate}
	\item[\textnormal{(a)}] There exist functions $(\phi_*,\psi_*,\mu_*,\nu_*)$ satisfying 
	\begin{align}
	\label{REG:LIM:0}
	\left\{
	\begin{aligned}
	\phi_* &\in 
	\L\infty {\Lx p}, 
	&\quad \psi_* &\in 
	\L\infty {\LxG q}, \\
	(\phi_*,\psi_*) &\in H^1(0,T;(\DDBI)'),\\
	\mu_* &\in L^2\big(0,T;H^1(\Om)\big), 
	&\quad\, \nu_* &\in L^2\big(0,T;H^1(\Ga)\big),
	\end{aligned}
	\right.
	\end{align}
	and 
	\begin{align}
		\label{MASS:DIR}
		(\phi(t),\psi(t)) \in \HHBMo \quad \text{for almost all $t\in [0,T]$}
	\end{align}
	such that
	\begin{align}
	\label{CONV:0}
	\left\{\;
	\begin{aligned}
		\phl & \to \phi_* &&\quad \hbox{weakly-$^*$ in $\L\infty {\Lx p} $, 
			and $a.e.$ in $\OT$,}
		\\
		\psl & \to \psi_* &&\quad \hbox{weakly-$^*$ in $\L\infty {\LxG q} $,
			and $a.e.$ on $\GT$,}
		\\
		(\phl,\psl) &\to (\phi_*,\psi_*) &&\quad \hbox{weakly in $\H1 {(\DDBI)'}  $,}
		\\
		\ml & \to \mu_* &&\quad \hbox{weakly in $\L2 {\Hx1}$,}
		\\
		\nl & \to \nu_* &&\quad \hbox{weakly in $\L2{\HxG1}$,}
		\\ 
		\beta \nl - \ml &\to 0 &&\quad \hbox{strongly in $L^2(\GT)$},
	\end{aligned}
	\right.
	\end{align}
	as $L\to 0$. This means that $\mu_*=\beta\nu_*$ $a.e.$ on $\Gamma$.
	
	Moreover, the triplet $(\phi_*,\psi_*,\mu_*)$ is the unique weak solution of the system \eqref{CH:DIR} in the sense of Definition~\ref{DEF:WP:DIR}. In particular, this comprises that $(\phi_*,\psi_*) \in C^{0,\frac 1 2}([0,T];(\DDBI)')$.
	
	\item[\textnormal{(b)}] Let us additionally assume that \eqref{ass:dom:strong}--\eqref{ass:pen:strong} hold and that $(\phi_0,\psi_0)\in \HH^2_{\beta,m}$. Then $(\phi^L,\psi^L,\mu^L,\nu^L)$ is a strong solution of the system \eqref{CH:ROB}, and it holds in addition to \eqref{CONV:0} that
	\begin{align}
	\label{CONV:1}
	\left\{\;
	\begin{aligned}
	(\phl,\psl) &\to (\phi_*,\psi_*) &&\quad \hbox{weakly in $\H1 {\LL^2}  $,}
	\\
	(\ml,\nl) & \to (\mu_*,\nu_*) &&\quad \hbox{weakly-$^*$ in $\L\infty {\HH^1}$}
	\end{aligned}
	\right.
	\end{align}
	as $L\to 0$.
	Moreover, it holds that $(\mu_*,\nu_*) \in \L2 {\HH^2}$ and thus, 
	recalling that ${\mu_*} = \beta \nu_*$, the triplet $(\phi_*,\psi_*,\mu_*)$ is the unique strong solution to the system \eqref{CH:DIR}.
	\end{enumerate}
\end{thm}

\begin{proof} \textit{Proof of assertion \textnormal{(a)}.}
	Let $(L_k)_{k\in\N} \subset (0,1]$ denote an arbitrary sequence satisfying $L_k\to 0$ as $k\to\infty$.
	For any $k\in\N$, let $(\phi^k, \psi^k, \mu^k, \theta^k) = (\phi^{L_k}, \psi^{L_k}, \mu^{L_k}, \theta^{L_k})$ denote the unique weak solution to the system \eqref{CH:ROB} corresponding to the parameter $L_k$.
	Due to the uniform bounds \eqref{UNI:1}--\eqref{UNI:3} and \eqref{UNI:5}, the Banach--Alaoglu theorem directly implies the existence of limit functions $(\phi_*,\psi_*,\mu_*,\nu_*)$ satisfying the regularity condition \eqref{REG:LIM:0} such that the convergence properties \eqref{CONV:0} hold with $L$ replaced by $L_k$ as $k\to\infty$ along a non-relabeled subsequence of $(L_k)_{k\in\N}$. This directly implies the relation $\mu_*=\beta\nu_*$ $a.e.$ on $\GT$.
	
	For $(\theta,\sigma)\in \DDBI$ arbitrary, testing \eqref{WF:ROB:1} with $\theta$ and \eqref{WF:ROB:2} with $\sigma$, and adding the resulting equations yields
	\begin{align*}
		\biginn{(\delt\phi,\delt\psi)}{(\theta,\sigma)}_{\DDBI}
		&= \inn{\delt\phi^k}{\theta}_{H^1(\Om)}
		+ \inn{\delt\psi^k}{\sigma}_{H^1(\Ga)} \\
		&= - \intO \grad\mu^k\cdot\grad\theta \dx
		- \intG \gradg\nu^k\cdot\gradg\sigma \dS.
	\end{align*}
	After passing to the limit $k\to\infty$, we obtain
	\begin{align}
	\label{EQ:DIR:1}
		\biginn{(\delt\phi,\delt\psi)}{(\theta,\sigma)}_{\DDBI} 
		= - \intO \grad\mu_*\cdot\grad\theta \dx
		- \intG \gradg\nu_*\cdot\gradg\sigma \dS,
	\end{align}
	which verifies \eqref{WF:DIR:1}. Let now $s,t\in[0,T]$ be arbitrary. Without loss of generality we assume that $s<t$. Integrating \eqref{EQ:DIR:1} in time from $s$ to $t$ gives
	\begin{align*}
	&\biginn{\big(\phi_*(t)-\phi_*(s)\,{,}\, \psi_*(t)-\psi_*(s)\big)}{(\theta,\sigma)}_{\DDBI} 
		\\[0.5ex]
	&\quad = \inn{\phi_*(t)-\phi_*(s)}{\theta}_{H^1(\Om)}
	+ \inn{\psi_*(t)-\psi_*(s)}{\sigma}_{H^1(\Ga)} \\[0.5ex]
	&\quad = - \int_s^t \intO \grad\mu_*\cdot\grad\theta \dx\dr
	- \int_s^t \intG \gradg\nu_*\cdot\gradg\sigma \dS\dr \\
	&\quad \le C \abs{t-s}^{1/2} \norm{(\theta,\sigma)}_{\DDBI} 
		\int_0^T \norm{\grad\mu_*(r)}_{L^2(\Om)}^2 + \norm{\gradg\nu_*(r)}_{L^2(\Ga)}^2 \dr.
	\end{align*}
	This proves that $(\phi_*,\psi_*)\in C^{0,\frac 1 2}([0,T];(\DDBI)')$ and hence, the triplet $(\phi_*,\psi_*,\mu_*)$ satisfies the regularity condition \eqref{REG:DIR}. Proceeding as in Step 6 of the proof of Theorem~\ref{THM:WP:ROB}, we conclude that $(\phi_*,\psi_*,\mu_*)$ also satisfies the weak formulations \eqref{WF:DIR:2} and \eqref{WF:DIR:3}, the mass conservation law \eqref{MASS:DIR} and the energy inequality \eqref{energy:DIR}. This means that $(\phi_*,\psi_*,\mu_*)$ is a weak solution to the system \eqref{CH:DIR}. 
	
	We next show that $(\phi_*,\psi_*,\mu_*)$ is the only weak solution to \ref{THM:WP:ROB}. To this end, we assume that there exists another weak solution $(\phi_{**},\psi_{**},\mu_{**})$ to the system \ref{THM:WP:ROB} and we write
	\begin{align*}
		(\phi,\psi,\mu) := (\phi_*,\psi_*,\mu_*) - (\phi_{**},\psi_{**},\mu_{**})
	\end{align*}
	to denote their difference. Plugging an arbitrary pair of test functions $(\theta,\sigma)\in\DDBI$ into \eqref{EQ:DIR:1} and integrating in time from $0$ to $t$ yields
	\begin{align}
	\label{EQ:UNIQ:1}
	\begin{aligned}
	\int_0^t\biginn{(\phi,\psi)}{(\theta,\sigma)}_{\DDBI} \dt 
	= - \intO \grad \left(\int_0^{t} \mu \ds\right)\cdot\grad\theta \dx
		+ \intG \gradg \left(\int_0^{t} \nu \ds\right) \cdot\gradg\sigma \dS.
	\end{aligned}
	\end{align}
	By the definition of $\SD$ in \eqref{pre:S}, we obtain the relations
	\begin{align*}
		\SD_\Om(\delt\phi) = \mu + c\beta,\quad \SD_\Om(\delt\psi) = \nu + c,\quad \SD_\Om(\phi) = \int_0^{t} \mu \ds + \beta ct, \quad \SD_\Ga(\phi) = \int_0^{t} \nu \ds + ct
	\end{align*}
	for all $t\in[0,T]$ and some constant $c\in\R$. 
	Recall that the assumption $\beta>0$ ensures that $\norm{\cdot}^2_{0,\beta,*}$ actually defines a norm on the space $\DD_{\beta}^{-1}$ (see \eqref{pre:S:0}). 
	A straightforward computation now reveals that, for any arbitrary $t_0 \in [0,T],$
	\begin{align*}
	\frac 12 \norm{(\phi,\psi)(t_0)}^2_{0,\beta,*}
	=  -\intOto \phi \mu \dx\dt  -	\intGto \psi \nu \dS\dt.
	\end{align*}
	As the weak formulations \eqref{WF:DIR:2} and \eqref{WF:DIR:3} (with $\mu\vert_\GT$ replaced by $\beta\nu$) are identical to \eqref{WF:ROB:3} and \eqref{WF:ROB:4}, we can proceed exactly as in the proof of Theorem~\ref{THM:WP:ROB} to conclude that
	\begin{align*}
		\norm{\SD (\phi,\psi)(t)}_{0,\beta} = \norm{(\phi,\psi)(t)}_{0,\beta,*} = 0 \quad \hbox{for almost all $t \in [0,T]$}.
	\end{align*}
	It follows that $\phi=0$ $\,a.e.$ in $\OT$ and $\psi=0$ $\,a.e.$ on $\GT$.
	Now, the identities $\mu=0$ $\,a.e.$ in $\OT$ and $\nu=0$ $\,a.e.$ on $\GT$ follow from \eqref{WF:DIR:2} and \eqref{WF:DIR:3} by a standard comparison argument.
	
	The uniqueness of the limit $(\phi_*,\psi_*,\mu_*)$ finally implies that the convergences established above do not depend on the extraction of the subsequence. Hence, the convergence results hold true for the whole sequence $(L_k)_{k\in\N}$. This means that \eqref{CONV:0} is established.
	
	\textit{Proof of assertion \textnormal{(b)}.} Due to the uniform bound \eqref{UNI:6} and the uniqueness of the limit $(\phi_*,\psi_*,\mu_*)$, the convergence property \eqref{CONV:1} follows directly by means of the Banach--Alaoglu theorem. Proceeding similarly as in the proof of Theorem~\ref{THM:STRONG:ROB}, we can use elliptic regularity theory to conclude a posteriori that $(\mu_*,\nu_*) \in \L2 {\HH^2}$. Consequently, the triplet $(\phi_*,\psi_*,\mu_*)$ is a strong solution of the system \eqref{CH:DIR}. Thus, the proof is complete.	
\end{proof}

\subsection{The singular limit $L\to \infty$ and well-posedness of the decoupled model}
\begin{thm}[The limit $L\to \infty$ and well-posedness of \eqref{CH:DEC:BULK} and \eqref{CH:DEC:SURF}] \label{THM:Ltoinfty}
	Let $T,L>0$, $m\in\R$ and $\beta\neq 0$ be arbitrary and suppose that the conditions {\eqref{ass:dom}--\eqref{ass:pen}} hold. 
	For any initial datum $(\phi_0,\psi_0)\in \HHBMo$
	satisfying $F(\cdot,\phi_0)\in L^1(\Omega)$ and $G(\cdot,\psi_0), B(\cdot,\psi_0)\in L^1(\Gamma)$, let
	$(\phi^L,\psi^L,\mu^L,\nu^L)$ denote the corresponding unique weak solution to the system \eqref{CH:ROB} in the sense of Theorem~\ref{THM:WP:ROB}. 
Then the following holds:
	\begin{enumerate}
	\item[\textnormal{(a)}]	
	There exist functions $(\phi^*,\psi^*,\mu^*,\nu^*)$ satisfying 
	\begin{align}
	\label{REG:LIM:INF}
	\left\{
	\begin{aligned}
	\phi^* &\in 
	\H1 {H^1(\Om)'} \cap \L\infty {\Lx p}, 
	\\
	\psi^* &\in 
	\H1 {H^1(\Ga)'} \cap \L\infty {\LxG q}, \\
	\mu^* & \in \L2 {\Hx1}, \quad 
	\nu^* \in \L2 {\HxG1}
	\end{aligned}
	\right.
	\end{align}
	and
	\begin{align}
	\label{MASS:HNC}
	\meano{\phi^*(t)} = \meano{\phi_0} 
	\quad\text{and}\quad
	\meang{\psi^*(t)} = \meang{\psi_0} 
	\quad\text{for almost all}\; t\in [0,T]
	\end{align}
	such that
	\begin{align}
	\label{CONV:INF}
	\left\{\;
	\begin{aligned}
	\phl & \to \phi^* &&\quad \hbox{weakly-$^*$ in $\L\infty {\Lx p}$,} \\
	&&&\qquad  \hbox{weakly in $\H1 {H^1(\Om)'}$, and $a.e.$ in $\OT$,}
	\\
	\psl & \to \psi^* &&\quad \hbox{weakly-$^*$ in $\L\infty {\LxG q}$,} \\
	&&&\qquad \hbox{weakly in $\H1{H^1(\Ga)'} $, and $a.e.$ on $\GT$,}
	\\
	\ml & \to \mu^* &&\quad \hbox{weakly in $\L2{\Hx1}$,}
	\\
	\nl & \to \nu^* &&\quad \hbox{weakly in $\L2 {\HxG1}$,}
	\\ 
	\tfrac 1L (\beta \nl - \ml) &\to 0 &&\quad \hbox{strongly in $L^2(\GT)$}
	\end{aligned}
	\right.
	\end{align}
	as $L\to\infty$. In addition, it holds that
	\begin{align}
	\label{REG:LIM:INF:2}
		\phi^*\in C^{0,\frac 1 2}([0,T]; {H^1(\Om)'}),
		\quad
		\psi^*\in C^{0,\frac 1 2}([0,T];H^1(\Ga)')
	\end{align}
	and the pair $(\phi^*,\mu^*)$ is the unique weak solution of the system \eqref{CH:DEC:BULK} whereas the pair $(\psi^*,\nu^*)$ is the unique weak solution of the system \eqref{CH:DEC:SURF} in the sense of Definition~\ref{DEF:LTOINFTY}. 
	\item[\textnormal{(b)}]
	Let us additionally assume that \eqref{ass:dom:strong}--\eqref{ass:pen:strong} hold and that $(\phi_0,\psi_0)\in \HH^2_{\beta,m}\,$. Then $(\phi^L,\psi^L,\mu^L,\nu^L)$ is a strong solution of the system \eqref{CH:ROB}, and it holds in addition to \eqref{CONV:INF} that
	\begin{align}
	\label{CONV:2}
	\left\{\;
	\begin{aligned}
	(\phl,\psl) &\to (\phi^*,\psi^*) &&\quad \hbox{weakly in $\H1 {\LL^2}  $,}
	\\
	(\ml,\nl) & \to (\mu^*,\nu^*) &&\quad \hbox{weakly-$^*$ in $\L\infty {\HH^1}\cap \L2 {\HH^2}$,}
	\\ 
	\deln \ml &\to 0  &&\quad \hbox{strongly in $L^2(\GT)$.}
	\end{aligned}
	\right.
	\end{align}
	as $L\to \infty$.
	Moreover, it follows that $\partial_{\bold n} \mu^*=0$
	$a.e.$ on $\GT$
	and thus, the pairs $(\phi^*,\mu^*)$ and $(\psi^*,\nu^*)$ are the unique strong solutions to the systems \eqref{CH:DEC:BULK} and \eqref{CH:DEC:SURF}, respectively.
\end{enumerate}
\end{thm}

\begin{proof}
\textit{Proof of assertion \textnormal{(a)}.}
	Let $(L_k)_{k\in\N} \subset [1,\infty)$ denote an arbitrary sequence satisfying $L_k\to \infty$ as $k\to\infty$.
	For any $k\in\N$, let $(\phi^k, \psi^k, \mu^k, \nu^k) = (\phi^{L_k}, \psi^{L_k}, \mu^{L_k}, \nu^{L_k})$ denote the unique weak solution to the system \eqref{CH:ROB} corresponding to the parameter $L_k$. 
	Since $L_k\ge 1$, the bounds \eqref{UNI:1}--\eqref{UNI:4} can be made uniform in $k$. Hence, we can apply the Banach--Alaoglu theorem to infer the existence of functions $(\phi^*,\psi^*,\mu^*,\nu^*)$ satisfying \eqref{REG:LIM:INF} such that the first four convergence properties in \eqref{CONV:INF} (with $L$ replaced by $L_k$) hold up to subsequence extraction. 
	We further notice that the last convergence of \eqref{CONV:INF} directly follows from \eqref{UNI:2} since, as $k\to\infty$, we have
	\begin{align*}
	\Big\|\frac 1{L_k}(\beta\mu^k-\nu^k)\Big\|_{L^2(\Ga)}^2
	=\frac{1}{{L_k^2}} \norm{\beta\mu^k-\nu^k}_{L^2(\Ga)}^2
		\le \frac C {L_k} \to 0.
	\end{align*}
	The property \eqref{REG:LIM:INF:2} can be established in the same fashion as the corresponding result in Theorem~\ref{THM:Ltozero}.
	As $F'$, $G'$ and $B'$ are continuous in their second argument, we infer that, as $k\to \infty$,
	\begin{align*}
	\begin{aligned}
		&F'(\cdot,\phi^k) \to F'(\cdot,\phi) \quad && a.e.\;\text{in}\; \OT, \\
		&G'(\cdot,\psi^k) \to G'(\cdot,\psi),\; B'(\cdot,\psi^k) \to B'(\cdot,\psi) \quad && a.e.\;\text{on}\; \GT.
	\end{aligned}
	\end{align*}
	Along with Lemma~\ref{LEM:JK}, this is enough to pass to the limit as $k\to\infty$ in 
	the weak formulation \eqref{WF:ROB} written for $(\phi^k, \psi^k, \mu^k, \nu^k)$ from which we conclude that the weak formulations \eqref{WF:NEU:BULK} and \eqref{WF:NEU:SUR} are satisfied.
	This implies that $(\phi^*,\psi^*,\mu^*,\nu^*)$ also satisfies the mass conservation laws \eqref{MASS:HNC}. Moreover, proceeding as in Step 6 of the proof of Theorem \ref{THM:WP:ROB} the energy inequalities \eqref{energy:NEU} can be verified.
	Hence, the pairs $(\phi,\mu)$ and $(\psi,\nu)$ are weak solutions to the systems \eqref{CH:DEC:BULK} and \eqref{CH:DEC:SURF}, respectively, in the sense of Definition~\ref{DEF:LTOINFTY}.
	
	It remains to prove uniqueness of these weak solutions. To this end, we assume that $(\phi^{**},\mu^{**})$ and $(\psi^{**},\nu^{**})$ are also weak solutions of \eqref{CH:DEC:BULK} and \eqref{CH:DEC:SURF}, respectively. We set
	\begin{align*}
		(\phi,\mu):=(\phi^*,\mu^*)-(\phi^{**},\mu^{**}),\quad
		(\psi,\nu):=(\psi^*,\nu^*)-(\psi^{**},\nu^{**}).
	\end{align*}
	Recalling \eqref{pre:N}, we deduce from \eqref{CH:DEC:BULK:1} that
	\begin{align*}
		\NN_\Om(\delt\phi) = \mu - \meano{\mu}
		\quad\text{and}\quad
		\NN_\Om(\phi) = \int_0^t \mu \ds - t \meano{\mu}.
	\end{align*}
	For any arbitrary $t_0\in[0,T]$, a straightforward computation gives
	\begin{align*}
		&\frac 12 \norm{\phi(t_0)}_{\Om,*}^2  = - \intOto \mu\phi \dx\dt \\
		&\quad = \intOto (J\ast\phi) \phi \dx\dt - \intOto \big( F'(\cdot,\phi^{*}) - F'(\cdot,\phi^{**}) \big)\phi \dx\dt.
	\end{align*}
	Proceeding similarly as in Step 7 of the proof of Theorem~\ref{THM:WP:ROB}, we can use the monotonicity of $F'$ and a Gronwall argument to conclude that
	\begin{align*}
		\big\|\grad \NN_\Om\big(\phi(t_0)\big)\big\|_{L^2(\Om)}
		= \norm{\phi(t_0)}_{\Om,*}
		= 0,
	\end{align*}
	which, due to the arbitrariness of $t_0$, directly implies that $\phi = 0$ $a.e.$ in $\OT$. Finally, the identity $\mu=0$ $\,a.e.$ in $\OT$ follows by comparison. This proves the uniqueness of the solution $(\phi^*,\mu^*)$. Moreover, the uniqueness of the solution $(\psi^*,\nu^*)$ can be established in a similar manner.

	In particular, this implies that the limit $(\phi^*,\psi^*,\mu^*,\nu^*)$ is unique and consequently, the convergences established above do not depend on the subsequence extraction. Hence, the convergence properties remain true for the whole sequence $(L_k)_{k\in\N} \subset [1,\infty)$.
	
	\textit{Proof of assertion \textnormal{(b)}.}	
	Arguing as above, it suffices to realize that the estimates
	\eqref{UNI:6}--\eqref{UNI:7}, which can now be established due to the enhanced regularity of the initial data,
	can be made uniform in $k$ as well.
	Moreover, by substituting the identity $L_k\deln\mu^k = \beta\nu^k-\mu^k$ $a.e.$ on $\GT$ into \eqref{UNI:2}, we get
		\begin{align*}
			\norm{\deln\mu^k}_{{L^2}(\Ga)}^2 
			= \frac{1}{{L_k^2}} \norm{{L_k}\deln\mu^k}_{L^2(\Ga)}^2
			= \frac{1}{{L_k^2}} \norm{\beta\mu^k-\nu^k}_{L^2(\Ga)}^2
			\le \frac C {L_k} \to 0
		\end{align*}
	as $k\to\infty$. This directly implies that $\deln\mu^*= 0$ $a.e.$ on $\GT$. 
	Hence, invoking once more the Banach--Alaoglu theorem, we infer the functions $(\phi^*,\psi^*,\mu^*,\nu^*)$ satisfy \eqref{CONV:2}. 
	Then, in light of these stronger convergences and Lemma~\ref{LEM:JK}, it is possible to pass to the limit as $k\to \infty$
	in the strong formulation \eqref{CH:ROB} written for $(\phi^k,\mu^k,\psi^k,\nu^k)$
	to conclude that $(\phi^*,\psi^*,\mu^*,\nu^*)$ is the strong solution. 
	Thus, the proof is complete.
\end{proof}

\pagebreak[2]

\appendix
\section{Appendix}

\begin{proof}[Proof of Lemma~\ref{LEM:JK:UNIQ}] 
	Let $\phi\in L^2(\Om)$ and $\psi\in L^2(\Ga)$ be arbitrary. 
	
	\subparagraph{Proof of \textnormal{(a)}.}
	We define the trivial extension of $\phi$ on $\R^d$ by
	\begin{align*}
	\ov\phi(x) := 
	\begin{cases}
	\phi(x) &\text{if}\; x\in\Omega,\\
	0 &\text{if}\; x\notin\Omega,
	\end{cases}
	\end{align*} 
	where $\phi$ is to be interpreted as an arbitrary but fixed representative of its equivalence class. Let now $\alpha\in\N_0^3$ be an arbitrary multi-index with $\abs{\alpha}\le 1$ and let $\partial^\alpha J$ denote the corresponding derivative. 
	Applying Young's inequality for convolutions (see, e.g., \cite[Thm.~4.2]{lieb-loss}), we obtain
	\begin{align*}
	\norm{\partial^\alpha J\ast\phi}_{L^2(\Om)} 
	\le \norm{\partial^\alpha J\ast\ov\phi}_{L^2(\R^d)} 
	\le C \norm{J}_{W^{1,1}(\R^d)}\, \norm{\ov\phi}_{L^2(\R^d)}
	= C \norm{J}_{W^{1,1}(\R^d)}\, \norm{\phi}_{L^2(\Om)}
	\end{align*} 
	for a constant $C>0$ depending only on $d$. This proves (a). 
	
	\subparagraph{Proof of \textnormal{(b)}.} To prove (b) we proceed as in \cite[Proof of Thm.~4.2]{lieb-loss}. Let $r>1$ be the exponent from \eqref{ass:kernel} and let $\xi\in L^2(\Ga)$ be arbitrary. We now set 
	\begin{align*}
		1<s := \frac{2r'}{r'+1} \le 2,\qquad \text{i.e.}, \;\; \frac 1s + \frac 1s + \frac 1r = 2.
	\end{align*}
	Moreover, recalling the definition of $r'$, we notice that
	\begin{align*}
		r'=\frac {r}{r-1}, \quad s'= \frac {2r'}{r'-1}= 2r.
	\end{align*}	
	Next, we define the functions $f,g,h:\Ga\times\Ga \to \R$ by
	\begin{align*}
	\begin{aligned}
		f(y,z) &:= \psi_+(z)^{s/s'} K(z-y)^{r/s'},\\
		g(y,z) &:= K(z-y)^{r/s'} \xi_+(y)^{s/s'}, \\
		h(y,z) &:= \psi_+(z)^{s/r'} \xi_+(y)^{s/r'}
	\end{aligned}
	\end{align*}
	where $\psi$, $K$ and $\xi$ are to be interpreted as arbitrary but fixed representative of their equivalence class and $\psi_+$ and $\xi_+$ denote the positive parts of these functions. Using the continuous embedding $W^{1,r}(\Om) \emb L^r(\Ga)$ and a change of variables, we obtain
	\begin{align*}
		\norm{K(z-\cdot)}_{L^r(\Ga)} 
			\le \norm{K(z-\cdot)}_{W^{1,r}(\Om)}
			\le \norm{K(z-\cdot)}_{W^{1,r}(\R^d)}
			= \norm{K}_{W^{1,r}(\R^d)}
	\end{align*}
	for almost all $z\in\Ga$.
	Hence, applying H\"older's inequality, we get
	\begin{align*}
		\norm{f}_{L^{s'}(\Ga\times\Ga)} &= \left( \intG \psi_+(z)^s \intG  K(z-y)^r \dS(y)\dS(z) \right)^{\frac{1}{s'}} \\
		& \le \left( \intG \psi_+(z)^s \, \norm{K}_{W^{1,r}(\R^d)}^r \dS(z) \right)^{\frac{1}{s'}}
			\le \norm{\psi}_{L^s(\Ga)}^{s/s'}  \norm{K}_{W^{1,r}(\R^d)}^{r/s'}.
	\end{align*}
	For $g$ and $h$ we derive the analogous estimates
	\begin{align*}
		\norm{g}_{L^{s'}(\Ga\times\Ga)} \le \norm{\xi}_{L^s(\Ga)}^{s/s'}  \norm{K}_{W^{1,r}(\R^d)}^{r/s'},
		\qquad
		\norm{h}_{L^{r'}(\Ga\times\Ga)} \le \norm{\psi}_{L^s(\Ga)}^{s/r'} \norm{\xi}_{L^s(\Ga)}^{s/r'}.
	\end{align*}
	Now, using H\"older's inequality along with the above estimates, we infer that	 
	\begin{align*}
		&\intG\intG f(y,z) g(y,z) h(y,z) \dS(y)\dS(z) =\intG\intG \psi_+(z) K(z-y) \xi_+(y) \dS(y)\dS(z)\\[1ex]
		& \quad \le \norm{\psi}_{L^s(\Ga)} \norm{K}_{W^{1,r}(\R^d)} \norm{\xi}_{L^s(\Ga)}
			\le \norm{\psi}_{L^2(\Ga)} \norm{K}_{W^{1,r}(\R^d)} \norm{\xi}_{L^2(\Ga)}.
	\end{align*}
	Proceeding analogously with the negative parts $\psi_-$ and $\xi_-$ and combining the estimates, we conclude that
	\begin{align*}
		\intG\intG \psi(z) K(z-y) \xi(y) \dS(y)\dS(z) 
		\le C\, \norm{\psi}_{L^2(\Ga)} \norm{K}_{W^{1,r}(\R^d)} \norm{\xi}_{L^2(\Ga)}.
	\end{align*}
	As $\xi\in L^2(\Ga)$ was arbitrary, this implies that the mapping 
	\begin{align*}
		L^2(\Ga) \ni \xi \mapsto \intG (K\cg\psi)(z) \xi(z) \dS(z) \in\R
	\end{align*}
	defines a bounded linear functional on $L^2(\Ga)$. In particular, since $L^2(\Ga) \cong (L^2(\Ga))'$, it holds that
	$K\cg \psi \in L^2(\Gamma)$ with
	\begin{align*}
		\norm{K\cg \psi}_{L^2(\Ga)} \le C\, \norm{K}_{W^{1,r}(\R^d)} \norm{\psi}_{L^2(\Ga)} .
	\end{align*}
	Proceeding analogously with the components of $\gradg K$ instead of $K$, we finally conclude $(\gradg K)\cg \psi \in L^2(\Gamma)$, and
	\begin{align*}
		\norm{K\cg \psi}_{H^1(\Ga)} \le C\, \norm{K}_{W^{2,r}(\R^d)} \norm{\psi}_{L^2(\Ga)} ,
	\end{align*}
	which proves (b).
	
	Thus, the proof is complete.
\end{proof}

\medskip

\begin{proof}[Proof of Lemma~\ref{LEM:JK}]
	We first deduce from Remark~\ref{rem:JK}(b) that there exist functions $\bar J \in W^{1,1}(\Om)$ and $\bar K \in W^{1,1}(\Ga)$ such that
	\begin{align*}
	(J\ast\phi_k) \wto \bar J \;\;\text{in}\; W^{1,1}(\Om)
	\quad\text{and}\quad
	(K\cg\psi_k) \wto \bar K \;\;\text{in}\; W^{1,1}(\Ga)
	\end{align*}
	after extraction of a subsequence. Since $p$ and $q$ satisfy \eqref{ASS:PQ}, the embeddings $W^{1,1}(\Om)\emb L^{p'}(\Om)$ and $W^{1,1}(\Ga)\emb L^{q'}(\Ga)$ are compact. Hence, it holds that
	\begin{align}
	\label{CONV:JK}
	(J\ast\phi_k) \to \bar J \;\;\text{in}\; L^{p'}(\Om)
	\quad\text{and}\quad
	(K\cg\psi_k) \to \bar K \;\;\text{in}\; L^{q'}(\Ga)
	\end{align}
	after another subsequence extraction. For arbitrary test functions $\zeta\in L^{p}(\Om)$ and $\xi\in L^{q}(\Ga)$ we now consider the following linear functionals:
	\begin{alignat*}{2}
	&\mathcal{J_\zeta}:L^p(\Omega)\to \R,\quad &&\phi \mapsto \intO (J\ast\phi) \zeta \dx,\\
	&\mathcal{K_\xi}:L^q(\Gamma)\to \R,\quad &&\psi \mapsto \intG (K\cg\psi) \xi \dS.
	\end{alignat*}
	Using H\"older's inequality, the continuous embeddings $W^{1,1}(\Om)\emb L^{p'}(\Om)$ and $W^{1,1}(\Ga)\emb L^{q'}(\Ga)$ and the estimates from Remark~\ref{rem:JK}, it is straightforward to check that both functionals are continuous.
	Hence, on the one hand, the weak convergence of $(\phi_k)_{k\in\N}$ in $L^p(\Om)$ and the weak convergence of $(\psi_k)_{k\in\N}$ in $L^q(\Ga)$ imply that
	\begin{align}
	\label{CONV:JK:2}
	\mathcal{J_\zeta}(\phi_k) \to \mathcal{J_\zeta}(\phi) = \intO (J\ast\phi) \zeta \dx
	\quad\text{and}\quad
	\mathcal{K_\xi}(\psi_k) \to \mathcal{K_{\xi}}(\psi) = \intG (K\cg\psi) \xi \dS
	\end{align}
	as $k\to\infty$. On the other hand, it follows from \eqref{CONV:JK} that
	\begin{align}
	\label{CONV:JK:3}
	\mathcal{J_\zeta}(\phi_k) \to \intO \bar J \zeta \dx
	\quad\text{and}\quad
	\mathcal{K_\xi}(\psi_k) \to \intG \bar K \xi \dS
	\end{align}
	as $k\to\infty$. Combining \eqref{CONV:JK:2} and \eqref{CONV:JK:3}, invoking the uniqueness of the limit, and recalling that the test functions $\zeta\in L^{p}(\Om)$ and $\xi\in L^{q}(\Ga)$ were arbitrary, we conclude from the fundamental lemma of calculus of variations that $\bar J = J\ast\phi$ $a.e.$ in $\Om$ and $\bar K = K\cg \psi$ $a.e.$ on $\Ga$. This completes the proof.
\end{proof}

\medskip

\begin{proof}[Proof of Lemma~\ref{LEM:NORMS}]
	In the case $L>0$, let $\bph = (\phi, \psi) \in \HH_{\beta,0}^{-1} $ and $\bet=(\zeta,\xi)\in \HH^1$ be arbitrary. Defining 
	\begin{align*}
		\zeta_0 := \zeta - c\beta,\quad \xi_0:= \xi - c, 
		\quad\text{with}\quad
		c := \frac{\beta\abs{\Om}\meano{\zeta}+\abs{\Ga}\meang{\xi}}{\beta^2\abs{\Om}+\abs{\Ga}}
	\end{align*}
	we see that $\bet_0:=(\zeta_0,\xi_0) \in \HHBI$. We notice that $\inn{\bph}{(\beta c,c)}_{\HH^1} = 0$ due to $\bph \in \HH_{\beta,0}^{-1} $ and thus, 
	\begin{align*}
		\norm{\bet_0}_{L,\beta}^2 
		&= \norm{\grad\zeta}_{L^2(\Om)}^2 + \norm{\gradg\xi}_{L^2(\Ga)}^2 + \frac 1 L \norm{\beta\xi-\zeta}_{L^2(\Ga)}^2
		\leq C \left( 1 + \frac{1}{L} \right) \norm{\bet}_{\HH^1}^2.
	\end{align*}
	Recalling that $\SS(\bph)$ satisfies the weak formulation \eqref{WF:S}, we can use the Cauchy--Schwarz inequality to infer that
	\begin{align*}
	&\Big| 	\inn{\bph}{\bet}_{\HH^1} \Big|
	= \Big| \inn{\bph}{\bet_0}_{\HH^1} \Big|
	= \left| \, \big( \SS(\bph), \bet_0 \big)_{L,\beta} \, \right| 
	\leq 
	\norm{\SS(\bph)}_{L,\beta}\, \norm{\bet_0}_{L,\beta}
	\\[0.5ex]
	&\quad
	\leq 
	C \left( 1 + \frac{1}{\sqrt{L}} \right) \norm{\SS(\bph)}_{L,\beta} \norm{\bet}_{\HH^1}.
	\end{align*}
	Hence, by invoking the definition of the operator norm on $\HH_{\beta,0}^{-1}$, we get
	\begin{align*}
	\norm{\bph}_{(\HH^1)'}
	= \sup_{\norm{\bet}_{\HH^1}=1} \big| \inn{\bph}{\bet}_{\HH^1} \big|
	\le C \left( 1 + \tfrac{1}{\sqrt{L}} \right)  \norm{\SS(\bph)}_{L,\beta} 
	\end{align*}
	and since $\norm{\SS(\bph)}_{L,\beta} = \norm{\bph}_{L,\beta,*}$ this proves \eqref{LEM:EST:1}. The estimate \eqref{LEM:EST:2} can be proved completely analogously and thus, the proof is complete.
\end{proof}

\medskip

\begin{proof}[Proof of Lemma~\ref{LEM:KERNEL}]
	As most of the assertions can be verified straightforwardly, we only sketch the most important steps. 
	First of all, it holds in all cases that $J\in C^1(\R^d\setminus\{0\})$ and $K\in C^2(\R^d\setminus\{0\})$ with 
	\begin{align*}
		J(x) \ge (2R)^{-\omega} \underset{\abs{y}\le 2R}{\min} \; \rho(\abs{y}) > 0
		\quad\text{and}\quad
		K(x) \ge (2R)^{-\gamma} \underset{\abs{y}\le 2R}{\min} \; \sigma(\abs{y}) > 0
	\end{align*}
	for all $x\in\ov{B_{2R}(0)}\setminus\{0\}$. Since $x-y\in \ov{B_{2R}(0)}$ for all $x,y\in\ov\Omega$, this implies that \eqref{DEF:AINF} is satisfied. 
	
	For $K(x) = \sigma(\abs{x})\, \abs{x}^{-\gamma}$, we first compute the partial derivatives $\delxj K$ and $\delxi\delxj K$ for $i,j \in \{1,2,3\}$ on $\R^3\setminus\{0\}$. Then, by transformation into spherical coordinates (the radius being denoted by $s$), a straightforward computation leads to the estimates
	\begin{align*}
		\norm{K}_{L^r(\R^d)}^r &\le C \int_0^\infty \sigma(s)^r s^{2-r\gamma} \ds, \\
		\norm{\delxj K}_{L^r(\R^d)}^r &\le C \int_0^\infty \sigma(s)^r s^{2-r(\gamma+1)} + \abs{\sigma'(s)}^r s^{2-r\gamma} \ds, \\
		\norm{\delxi\delxj K}_{L^r(\R^d)}^r &\le C \int_0^\infty \sigma(s)^r s^{2-r(\gamma+2)} + \abs{\sigma'(s)}^r s^{2-r(\gamma+1)} + \abs{\sigma''(s)}^r s^{2-r\gamma} \ds
	\end{align*}
	for all $i,j\in\{1,2,3\}$. Due to the decay conditions on $\sigma$ the above integrals exist if and only if $d-1-r(\gamma+2)>-1$ which is true since $r<3/(\gamma+2)$. This implies that $K\in W^{2,r}(\R^3)$ and it is now easy to check that $K$ satisfies all relevant conditions in \eqref{ass:kernel}. All remaining assertions can be proved by similar arguments.
\end{proof}

\section*{Acknowledgment}
Andrea Signori gratefully acknowledges financial support from the LIA-COPDESC initiative.
Moreover, both the authors were partially supported by the RTG 2339 ``Interfaces, Complex Structures, and Singular Limits''
of the German Science Foundation (DFG). 
The support is gratefully acknowledged.

The authors also want to thank Helmut Abels, Harald Garcke, Kei Fong Lam, Chun Liu and Stefan Metzger for helpful discussion.

\footnotesize
\bibliographystyle{plain}
\bibliography{KS}

\end{document}